\documentclass[11pt]{amsart}

\usepackage[a4paper,inner=3.6cm,outer=3.6cm,top=3.7cm,bottom=3.7cm]{geometry}

\usepackage{textcmds} % amsrefs needs this, but it has to be loaded early to avoid re-defs.
\usepackage{amsmath, amssymb, amsfonts, amstext, amsthm, mathrsfs,mathtools}
\usepackage[all]{xy}
\usepackage[usenames,dvipsnames]{color}
\usepackage{enumitem}
\usepackage{bm}
\usepackage{subfigure}
\usepackage{array}
\usepackage{graphics,graphicx}

\usepackage[normalem]{ulem}

\usepackage{xfrac} 

\setlist[itemize]{leftmargin=*}
\setlist[enumerate]{leftmargin=*}

\let\oldtocsection=\tocsection

\let\oldtocsubsection=\tocsubsection

\renewcommand{\tocsection}[2]{\hspace{0em}\oldtocsection{#1}{#2}}
\renewcommand{\tocsubsection}[2]{\hspace{1em}\oldtocsubsection{#1}{#2}}
\usepackage[bookmarks=true,colorlinks=true,linkcolor=black,citecolor=red,urlcolor=blue,citebordercolor={0 0 1},urlbordercolor={0 0 1},linkbordercolor={0 0 1}]{hyperref} %needs to be loaded after most things
\usepackage{caption}

\newcommand{\m}{\overline{m}}

\newcommand{\C}{\mathbb{C}}

\newcommand{\R}{\mathbb{R}}
\newcommand{\Q}{\mathbb{Q}}
\newcommand{\Z}{\mathbb{Z}}

\newcommand{\M}{\mathcal{M}}
\newcommand{\pt}{\mathrm{pt}}

\newcommand{\ai}{$A_\infty$\ }
\renewcommand{\ss}{\mathfrak{s}}

\newcommand{\Lo}{\Lambda^{\textit{rel}}_0}
\newcommand{\Lop}{\Lambda^{\textit{rel}}_+}

\newcommand{\st}{{^s_t}}

\newcommand{\codim}{\text{codim}}

\newcommand{\om}{\omega}
\newcommand{\al}{\alpha}
\newcommand{\la}{\lambda}
\newcommand{\Om}{\Omega}

\newcommand{\eps}{\epsilon}

\newcommand{\CP}{\mathbb{C}P}
\newcommand{\PxP}{\mathbb{C}P^1 \times \mathbb{C}P^1}

\newcommand{\Pol}{\mathfrak{P}}

\newcommand{\br}{\textbf{r}}
\newcommand{\bv}{\textbf{v}}

\newcommand{\fr}{\mathfrak{r}}
\newcommand{\ff}{\mathfrak{f}}

\renewcommand{\t}{\tilde}

\newcommand{\id}{\mathrm{Id}}

\newcommand{\val}{\mathrm{val}}

\newcommand{\Symp}{\mathrm{Symp}}

\newcommand{\Flux}{\mathrm{Flux}}

\newcommand{\del}{\partial}
\renewcommand{\st}{\star}
\newcommand{\std}{\mathrm{std}}
\renewcommand{\P}{\mathcal{P}}
\newcommand{\Conv}{\mathrm{Conv}}

\renewcommand{\H}{\mathcal{H}}

\newcommand{\bH}{\mathbb{H}}

\newcommand{\bC}{\overline{C}}
\renewcommand{\o}{\circ}

\def\co{\colon\thinspace}

\newcommand{\blk}{\color{Black}} 

\numberwithin{equation}{section}

\newtheorem{thm}{Theorem}[section]
\newtheorem*{thm*}{Theorem}
\newtheorem{prp}[thm]{Proposition}
\newtheorem{prop}[thm]{Proposition}
\newtheorem{lem}[thm]{Lemma}
\newtheorem{lma}[thm]{Lemma}
\newtheorem{cor}[thm]{Corollary}

\newtheorem{qu}[thm]{Question}
\newtheorem{alphth}{Theorem}[section]

\newtheorem{alphprop}[alphth]{Proposition}

\theoremstyle{definition}
\newtheorem{dfn}[thm]{Definition}
\newtheorem{df}[thm]{Definition}

\theoremstyle{remark}
\newtheorem{rmk}{Remark}[section]
\newtheorem{ex}{Example}[section]

\numberwithin{equation}{section}

\begingroup 
\makeatletter 
\@for\theoremstyle:=definition,remark,plain,TheoremNum\do{% 
	\expandafter\g@addto@macro\csname th@\theoremstyle\endcsname{% 
		\addtolength\thm@preskip\parskip 
	}% 
} 
\endgroup 

\title[Geometry of flux]{Geometry of symplectic flux
	\\and Lagrangian torus fibrations}

\author{Egor Shelukhin}
\thanks{}

\author{Dmitry Tonkonog}

\author{Renato Vianna}

\address{(ES) Department of Mathematics and Statistics,
University of Montreal,
C.P. 6128 Succ.  Centre-Ville
Montreal, QC, H3C 3J7, Canada}
\email{egor.shelukhin@umontreal.ca}

\address{(DT)
	University of California, Berkeley}

\address{
(RV)
Universidade Federal do Rio de Janeiro,
Centro de Tecnologia - Bloco C, Cidade Universitária, Av. Athos da Silveira Ramos 149, Ilha do Fundão, Rio de Janeiro RJ, 21941-909, Brazil}
\email{renato@im.ufrj.br}

\begin{document}

\let\thefootnote\relax\footnote{
	ES was supported by an NSERC Discovery Grant, by the Fonds de recherche du Qu\'{e}bec - Nature et technologies, by the Fondation Courtois, and by an Alfred P. Sloan research fellowship.
	DT was partially supported by the Simons
	Foundation grant \#385573, Simons Collaboration on Homological Mirror
	Symmetry. RV was supported by the Herchel Smith postdoctoral fellowship from
	the University of Cambridge and by the NSF under grant 
	No.~DMS-1440140 while the author was in residence at the MSRI during the
	Spring 2018 semester. RV was also supported by Brazil's National Council
of scientific and technological development CNPq, via the research fellowships
405379/2018-8 and 306439/2018-2, by the Serrapilheira Institute grant
Serra-R-1811-25965, by FAPERJ grant E-26/200.230/2023 (282916), and by the Coordenação de Aperfeiçoamento de Pessoal de Nível Superior - Brasil (CAPES) - Finance Code 001.}

 \begin{abstract} 
  Symplectic flux measures the areas of cylinders swept in the process of a
  Lagrangian isotopy. We study flux via a numerical invariant of a Lagrangian
  submanifold that we define using its Fukaya algebra. The main geometric
  feature of the invariant is its concavity over isotopies with linear flux.
  
  We derive constraints on flux, Weinstein neighbourhood embeddings and holomorphic
disk potentials for Gelfand-Cetlin fibres of Fano varieties in terms of
their polytopes. We also describe the space of fibres
of almost toric fibrations on the complex projective plane up to Hamiltonian
isotopy, and provide other applications.

%We show that Calabi-Yau SYZ fibres have unobstructed Floer theory  under a general assumption. 
  
  \end{abstract}

\maketitle

\setcounter{tocdepth}{1}
\tableofcontents

%!TEX root = Shapes.tex

\section{Overview}

\label{sec:Intro}

This paper studies quantitative features of symplectic manifolds, namely the behaviour of \emph{symplectic
flux} and bounds on Weinstein neighbourhoods of Lagrangian submanifolds, using Floer theory. Besides providing
some constructive examples of flux, which in particular allows us to completely describe the flux for the
standard torus in $\C^n$, we provide constraints on (linear) flux, Weinstein neighbourhood embeddings and
holomorphic disk potentials for Gelfand-Cetlin fibres of Fano varieties in terms of their polytopes. We can
also use our invariant to distinguish between (non-monotone) Lagrangians, and in particular we provide a
description of the space of fibres of almost toric fibrations on the complex projective plane up to
Hamiltonian isotopy.

%examples, we will apply our technique to Lagrangian torus fibrations in two important contexts: \emph{SYZ
%fibrations} on Calabi-Yau varieties, and broadly defined \emph{Gelfand-Cetlin fibrations} on Fano varieties. 

Our technique is heavily influenced by the ideas of Fukaya and the Family Floer homology approach to mirror
symmetry. However, it is hard to point at a precise connection because a discussion of the latter theory for
Fano varieties (or in other cases when the mirror should support a non-trivial Landau-Ginzburg potential) has
not appeared in the literature yet. Intuitively, the numerical invariant $\Psi$ introduced in this paper
measures the minimal area of holomorphic disks with boundary on the given Lagrangian. Alternatively, and with
respect to a Lagrangian torus fibration, $\Psi$ should be thought of as the tropicalisation of the
Landau-Ginzburg potential defined on the rigid analytic mirror to the given variety.

\subsection{Flux and shape}
We begin by reviewing the classical symplectic invariants of interest.
Let $(X,\omega)$ be a symplectic manifold and $\{L_t\}_{t\in[0,1]}$ a Lagrangian isotopy, i.e.~a family of Lagrangian submanifolds $L_t\subset X$ which vary smoothly with $t\in[0,1]$. Denote $L=L_0$. The \emph{flux of $L_t$}, $$\Flux(\{L_t\}_{t\in [0,1]})\in H^1(L;\R),$$
is defined in the following way. Fix an element $a\in H_1(L;\R)$, realise it by a real 1-cycle $\alpha_0\subset L_0$, and consider its trace under the isotopy, that is, a 2-chain $C_a$ swept by $\alpha_0$ in the process of isotopy. The 2-chain $C_a$ has boundary on $L_0\cup L_1$. One defines
$$
\Flux(\{L_t\})\cdot a=\textstyle\int_{C_\alpha}\omega\in \R.
$$   	
Above, the dot means  Poincar\'e pairing. It is easy to see that $\Flux(\{L_t\})$  depends only on $a\in H_1(L;\R)$, and is linear in $a$. Therefore it can be considered as an element of $H^1(L;\R)$.

Let $L\subset (X,\omega)$ be a Lagrangian submanifold. The \emph{shape} of $X$
relative to $L$ is the set of all possible fluxes of Lagrangian isotopies
beginning from $L$: 

$$ Sh_{L}(X) = \{\Flux(\{L_t\}_{t\in[0,1]}): L_t \subset X
\text{ a Lag.~isotopy, } L_0=L\}\subset H^1(L;\R). $$ 

At a first sight this is a very natural invariant of $L$, but we found out that
it frequently behaves wildly for compact symplectic manifolds. For example, the
shape of $\CP^2$ relative to the standard monotone Clifford torus is unbounded;
in fact, that torus has an unbounded product neighbourhood which symplectically
embeds into $\CP^2$, viewed in the almost toric fibration 
shown in Figure~\ref{fig:Sausage}. See Section~\ref{sec:Ex}
for details.

\begin{figure}[h!]   
	\begin{center}
		\includegraphics[bb=200 0 100 100]{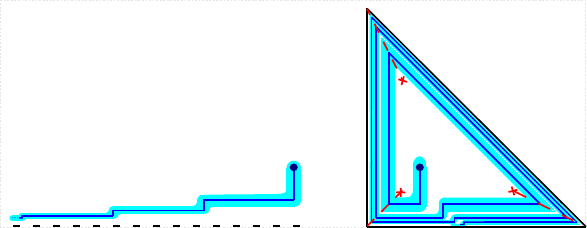}
		\caption{On the right, the ``wild'' unbounded non-convex product neighbourhood $L\times Q \subset T^*L = L \times \R^2$ of the Clifford torus $L$ in $\CP^2$. 
		The domain $Q\subset\R^2$ is shown on the left and is viewd also as a subset of $Sh_{L}(X)$.}
		\label{fig:Sausage} 
	\end{center} 
\end{figure}

To remedy this  and obtain a better behaved invariant of Lagrangian submanifolds using flux, it is natural to introduce the following notion.
We call a Lagrangian isotopy $\{L_t\}_{t\in[0;1]}$ a
 \emph{star-isotopy} if
 $$
 \Flux(\{L_t\}_{t\in[0;t_0]})=t_0\cdot \Flux(\{L_t\}_{t\in[0;1]}),\quad \text{for each }  0\le t_0\le 1.
 $$
 In other words,  flux must develop linearly in time along a fixed ray in $H^1(L_0;\R)$.
The \emph{star-shape} of $X$ relative to $L$ is defined to be
$$ Sh^\st_{L}(X) = \{\Flux(\{L_t\}_{t\in[0,1]}):
L_t \subset X \text{ a Lag.~ star-isotopy, } L_0=L\}\subset H^1(L;\R). $$
We shall soon see that this invariant captures the geometry of $X$ in a more robust way. One reason is that   star-shape is  invariant under Hamiltonian isotopies, while shape is invariant under all Lagrangian isotopies.

A historical note is due.
Symplectic shape was introduced by Sikorav \cite{Si89}, cf.~\cite{Eli91}, in the context of exact symplectic manifolds.
The paper \cite{EGM16} studied an  invariant $\mathit{def}_L\co H^1(L, \R) \to (0 , +\infty]$ which is equivalent to star-shape. We refer to that paper for further context surrounding flux in symplectic topology.
	
\begin{ex}
	Suppose $D\subset \R^n$ is an open domain,  $X = T^n \times D \subset T^*T^n \cong T^n
	\times \R^n$ is a \emph{product neighbourhood} of the $n$-torus with the standard symplectic form, and $0\in D$.
	Let $L=T^n\times\{0\}$ be the 0-section in this neighbourhood. The Benci-Sikorav theorem \cite{Si89} says that $Sh_L(X)=D$.
	If $D$ is
	star-shaped with respect to the origin, then $Sh^\st_{L}(X) = D$, see 
\cite[Theorem~1.3]{EGM16}.
\end{ex}

\begin{rmk}
Suppose $L \subset X$ is a Lagrangian torus. Then 
$Sh_{L}(X)$ gives an obvious bound on product Weinstein neighbourhoods of $L$ embeddable into $X$. Namely, if there is a symplectic embedding of
 $T^n\times D$ into $X$ taking the 0-section to $L$, then $D \subset Sh_{L}(X)$. If 
$D$ is star-shaped with respect to the origin, then also $D \subset Sh^{\st}_{L}(X)$.  
\end{rmk}

\subsection{The invariant $\Psi$ and its concavity}
We are going to study the geometry of flux, including star-shapes, with the help of a numerical invariant that associates a number $\Psi(L)\in(0,+\infty]$ 
(possibly $+\infty$) to any orientable and spin Lagrangian submanifold $L\subset X$. Fix a compatible almost complex structure $J$; 
the definition of $\Psi(L)$ will be given in terms of the Fukaya \ai algebra of $L$.

Roughly speaking, $\Psi(L)$ is the lowest symplectic area $\omega(\beta)$ of a class $\beta\in H_2(X,L;\Z)$ such that holomorphic disks in class $\beta$ exist and, moreover, contribute non-trivially to some symmetrised \ai structure map on odd degree elements of $H_*(L)$. 
The latter means, again roughly, that there exists a number $k\ge 0$ and a collection of cycles $c_1,\ldots,c_k\in H_{odd}(L;\R)$ such that holomorphic disks in class $\beta$ whose boundaries are incident $c_1,\ldots,c_k$ form a 0-dimensional moduli space, thus posing an enumerative problem. The count for this problem should be non-zero.

The  definition of $\Psi(L)$ appears in Section~\ref{sec:Psi}, and the background on \ai algebras is revised in Section~\ref{sec:Fuk}.
Quite differently from the above sketch, we take the primary definition to be the following:
$$
\Psi(L)=\inf\{\val\, m(e^b):b\in H^{odd}(L;\Lop)\}.
$$
Here is a quick outline of the notation:
$\Lop$ is the maximal ideal in the Novikov ring $\Lo$;
$\val\co H^{*}(L;\Lop)\to \R_{>0}$ is the valuation;
$$m(e^b)=m_0(1)+m_1(b)+m_2(b,b)+\ldots$$ is the expression called the \emph{Maurer-Cartan prepotential} of $b$; and the $m_i$ are the  structure maps of the (curved) Fukaya \ai algebra of $L$. 

An important technical detail, reflected in the formula for $\Psi(L)$, is that we define $\Psi(L)$ using a \emph{classically minimal} model of the Fukaya algebra of $L$, i.e.~one over the singular cohomology vector space $H^*(L;\Lo)$. Such models always exist, by a version of homological perturbation lemma.

Using the fact that the Fukaya algebra does not depend on the choice of $J$ and Hamiltonian isotopies of $L$ up to weak homotopy equivalence, we show in Section~\ref{sec:Psi} that $\Psi(L)$ is well-defined and invariant under Hamiltonian isotopies of $L$. 

The above definition is convenient for proving the invariance of $\Psi(L)$, but not quite so for computations and for understanding its geometric properties.
To this end, we give a more explicit formula which was hinted above, see Theorem~\ref{theorem: explicit}:

\[\Psi(L) =  \min \left\{ \omega(\beta): \exists\,  c_1,\ldots,c_k \in H^{odd}(L;\R)\textit{ s.t. } \sum_{\sigma \in S_k} m_{k,\beta} (c_{\sigma(1)},\ldots,c_{\sigma(k)}) \neq 0\right\}.\]
Here $m_{k,\beta}$ is the \ai operation coming from disks in class $\beta\in H_2(X,L;\Z)$. Below is the main result linking $\Psi$ to the geometry of flux.

\begin{alphth}[=Theorem~\ref{th:psi_concave}]
	Let $\{L_t\}_{t\in [0,1]}$ be a Lagrangian star-isotopy. Then the function $\Psi(L_t)\co [0,1]\to (0,+\infty]$ is continuous and concave in $t$.
\end{alphth}

\smallskip 
{\it Proof idea.}
The idea lies in Fukaya's trick, explained in Section~\ref{sec:stretch}. It says that there exist compatible almost complex structures $J_t$ such that the structure maps $m^t_{k,\beta^t}$ for the Lagrangians  $L_t$ are locally constant in a neighbourhood of a chosen moment of the isotopy. But the areas of classes $\beta^t$ change linearly in $t$ during a star-isotopy. So $\Psi(L_t)$ is locally computed as the \emph{minimum of several linear functions}; hence it is concave.
\smallskip

Now suppose that $X$ admits a singular Lagrangian torus fibration $X\to B$ over a base $B$; it is immaterial how complicated the singularities are, or what their nature is.  By the Arnold-Liouville theorem, the locus $B^\o\subset B$ supporting regular fibres carries a natural integral affine structure.
Consider the map $\Psi\co B^\o\to(0,+\infty]$ defined by $p\mapsto \Psi(L_p)$, where $L_p\subset X$ is the smooth Lagrangian torus fibre over $p\in B^\o$. The previous theorem implies that this function is concave on all affine line segments in $B^\o$. This is a strong property that allows to compute $\Psi$ for wide classes of fibrations on Calabi-Yau and Fano varieties, with interesting consequences. For instance, it suggests a possible approach to proving that the fibres are unobstructed in the Calabi-Yau case, which will be investigated in future work.

\subsection{Fano varieties}
 Fano varieties are discussed in Section~\ref{sec:Ex}.
We introduce a class of singular Lagrangian torus fibrations called \emph{Gelfand-Cetlin fibrations} (Section~\ref{sec:Ex}). Roughly speaking, they are \emph{continuous} maps $X\to \Pol$ onto a convex lattice polytope $\Pol\subset \R^n$ which look like usual smooth toric fibrations away from the union of codimension two faces of $\Pol$. This includes actual toric fibrations and  classical Gelfand-Cetlin systems on flag varieties (from which we derived the name). 
It is not unreasonable to conjecture that all Fano varieties admit a Gelfand-Cetlin fibration.

\begin{alphth}[=Theorem~\ref{thm:GCF2}] \label{th:GCF}
	Let $X$ be a Fano variety, $\mu: X\to \Pol\subset\R^n$ a Gelfand-Cetlin fibration, and  $L\subset X$ its monotone Lagrangian fibre. 
	
	Let $\P_L^\vee\subset H^1(L;\R)$ be the interior of the dual of the Newton polytope
	associated with the Landau-Ginzburg potential of $L$ (Section~\ref{subsec:LG_Pot}). Let $c$ be the monotonicity constant of $X$, and assume $\Pol$ is translated so that the origin corresponds to the fibre $L$. Then the following three subsets of $H^1(L;\R)\cong \R^n$ coincide:
	$$2c\cdot  \P_L^\vee = \Pol^0=Sh^\st_L(X).$$
\end{alphth}

Note that there are obvious star-isotopies given by moving $L$ within the fibres
of the fibration, achieving any flux within $\Pol^0$. The equality $\Pol^0
=Sh^\st_L(X)$ says that there are no other star-isotopies of $L$ in $X$ (which
are not necessarily fibrewise) achieving different flux than that.

The enumerative geometry part of the theorem, about the Newton polytope of the
Landau-Ginzburg potential, is interesting in view of the program for classifying
Fano varieties via maximally mutable Laurent polynomials, or via corresponding
Newton polytopes which are supposed to have certain very special combinatorial
properties \cite{CCGGK14,CCGK16,CKP17}. Toric Fano varieties correspond via this
bijection to their toric polytopes. One wonders about the symplectic meaning of
a polytope corresponding this way to a non-toric Fano variety $X$. The answer
suggested by the above theorem is that it should be the polytope of some
Gelfand-Cetlin fibration on $X$; the equality $2c\cdot \P_L^\vee = \Pol^0$
supports this expectation.

On the way, we compute the $\Psi$-invariant of all fibres of a Gelfand-Cetlin system.

\begin{alphprop} \label{prop:Psi_GCF}
Consider a Gelfand-Cetlin fibration over a polytope, with point $p$ corresponding to a monotone fibre $L$.
 Consider the function over the interior of the polytope whose value at a point is given by the value of $\Psi$ of the corresponding fibre.
 Then this is a PL function whose graph 
is a cone over the polytope (see Figure~\ref{fig:psi_graph}).
The vertex of the cone is located over $p$,
and its height equals the area of the Maslov 
index 2 disks with boundary on $L$.
\end{alphprop}

\begin{figure}[h]
	\includegraphics[]{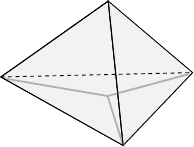}
	\caption{The graph of $\Psi$ over the moment triangle for $\CP^2$. The function $\Psi$ is non-linear over the three gray segments.}
	\label{fig:psi_graph}
\end{figure}

\subsection{Other results} In Section~\ref{sec:Ex} we determine shapes and star-shapes of certain tori in $\C^n$, and study the wild behaviour of (non-star) shape in $\CP^2$. In Section~\ref{sec:LagModuli_CP2} we determine the non-Hausdorff moduli space of all (not necessarily monotone) Lagrangian tori in $\CP^2$ arising as fibres of almost toric fibrations, modulo Hamiltonian isotopy. In Section~\ref{sec:stretch}, as a warm-up, we discuss bounds on flux within a convex neighbourhood $L\subset U\subset X$; this gives us an opportunity to recall the Fukaya trick and establish a non-bubbling lemma that is useful for
the arguments in Section~\ref{sec:Ex}. We also briefly discuss dynamical applications along the lines of \cite{EGM16} in Section \ref{sec: stab}.

\subsection{Technical remark} \label{subsec:IntroRmk}
Our main invariant, $\Psi(L)$, is defined using the Fukaya algebra of $L$.
We remind that
whenever the symplectic form on $X$ has rational cohomology class, 
the
Fukaya algebra of $L$ is defined via classical transversality methods
using the
technique of stabilising divisors \cite{CM07,CW15,CW17}.

In general, the
definition of the Fukaya algebra requires the choice of a virtual perturbation scheme. 
Our results are not sensitive to the details of how it is implemented. They rely on the general algebraic properties of Fukaya
algebras reminded in Section~\ref{sec:Fuk}. We shall use \cite{FO3Book} as the common
reference for these basic properties; in the setting with
stabilising divisors, they were established in \cite{CW15,CW17}.

\subsection*{Acknowledgements} We thank Denis Auroux for many valuable
conversations and Michael Entov for useful communications regarding dynamical applications.

This work was initiated during the ``Symplectic topology, sheaves
and mirror symmetry'' summer school at Institut de Math\'{e}matiques de Jussieu,
2016. We acknowledge the hospitality of the Institute of Advanced Study,
Princeton, and IBS Center for Geometry and Physics, Pohang, where part of the
work was carried out.
 
ES was partially supported by NSF grant No.~DMS-1128155 at the IAS, and by an NSERC Discovery Grant, and by the Fonds de recherche du Qu\'{e}bec - Nature et technologies, by the Fondation Courtois, and by an Alfred P. Sloan research fellowship, at the University of Montr\'{e}al.

DT was partially supported by the Simons
Foundation grant \#385573, Simons Collaboration on Homological Mirror
Symmetry, and carried out  initial stages of the work at Uppsala University, supported by the Geometry and Physics project grant from the Knut and Alice Wallenberg Foundation.

RV was supported by the Herchel Smith postdoctoral fellowship from
	the University of Cambridge and by the NSF under grant 
	No.~DMS-1440140 while the author was in residence at the MSRI during the
	Spring 2018 semester. RV was also supported by Brazil's National Council
of scientific and technological development CNPq, via the research fellowships
405379/2018-8 and 306439/2018-2, by the Serrapilheira Institute grant
Serra-R-1811-25965, and by FAPERJ grant E-26/200.230/2023 (282916).

%!TEX root = Shapes.tex

\section{Enumerative geometry in a convex neighbourhood}

\label{sec:stretch}

This section is mainly a warm-up.
Suppose $L\subset X$ is a \emph{monotone} Lagrangian submanifold. Using standard Symplectic Field Theory stretching techniques and without using Fukaya-categorical invariants, we are going to obtain bounds on the shape of Lioville neighbourhoods $L\subset U\subset X$ of $L$ that are symplectically embeddable into $X$. 
Along the way we recall Fukaya's trick and establish a useful no-bubbling result, Lemma~\ref{lem:neck_stretch}.

Let $J$ be a tame almost complex
structure. For a class $\beta \in H_2(X,L;\Z)$ of Maslov index~2, let $\M_\beta(J)$
be the 0-dimensional moduli space of unparametrised $J$-holomorphic disks $(D,\del
D)\subset(X,L)$ with boundary on $L$, whose boundary passes through a specified
point $\pt\in L$, and whose relative homology class equals $\beta$.
We will be assuming that the above disks are regular, whenever $\M_\beta(J)$ is
computed. Their count $\# \M_\beta(J)\in \Z$ is invariant under choices of $J$ and Hamiltonian isotopies of $L$, by the monotonicity assumption.

\begin{dfn}
	\label{dfn:Liouv_nbhd}
	We call an open subset $U\subset (T^*M,\omega_\std)$ a {\it Liouville neighbourhood (of the zero-section)} if $U$ contains the zero-section, and there exists a Liouville 1-form $\theta$ on $U$ such that $d\theta=\omega_\std$, and the zero-section is $\theta$-exact.
\end{dfn}

The next theorem establishes a shape bound on a  Liouville
neighbourhood $U$ admitting a symplectic embedding $\phi\co U\to X$ which takes
the zero-section to $L$.

\begin{thm}
	\label{th:nbhood_comp}
	Let $L\subset X$ be a monotone Lagrangian submanifold and $J$ a tame almost
	complex structure. Let $U\subset X$ 
	be an open subdomain containing $L$ and symplectomorphic to a Liouville neighbourhood 
	of $L\subset U \hookrightarrow T^*L$. 
	For a Maslov index~2 class $\beta$, if
	
	 $\# \M_\beta(J) \ne 0$, then the shape $Sh_L(U)$ belongs to the following affine half-space:
	$$
	Sh_L(U)\subset B_\beta = 
	\{\ff \in H^1(X;\R): 2c + \ff \cdot \del \beta >0\},
	$$ 
	where $\cdot$ is the Poincar\'e pairing and $c$ is the monotonicity constant, i.e.~$2c = \omega(\beta)$.
\end{thm}

\subsection{Fukaya's trick}
\label{subsec:gromov_comp}
Fukaya's trick is a  useful observation which has been used as an ingredient to set up Family Floer homology \cite{Fuk10,Ab14}.
This trick will enable us to apply Gromov and SFT compactness theorems to holomorphic curves with boundary on a moving Lagrangian submanifold, when this isotopy is not Hamiltonian.
Let $L_t\subset X$ be a Lagrangian isotopy, $t\in [0,1]$. Choose a family of diffeomophisms
\begin{equation} 
\label{eq:f_t}
f_t\co X\to X,\quad f_t(L_0)=L_t.
\end{equation}
Denote $\omega_t=f_t^*\omega$. Let $ J_t'$ be a generic family of almost complex structures such that $J_t'$ tames $\omega_t$. When counting holomorphic disks (or other holomorphic curves) with boundary on $L_t$, we will do so using almost complex structures of the form 
\begin{equation}
\label{eqn:J_t_Lag_iso}
J_t=(f_t)_*J_t'
\end{equation}
 where $J_t'$ is as above. The idea is that $f_t^{-1}$ takes $J_t$-holomorphic curves with boundary on $L_t$ to $J_t'$-holomorphic curves with boundary on $L_0$. In this reformulation, the Lagrangian boundary condition $L_0$ becomes constant, which brings us to the standard setup for various aspects of holomorphic curve analysis, such as compactness theorems.
  
Although  there may not exist a single symplectic form taming all $J'_t$, for each $t_0$ there exists a $\delta>0$
such that for all $t\in[t_0-\delta;t_0+\delta]$, $J_t'$ tames $\omega_{t_0}$. For the purposes of holomorphic curve analysis, this property is as good as being tamed by a single symplectic form.
Here is a summary of our notation, where the right column and the left column differ by applying $f_t$:
$$
\begin{array}{cccc}
\textit{Almost complex~str.}& J_t'&\xrightarrow{f_t}  & J_t\\
\textit{Tamed by}&\omega_t& &\omega\\
\textit{Lag.~boundary cond.} & L_0&  &L_t\\
\end{array}
$$
This should be compared with $L_t$ actually being a Hamiltonian isotopy; in this case we could have taken $\omega_t\equiv \omega$ and $J_t'$ tamed by the fixed symplectic form $\omega$; this case is standard in the literature.

\subsection{Neck-stretching}
Recall the  setup of Theorem~\ref{th:nbhood_comp}:
$L\subset X$ is a monotone Lagrangian submanifold, and $L\subset U\subset X$ where $U$ is symplectomorphic to a Liouville neighbourhood in the sense of Definition~\ref{dfn:Liouv_nbhd}, which identifies $L$ with the zero-section.

\begin{lem}
\label{lem:neck_stretch}
For each $E>0$ and any family $\{L_b\}_{b\in B}$ of Lagrangian submanifolds
$L_b\subset U\subset X$ parametrised by a compact set $B$ which are Lagrangian isotopic to $L$,
 there exists an almost complex structure $J$ on $X$ such that each $L_b$ bounds no $J$-holomorphic disks of Maslov index $\le 0$ and area $\le E$ in $X$.
\end{lem}

\begin{proof}
We will show that almost complex structures which are sufficiently neck-stretched around $\del U$ have the desired property.
Pick a tame $J$ on $X$;
neck-stretching around $\del U$ produces a family of tame almost complex structures $J_n$, $n\to+\infty$, see e.g.~\cite{EGH00,CompSFT03}.
We claim that the statement of Lemma~\ref{lem:neck_stretch} holds with respect $J_{n}$ for a sufficiently large  $n$. 

Suppose, on the contrary, that  $L_{b_n}$ bounds a $J_{n}$-holomorphic disk of Maslov index $\mu_n\le 0$ and area $\le E$ all sufficiently large $n$. Passing to a subsequence if necessary, we can assume that $b_n\to b\in B$.
Apply the SFT compactness theorem, which is a version \cite[Theorem~10.6]{CompSFT03} for curves with Lagrangian boundary condition $L_{b_n}$. Using Fukaya's trick, one easily reduces the desired compactness statement to one about the fixed Lagrangian submanifold $L_b$.

 The outcome of SFT compactness is a broken holomorphic building; see Figure~\ref{fig:stretch} for an example of how the building may look like, ignoring $C$ for the moment.
We refer to \cite{CompSFT03} for the notion of holomorphic buildings and only  record the following basic properties:
\begin{itemize}
	\item all curves in the holomorphic building have positive $\omega$-area;
	\item  topologically, the curves glue to a disk with boundary on $L_{b}$ and $\mu\le 0$;
	\item there exists a single curve in the building, denoted by $S$,  which has boundary on $L_{b}$; this curve lies inside $U$ and may have several punctures asymptotic to  Reeb orbits $\{\gamma_i\}\subset \del U$.
\end{itemize}
The Reeb orbits mentioned above are considered with respect to the
contact form $\theta|_{\del U}$, where $\theta$ is the Liouville form on $U$
provided by Definition~\ref{dfn:Liouv_nbhd}; this means that $d\theta=\omega$
and $L$ is $\theta$-exact. (Recall that the choice of $\theta$ near $\del
U$ is built into the neck-stretching construction, and we assume that we have
used this particular $\theta$ for it.)

\begin{figure}[h]
\includegraphics[]{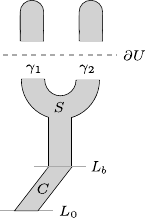}
\caption{A broken disk with boundary on $L_{b}$, consisting of $S$ and the upper disks. The attaching cylinder $C$ is not part of the broken disk.}
\label{fig:stretch}
\end{figure}

Next, consider a topological cylinder $C\subset U$ with boundary $\del C\subset L_{b}\cup L_0$,
and such that the $L_{b}$-component of the boundary of $C$ matches the one of $S$, see Figure~\ref{fig:stretch}.
Note that $S\cup C$, considered as a 2-chain in $U$, has boundary of the form:
$$
\del(S\cup C)=\textstyle \sum_i\gamma_i-l
$$
where $l\subset L_0$ is a 1-cycle.
Let us compute the area:
$$
 \omega(S\cup C)=\textstyle\sum_i \textstyle\int_{\gamma_i}\theta-\textstyle\int_l\theta=
 \textstyle\sum_i \textstyle\int_{\gamma_i}\theta>0.
$$
Indeed, $\int_l \theta=0$ because $L$ is $\theta$-exact, and $\int_{\gamma_i}\theta>0$
for any Reeb orbit $\gamma_i$. 

Let us now construct a topological disk $(D',\del D')\subset (X,L)$ by the
following procedure: first, glue together all pieces of the $\omega$-tamed
holomorphic building constructed above, including $S$, then additionally glue
$C$ on to the result. Clearly, we get a topological disk with boundary on $L_0$,
and moreover $\mu(D')  \le 0$. Finally, $\omega(D')>0$ because $\omega(C\cup
S)>0$ and all other curves in the building also have positive $\omega$-area.
These two properties of $D'$ contradict the fact that $L$ is monotone inside
$X$. \end{proof}

\subsection{Conclusion of proof}

%
%\begin{proof}[Proof of Theorem~\ref{th:nbhood_comp}]
%
%Assume $L_t\subset U$ is a Lagrangian isotopy realising  flux $\ff\in
%H^1(L;\R)$. Take $J_t$ from Lemma~\ref{lem:neck_stretch}, and let $\M_\beta(J_t)$ be
%the moduli space of Maslov index~2 holomorphic disks with boundary on $L_t$ and passing through a specified point on it. As
%explained in Subsection~\ref{subsec:gromov_comp}, these holomorphic disks can be
%understood as $J_t'$-holomorphic disks on the fixed Lagrangian submanifold
%$L_0=L$. So Gromov compactness applies to show that the count
%$\# \M_\beta(J_t)$ is independent of $t$ unless a bubbling occurs. However, any
%such bubbling will produce a $J_t$-holomorphic disk of Maslov index $\le 0$ with
%boundary on $L_t$, which is impossible by Lemma~\ref{lem:neck_stretch}. We
%conclude that $\# \M_\beta(J_i')\neq 0$, so for $i=0,1$, $L_i$ bounds a
%$J_i$-holomorphic disk $D_i$. As 2-chains, these disks differ by a cylinder
%swept by a cycle in class $-\del \beta$. By the definition
%of flux: $$ \omega(D_0)-\omega(D_1)=-\ff\cdot \del\beta. $$ Finally, we have
%$\omega(D_0)=2c$ by monotonicity, and $\omega(D_1)>0$ because $D_1$ is
%holomorphic; therefore $\ff\cdot \del\beta\ge - 2c$. \end{proof}

%{\color{green} Rewrote this part to include regularity of the acs family. Please take a look.}

\begin{proof}[Proof of Theorem~\ref{th:nbhood_comp}] Take $J$ from
Lemma~\ref{lem:neck_stretch} for energy level $E > \omega(\beta).$ Pick a point
$p_t \in L_t.$ By an application of Gromov compactness, and standard
transversality techniques, there exist almost complex structures $J_t$ for $0
\leq t \leq 1$, sufficiently close $J,$ such that $J_t$ still admit no
holomorphic disks of Maslov index $\leq 0$ and energy $\leq E$ with boundary on
$L_t,$ for all $t,$ and that the moduli spaces $\M_\beta(J_i),$ $i=0,1$ of
Maslov index~2 holomorphic disks with boundary on $L_i$ and $p_i$, are regular,
as well as the parametric moduli space $\sqcup_{0 \leq t \leq 1} \M_\beta(J_t)$
where now an element of $\M_\beta(J_t)$ passes through $p_t.$ As explained in
Subsection~\ref{subsec:gromov_comp}, these holomorphic disks can be understood
as $J_t'$-holomorphic disks on the fixed Lagrangian submanifold $L_0=L$ (and the
curve $p_t$ could be chosen to correspond to a fixed $p = p_0 \in L$). So Gromov
compactness, again, applies to show that the count $\# \M_\beta(J_i)$ is
independent of $i$ unless a bubbling occurs for some $t$. However, any such
bubbling will produce a $J_t$-holomorphic disk of Maslov index $\le 0$ with
boundary on $L_t$, which is impossible by construction. We conclude that $\#
\M_\beta(J_i)\neq 0$, so for $i=0,1$, $L_i$ bounds a $J_i$-holomorphic disk
$D_i$. As 2-chains, these disks differ by a cylinder swept by a cycle in class
$-\del \beta$. By the definition of flux: $$ \omega(D_0)-\omega(D_1)=-\ff\cdot
\del\beta. $$ Finally, we have $\omega(D_0)=2c$ by monotonicity, and
$\omega(D_1)>0$ because $D_1$ is holomorphic; therefore $\ff\cdot \del\beta\ge -
2c$. \end{proof}

\section{Fukaya algebra basics}

\label{sec:Fuk}

\subsection{Fukaya algebras}
Fix a ground field $\mathbb{K}$ of characteristic zero. We take $\mathbb{K}=\R$ throughout, although all arguments are not specific to this.
Let $X$ be a symplectic manifold and $L\subset X$ a Lagrangian submanifold. We shall use the following version of the Novikov ring with formal parameters $T$ and $q$:

\begin{multline*}
\Lo = \Biggr\{ 
\sum_{i=0}^{\infty} a_iT^{\omega(\beta_i)} q^{\beta_i} 
\ \Bigg|  \ a_i \in \R,  \ \beta_i \in 
H_2(X,L;\Z), \  \omega(\beta_i) \in \R_{\ge 0},
\\ 
 \lim_{i \to \infty}
 \omega(\beta_i) = +\infty, \
 \omega(\beta_i) = 0 \iff \beta_i = 0 
 \Biggr\}   
\end{multline*}

We also use the ideal:
\begin{multline*}\Lop = \Biggr\{ \sum_{i=0}^{\infty} a_iT^{\omega(\beta_i)} q^{\beta_i} 
\ \Bigg|  \ a_i \in \R,  \ \beta_i \in 
%\relpiH
H_2(X,L;\Z), \ \omega(\beta_i) \in \R_{> 0},\\
\ \lim_{i \to \infty} \omega(\beta_i) = +\infty,\
\omega(\beta_i) = 0 \iff \beta_i = 0 \Biggr\}
\end{multline*}

This Novikov ring is bigger than the conventional Novikov ring $\Lambda_0$ used
in Floer theory, which only involves the $T$-variable. In the context of the
Fukaya \ai algebra of a Lagrangian submanifold, the exponents of the
$q$-variable are, by definition, placeholders for relative homology classes of
holomorphic disks contributing to the structure maps. Abstractly, the theory of
gapped \ai algebras used below works in the same way for $\Lo$ as it does for
$\Lambda_0$. 

%In fact, the only reason we use $\Lo$ instead of $\Lambda_0$ is to
%make one argument in Section~\ref{sec:Psi} cleaner. That argument could be as
%well be run over $\Lambda_0$ if we break up the structure maps into the
%contributions of various classes $\beta$ aposteriori. So the reader may take
%$\Lambda_0$ instead of $\Lo$ throughout the paper.

There are valuation maps
$$
\val\co \Lo\to \R_{\ge 0},\quad \val\co\Lop\to \R_{>0}
$$
defined by
$$
\val\left(\textstyle\sum_i a_iT^{\omega(\beta_i)} q^{\beta_i}\right)=\min\{\omega(\beta_i):a_i\neq 0\}.
$$

Fix an orientation and a spin structure on $L\subset X$. Let $C^*(L;\Lo)$ be a
cochain complex on $L$ with coefficients in $\Lo$; to us, it is immaterial which
cochain model is used provided the Fukaya algebra can be defined over it, 
recall Section~\ref{subsec:IntroRmk}. We use
the natural grading on $C^*(L;\R)$ and the following gradings of the formal
variables: $|T|=0$, $|q^\beta|=\mu(\beta)$. Only the reduction of the grading to
$\Z/2$ will be important for us. Since the Maslov indices of all disks with
boundary on $L$ are even (by the orientability of $L$), the reduced grading
simply comes from the reduced grading on $C^*(L;\R)$

Fix a tame almost complex structure $J$ on $X$, and a suitable perturbation scheme $\ss$ turning the relevant moduli spaces of $J$-holomorphic disks with boundary on $L$ into transversely cut out manifolds,  see e.g.~\cite[Proposition~3.5.2]{FO3Book}. 
Holomorphic curve theory shows that
the vector space $C^*=C^*(L;\Lo)$ has  the structure of a \emph{gapped curved  \ai algebra} structure, called the \emph{Fukaya algebra} of $L$ \cite[Theorem~3.1.5]{FO3Book}. We denote it by
\begin{equation}
\label{eq:fuk_ai}
CF^*(L;\Lo;J,\ss)\quad \text{or} \quad (CF^*(L;\Lo),m)
\end{equation}
where $m=\{m_k\}_{k\ge 0}$ are the \ai structure operations.

Abstractly, let $C^*$ be a graded vector space over $\Lo$.
We remind that a  gapped curved \ai structure
$
(C^*,m)
$
is determined by a sequence of maps
$$
m_k\co (C^*)^{\otimes k}\to  C^*, \quad k\ge 0,
$$
of degree $2-k$, where $m_0\co\R\to C^2$
is called the \emph{curvature} and is determined by
$$
m_0(1)\in C^{2}.
$$
The curvature term is required to have non-zero valuation, that is:
\begin{equation}
\label{eq:m0}
m_0(1)=0\mod\Lop.
\end{equation}
Next, the operations
satisfy the curved \ai relations. If we denote 
$$\deg x=|x|-1$$
where $|x|$ is the grading of $x\in C^*$, see \cite[(3.2.2)]{FO3Book},
then the \ai relations read \cite[(3.2.22)]{FO3Book}
\begin{equation}
\label{eq:ai}
\sum_{i,j}(-1)^{\maltese_i}m_{k-j+1}(x_1,\ldots,x_{i},m_{j}(x_{i+1},\ldots,x_{i+j}),x_{i+j+1},\ldots,x_k)=0
\end{equation}
where
$$
(-1)^{\maltese_i}=\deg x_1+\ldots +\deg x_i+i
$$
(This convention differs from  \cite{SeiBook08} by reversing the order in which the inputs are written down.) The inner appearance of $m_j$ may be the curvature term $m_0(1)$ involving no $x_i$-inputs. For example, the first two relations read:
$$
\begin{array}{l}
m_1(m_0(1))=0,\\
m_2(m_0(1),x)+(-1)^{\deg x+1}m_2(x,m_0(1))+m_1(m_1(x))=0.
\end{array}
$$
Finally, the condition of being gapped means that the valuations of the $m_k$ operations ``do not accumulate'' anywhere except at infinity, which e.g.~guarantees the convergence of the left hand side of (\ref{eq:ai}) over the Novikov field (the \emph{adic convergence}).
We refer to \cite{FO3Book} for a precise definition of gappedness. The fact that the Fukaya algebra is gapped  follows from Gromov compactness.

The \ai relations can be packaged into a single equation by passing to the \emph{bar complex}.
First,
we extend the operations $m_k$
to
$$
\hat m_k\co (C^*)^{\otimes i}\to (C^*)^{\otimes (i-k+1)} 
$$
via
$$
\hat m_k(x_1\otimes \ldots\otimes x_i)=\sum_{l=0}^{i-k}(-1)^{\maltese_l}x_1\otimes \ldots,\otimes x_l\otimes m_k(x_{l+1},\ldots x_{l+j})\otimes x_{l+j-1}\otimes \ldots\otimes x_{k}.
$$
This in particular means that the operations are trivial whenever $k>i$, and the expression for $k=0$ reads
$$
\sum_{l=0}^{i}x_1\otimes \ldots,\otimes x_l\otimes m_0(1)\otimes x_{l+1}\otimes \ldots\otimes x_{k}.
$$
We introduce the bar complex
\begin{equation}
\label{eq:bar}
B(C^*)=\bigoplus_{i=0}^\infty (C^{*+1})^{\otimes i}
\end{equation}
and define
\begin{equation}
\label{eq:m_hat}
\hat m=\sum_{k=0}^\infty \hat m_k\co B(C^*)\to B(C^*).
\end{equation}
(This operation is denoted by $\hat d$ in \cite{FO3Book}.) The \ai relations are equivalent to the single relation
\begin{equation}
\label{eq:ai_bar}
\hat m\circ \hat m=0.
\end{equation}

\subsection{Breakdown into classes} 
Let $(C^*,m)$ be a gapped curved \ai algebra over $\Lo$.
We can decompose the \ai operations into classes $\beta\in H_2(X,L;\Z)$ as follows:
$$
m_k(x_1,\ldots,x_k)=\sum_{\beta\in H_2(X,L;\Z)}T^{\omega(\beta)}q^{\beta}m_{k,\beta}(x_1,\ldots,x_k).
$$
The operations $m_{k,\beta}$ are defined over the ground field $\R$:
\begin{equation}
\label{eq:m_k_b}
m_{k,\beta}\co C^*(L;\R)^{\otimes k}\to C^*(L;\R),
\end{equation}
and then extended linearly over $\Lo$; compare \cite[(3.5.7)]{FO3Book}.
The degree of (\ref{eq:m_k_b}) is $2-k+2\mu(\beta)$.
The gapped condition guarantees that the above sum converges adically: there is a finite number of classes $\beta$ of area bounded by a given constant that have non-trivial appearance in (\ref{eq:m_k_b}).

Geometrically, if $(C^*,m)=C^*(L;\Lo;J,\ss)$ is the Fukaya \ai algebra of a Lagrangian submanifold,
the $m_{k,\beta}$ are, by definition, the operations derived  from the moduli spaces of holomorphic disks in class $\beta\in H_2(X,L;\Z)$, see again \cite{FO3Book}.

\subsection{The classical part of an algebra}
Let $(C^*,m)$ be a gapped curved \ai algebra over $\Lo$.
Let 
$$\bC^*=C^*\otimes_{\Lo}\R$$
be the reduction of the vector space $C^*$ to the ground field $\R$.
Together with this, one can reduce the structure maps $m_k$  modulo $\Lop$. This is equivalent to setting $T=0$ or $q=0$ (equivalently: $T=q=0$ simultaneously), and gives an \ai structure defined over the ground field $\R$:
$$
\m_k\co (\bC^*)^{\otimes k}\to \bC^*,
$$
see \cite[Definition~3.2.20]{FO3Book}.
These operations are the same as the $m_{k,\beta}$ from (\ref{eq:m_k_b}) with $\beta=0$, by the gapped property. This \ai structure is no longer curved, meaning $\m_0=0$, by (\ref{eq:m0}). It is called the  \emph{classical part}  of $(C^*,m)$, and denoted by
$$
(\bC^*,\m).
$$

Now suppose that $(C^*,m)=C^*(L;\Lo;J,\ss)$ is the Fukaya algebra of a Lagrangian submanifold. Then on chain level, $\bC^*=C^*(L;\R)$. In this case $(\bC^*,\m)$ is called the \emph{topological \ai algebra} of $L$. 
The following is proven in \cite[Theorem~3.5.11 and Theorem~X]{FO3Book}.

\begin{thm}
\label{thm:top_ai}
The topological \ai algebra of $L$ 
 is quasi-isomorphic  to the de~Rham dg algebra of $L$.\qed
\end{thm}

The definition of quasi-isomorphism will be reminded later in this section.
The term \emph{quasi-isomorphism} follows Seidel's terminology \cite{SeiBook08};
the same notion is termed a \emph{weak homotopy equivalence} (between non-curved \ai algebras) in \cite[Definition~3.2.10]{FO3Book}.

Recall that $L$ is called \emph{topologically formal} if its de~Rham dg algebra is quasi-isomorphic to the cohomology algebra
$
H^*(L;\R)$
with the trivial differential.
The theorem below is due to Deligne, Griffiths, Morgan and Sullivan \cite{DGMS75}.

\begin{thm}
If a compact manifold $L$ admits a K\"ahler structure, then it is topologically formal.\qed
\end{thm}

\begin{ex}
The $n$-torus is topologically formal.
\end{ex}

\subsection{Weak homotopy equivalences} \label{subsec:whe}
Suppose $(C^*,m)$ and $(C'^*,m')$ are two gapped curved \ai algebras over $\Lo$. We remind the notion of a \emph{gapped curved \ai morphism} between them:
$$
f\co (C^*,m)\to (C'^*,m'). 
$$
It is composed of maps
$$
f_k\co (C^*)^{\otimes k}\to C'^*,\quad k\ge 0,
$$
of degree $1-k$
with  the following properties.
The  term  
$
f_0\co \R\to (C')^1
$
is required to have non-zero valuation, that is:
\begin{equation}
\label{eq:f0}
f_0(1)=0\mod \Lop.
\end{equation}
Next, the maps
$\{f_k\}$ satisfy the equations for being a curved \ai functor, see e.g.~\cite[(1.6)]{SeiBook08}:
$$
\begin{array}{l}
\sum_{i,j}
(-1)^{\maltese_i} f_{k-j+1}(x_1,\ldots,x_i,m_{j}(x_{i+1},\ldots,x_{i+j}),x_{j+1},\ldots,x_k)
\\
=\sum_r\sum_{s_1,\ldots,s_r} m'_r(f_{s_1}(x_1,\ldots,x_{s_1}),\ldots,f_{s_r}(x_{s_r},\ldots, x_k)).
\end{array}
$$
For example, the first equation reads:
\begin{equation}
\label{eq:f_1}
f_1(m_0(1))=m_0'(1)+m_1'(f_0(1))+m_2'(f_0(1),f_0(1))+\ldots
\end{equation}
where the right hand side converges by (\ref{eq:f0}). Finally, the maps $f_k$ must be gapped: roughly speaking, this again means that their valuations have no finite accumulation points so that the above sums converge. We refer to \cite{FO3Book} for details.

As earlier, we can package the \ai functor equations into a single equation passing to the bar complex. To this end, introduce a single map between the bar complexes
\begin{equation}
\hat f\co B(C^*(L;\Lo))\to B(C^*(L;\Lo))
\end{equation}
by its action on homogeneous elements:
\begin{equation}
\label{eq:f_hat_expansion}
\hat f(x_1\otimes \ldots\otimes x_k)=\sum_{s_1+\ldots+s_r=k}f_{s_1}(x_1,\ldots,x_{s_1})\otimes\ldots \otimes f_{s_r}(x_{s_r},\ldots, x_k),
\end{equation}
and extend it linearly. In particular,
$$
\hat f(1)=1+f_0(1)+f_0(1)\otimes f_0(1)+\ldots
$$
The \ai functor equations are equivalent to:
\begin{equation}
\label{eq:f_hat_eqn}
\hat f\circ \hat m=\hat m'\circ \hat f,
\end{equation}
where $\hat m$, $\hat m'$ are as in (\ref{eq:m_hat}).

We proceed to the notion of weak homotopy equivalence. Note that reducing $f$ modulo $\Lop$ gives a non-curved \ai morphism
$$
\bar f\co(\bC^*,\m)\to (\bC'^*,\m')
$$
between the non-curved \ai algebras: it satisfies $\bar f_0=0$. The first non-trivial \ai relation says that $\bar f_1$ is a chain map with respect to the differentials $\m_1$ and $\m_1'$.

\begin{dfn}
Let $f$ be a gapped curved \ai morphism as above.
We say that $\bar f$ is a \emph{quasi-isomorphism} if $\bar f_1$ induces an isomorphism on the level of homology $H(C^*,\m_1)\to H(C'^*,\m'_1)$. We say that $f$ is a \emph{weak homotopy equivalence} if $\bar f$ is a quasi-isomorphism.
\end{dfn}

It is well known that weak homotopy equivalences of \ai algebras have weak inverses. See e.g.~\cite{SeiBook08} in the non-curved case, and \cite[Theorem~4.2.45]{FO3Book} in the gapped curved case. We  record a weaker version of this result.

\begin{thm}
	\label{thm:weak_inverse}
	Let $(C^*,m)$ and $(C'^*,m')$ be two gapped curved \ai algebras over $\Lo$. If there exists a weak homotopy equivalence $(C^*,m)\to (C'^*,m')$, then there also exists a weak homotopy equivalence $(C'^*,m')\to (C^*,m)$.\qed
\end{thm}
The theorem above  can be strengthened by asserting that the two morphisms in question are weak inverses of each other; we shall not need this addition.

As a separate but more elementary property, 
being weakly homotopy equivalent is a transitive relation, because
there is an explicit formula for the composition of two \ai morphisms. Combining it with Theorem~\ref{thm:weak_inverse}, we conclude that  weak homotopy equivalence is indeed an equivalence relation.

Finally, we recall the fundamental invariance property of the Fukaya algebra of a Lagrangian submanifold.

\begin{thm}
	\label{th:inv}
	Let $L\subset X$ be a Lagrangian submanifold.
	Given two choices $J,\ss$ and $J',\ss'$, there is is weak homotopy equivalence between the Fukaya \ai algebras $CF^*(L;\Lo;J,\ss)$ and $CF^*(L;\Lo;J',\ss')$.\qed
\end{thm}

\subsection{Classically minimal algebras}

The definition of the $\Psi$-invariant given in the next section will use classically minimal models of \ai~algebras. The book \cite{FO3Book} uses a different term: it calls them \emph{canonical} algebras.

\begin{dfn}
	\label{dfn:clas_min}
Let $(C^*,m)$ be a gapped curved \ai algebra over $\Lo$. It is called \emph{classically minimal} if $\m_1=0$.
\end{dfn}

The following is a version of the homological perturbation lemma, see e.g.~\cite[Theorem~5.4.2]{FO3Book}.

\begin{thm}
	\label{thm:equiv_min}
	Let $(C^*,m)$ be a gapped curved \ai algebras over $\Lo$. 	
	Then it is weakly homotopy equivalent to a classically minimal one. 
	
	Moreover, suppose that the classical part $(\bC^*,\m)$ is formal, i.e.~quasi-isomorphic to an algebra 
	$$(H^*,\mu)$$ with an associative product $\mu$ and all other structure maps trivial. Then $(C^*,m)$ is weakly homotopy equivalent to a gapped curved \ai structure 
	$$
	(H^*\otimes_\R\Lo,m')
	$$ 
	where 
\begin{equation*}
	\m'_2=\mu,\quad\text{and}\quad  \m'_k=0\quad\text{for all}\quad k\neq 2.
\end{equation*}
\end{thm}

The following obvious lemma will play a key role in the invariance property of the $\Psi$-invariant which we will define soon.

\begin{lem}
	\label{lem:top_min}
Let $f\co (C^*,m)\to(C'^*,m')$ be a weak homotopy equivalence between two classically minimal gapped curved \ai algebras over $\Lo$. Then $\bar f_1\co \bC^*\to \bC'^*$ is an isomorphism.
\end{lem}	
\begin{proof}
By definition, $\bar f_1$ is a quasi-isomorphism between the vector spaces $C^*$ and $C'^*$ with trivial differential; hence it is an isomorphism. 
\end{proof}
We remind that, by definition,
\begin{equation}
\label{eq:f1}
f_1=\bar f_1\mod \Lop.
\end{equation}

\begin{rmk}
Suppose $J,\ss$ and $J',\ss'$ are two choices of a tame almost complex structures  with a  perturbation scheme. Denote by
$m_k$, $m_k'$
the corresponding Fukaya \ai structures on $C^*=C^*(L;\Lo)$ extending \emph{the same} topological \ai algebra structure
on $L$. 
Then there is a weak homotopy equivalence	$f\co (C^*,m)\to (C'^*,m')$ satisfying
$$
\bar f_1=\id
$$
regardless of whether  the given topological \ai structure is minimal. See \cite{FO3Book}, compare \cite[Section~5]{Sei08}.
\end{rmk}	
	
\subsection{Around the Maurer-Cartan equation}
\label{subsec:mc}
Let $(C^*,m)$ be a gapped curved \ai algebra over $\Lo$.
Consider an element 
\begin{equation}
\label{eq:a}
a\in C^{*},\quad 
a=0\mod \Lop.
\end{equation}
The \emph{prepotential} of $a$ is
\begin{equation}
\label{eq:mea}
m(e^a)=m_0(1)+m_1(a)+m_2(a,a)+\ldots\in C^*.
\end{equation}
The condition that $a=0\mod\Lop$ guarantees that the above sum converges. Moreover, (\ref{eq:m0}) implies that $m(e^a)=0\mod\Lop$.

One says that $a$ is a \emph{Maurer-Cartan element} or a \emph{bounding cochain}
if $m(e^a)=0$. Maurer-Cartan elements are of central importance for Floer theory
as they allow to deform the initial \ai\ algebra into a \emph{non-curved} one,
so that one can compute its homology. A Lagrangian $L$ whose Fukaya algebra has
a Maurer-Cartan element is called \emph{weakly unobstructed}.

Our interest in the Maurer-Cartan equation is somewhat orthogonal to this. We shall look at the prepotential itself, and will not be interested in Maurer-Cartan elements per se.

Denote
\begin{equation}
e^a=1+a+a\otimes a+\ldots\in B(C^*),
\end{equation}
where $B(C^*)$ is the bar complex (\ref{eq:bar}). For $\hat m$ as in (\ref{eq:m_hat}), it holds that
\begin{equation}
\label{eq:hmea}
\hat m(e^a)=\sum_{i,j\ge 0}a^{\otimes i}\otimes m(e^a)\otimes a^j\ \in\  B(C^*).
\end{equation}
In particular the Maurer-Cartan equation on $a$ is equivalent to the equation
$$
\label{eq:mc}
\hat m(e^a)=0,
$$
see \cite[Definition~3.6.4]{FO3Book}. More relevantly for our purposes, observe the following identity:
\begin{equation}
\label{eq:val_eb}
\val(m(e^a))=\val(\hat m(e^a)).
\end{equation}
This is because the lowest-valuation summand in (\ref{eq:hmea}) is obviously the one with $i=j=0$. Here we have extended the valuation from $C^*$ to $B(C^*)$ in the natural way: for a general element 
$$
x=x_{i^1_1}\otimes \ldots \otimes x_{i^1_{r_1}}+x_{i^2_1}\otimes \ldots \otimes x_{i^2_{r_2}}+\ldots\in B(C^*(L;\Lo)),\quad x_{i^j_k}\in C^*,
$$
we set
\begin{equation}
\val(x)=\min\{\val(x_{i^1_1})\ldots \val(x_{i^1_{r_1}}),\ \val(x_{i^1_1})\ldots \val(x_{i^2_{r_2}}),\ \ldots\}.
\end{equation}
This minimum exists for gapped $A_\infty$ algebras.
We proceed to the functoriality properties of the expression $\hat m(e^a)$.
Consider a gapped \ai morphism
$$
f\co (C^*,m)\to (C'^*,m').
$$
Take an element $a$ as in (\ref{eq:a}) and denote after \cite[(3.6.37)]{FO3Book}:
\begin{equation}
\label{eq:fstar}
f_*a=f_0(1)+f_1(a)+f_2(a,a)+\ldots\in C'^*.
\end{equation}
Note that  by (\ref{eq:f0}), it holds that $f_*(a)=0\mod\Lop$. 
Next, one has the identities
\begin{equation}
\label{eq:fea}
\hat f(e^a)=e^{f_*a}
\end{equation}
and
\begin{equation}
\label{eq:mfea}
\hat m'(\hat f(e^a))=\hat f(\hat m(e^a)),
\end{equation}
see \cite[Proof of Lemma 3.6.36]{FO3Book}. In particular, $f_*$ takes Maurer-Cartan elements to Maurer-Cartan elements; however instead of this property, we shall need the  key proposition below.
\begin{prop}
\label{prop:vals}	
Suppose $(C^*,m)$ and $(C'^*,m')$ are classically minimal (see Definition~\ref{dfn:clas_min})  gapped curved \ai algebras over $\Lo$, and $f\co (C^*,m)\to (C'^*,m')$ is a weak homotopy equivalence.
	For any element $a$ as in (\ref{eq:a}), it holds that 
	$$
	\val(m'(e^{f_*a}))=\val (m (e^a)).
	$$
\end{prop}

\begin{rmk}
	Observe that we are using $m,m'$ not $\hat m,\hat m'$ here, so that $m'(e^{f_*a})$ and $m(e^a)$ are elements of $C'^*$ rather than the bar complex, see (\ref{eq:mea}). But in view of (\ref{eq:val_eb}), we could have used the hat-versions instead.
\end{rmk}	

\begin{proof}
Using (\ref{eq:val_eb}), (\ref{eq:fea}) and (\ref{eq:mfea}) we obtain
$$
\val(m'(e^{f_*a}))
=\val(\hat m'(e^{f_*a}))=\val(\hat m'(\hat f(e^a)))=\val(\hat f(\hat m(e^a))).
$$  
Therefore we need to show that
\begin{equation}
\label{eq:fmea_val}
\val(\hat f(\hat m(e^a)))=
\val(m(e^a)).
\end{equation}
Denote 
$\lambda=\val(m(e^a))$, so that $m(e^a)$ has the form
$$
m(e^a)=T^\lambda y+o(T^\lambda)
$$
for some $y\in \bC^*$, i.e. $\val(y)=0$. In view of (\ref{eq:hmea}), it also holds that
$$
\hat m(e^a)=T^\lambda y+o(T^\lambda)\in B(C^*)
$$
where $y$ is considered as a length-1 element of the bar complex;
compare (\ref{eq:val_eb}).
Because the application of $\hat f$ (\ref{eq:f_hat_expansion}) to length-1 elements of the bar complex reduces to the application of $f_1$, we have:
\begin{equation}
\label{eq:proof_barf1y}
\hat f(T^\lambda y+o(T^\lambda))=T^\lambda f_1(y)+o(T^\lambda)=T^\lambda \bar f_1(y)+o(T^\lambda).
\end{equation}
Let us now use the hypothesis  that  $C^*,C'^*$ are classically minimal; by Lemma~\ref{lem:top_min}, it implies that $\bar f_1$ is an isomorphism. In particular $\bar f_1(y)\neq 0$, therefore $\val(\bar f_1(y))=0$.
We have shown that $\val(\hat f(\hat m(e^a)))=\lambda$, which amounts to (\ref{eq:fmea_val}).
\end{proof}

\begin{rmk}
	The above proof breaks down if $\bar f_1$ has kernel. Indeed, suppose $\bar f_1(y)=0$; this means that $\val(\bar f_1(y))$ equals $+\infty$ rather than $0$, and the valuation of the right hand side of (\ref{eq:proof_barf1y}) is strictly greater than $\lambda$. Hence, classical minimality is an essential condition for Proposition~\ref{prop:vals}.
\end{rmk}

	\begin{rmk} \label{rmk:etabeta}
	Although the above argument is inspecific to this, recall that we are working with $\Lop$ which has an additional variable $q$.
In this ring, a term  $T^\lambda y$ can be written in full form as $T^\lambda = c_\beta q^\beta T^{\omega(\beta)} \bar{y}$ where
		$\bar{y} \in \bC^*$, $c_\beta\in\R$ and $\lambda=\omega(\beta)$. 
		
		Now in the above proof, take
		$a\in C^*$ which is zero modulo $\Lop$, and let 
		$c_\beta q^\beta T^{\omega(\beta)} \bar{y}$
		be one of the lowest-valuation terms of $m(e^a)$:
 $\val(m(e^a))=\omega(\beta)$. 
 Then $m'(e^{f_*a})$  has the corresponding lowest-valuation term $c_\beta  q^\beta T^{\omega(\beta)}\bar f_1(\bar y)$. This 
		fact will be used to prove Proposition \ref{prp:etabeta} below.
		%More precisely, we will use it to show that $\xi_\beta(a) = 
		%\xi_\beta(f_*b)$, for $f$ the weak homotopy equivalence defined in \cite{FO3Book}
	\end{rmk}

\section{The invariant $\Psi$}

\label{sec:Psi}

\subsection{Definition and invariance}
\label{subsec:psi_inv}
We are ready to define the $\Psi$-invariant of a  curved, gapped \ai algebra over a Novikov ring; see the previous section for the terminology.
The Novikov ring could be $\Lambda_0$ or $\Lo$, and we choose the second option. For simplicity, we reduce all gradings modulo~2, and denote by $C^{odd}$ the subspace of odd-degree elements in a graded vector space $C^*$. We denote $[0,+\infty]=[0,+\infty)\cup\{+\infty\}$.

\begin{dfn}
\label{dfn:psi}

Let $(C^*,m)$ be a \emph{classically minimal} (see Definition~\ref{dfn:clas_min})  curved gapped \ai algebra over the Novikov ring $\Lo$.
The \emph{$\Psi$-invariant} of this algebra is defined as follows:
$$
\Psi(C^*,m)=\inf\{\val(m(e^b)) : b\in C^{odd}, \ b=0\mod \Lop\}\in[0,+\infty]
$$
Recall that the expression $m(e^b)$ was introduced in (\ref{eq:mea}).
Now suppose $(C^*,m)$ is a curved gapped \ai algebra over $\Lop$ which is not necessarily classically minimal. Let $(C'^*,m')$ be a weakly homotopy equivalent classically minimal \ai algebra, which exists by Theorem~\ref{thm:equiv_min}. We define
$$
\Psi(C^*,m)=\Psi(C'^*,m').
$$
\end{dfn}

Two remarks are due.

\begin{rmk}
The definition requires to look at \emph{odd degree} elements $b$. This is opposed to the discussion in Subsection~\ref{subsec:mc} where we imposed no degree requirements on the element $a$. Indeed one could give a definition of the $\Psi$-invariant using elements $b$ of any degree; this would be a well-defined invariant which, however, frequently vanishes. Suppose the classical part $(\bC^*,\m)$ is a minimal topological algebra of a smooth manifold $L$; then $\m_2$ is graded-commutative. If we take 
$$
b=T^\epsilon y,\quad y\in \bC^*
$$
with $\epsilon$ sufficiently small,
the expansion of $m(e^b)$ according to (\ref{eq:mea}) will contain the lowest-energy term
$$
T^{2\epsilon}\, \m_2(y,y).
$$
If $y$ has odd degree, this term vanishes by the graded-commutativity. However, if $y$ has even degree, its square may happen to be non-zero, in which case letting $\epsilon\to 0$ implies  $\Psi=0$. Note that this is not an issue when $(\bC^*,\m)$ is the (formal) topological \ai algebra of an $n$-torus (which is the most important case as far as our applications to symplectic geometry are concerned), since the cohomology of the torus is generated as a ring by odd-degree elements. But with the general case in mind, and guided by the analogy with Maurer-Cartan theory, we decided that Definition~\ref{dfn:psi}
is most natural if we only allow odd-degree elements $b$. See also Remark~\ref{rmk:l_infty}.
 \end{rmk}

\begin{rmk}
The condition that $b=0\mod\Lop$ is important due to guarantee the convergence of $m(e^b)$; compare with (\ref{eq:a}) from Subsection~\ref{subsec:mc}.
\end{rmk}

\begin{thm}
\label{thm:psi_inv}	
Let $(C^*,m)$ be a curved gapped \ai algebra over $\Lo$. Then $\Psi(C^*,m)$ is well-defined and is an invariant of the weak homotopy equivalence class of $(C^*,m)$.
\end{thm}
\begin{proof}
Any two classically minimal \ai algebras which are both weakly homotopy equivalent to $(C^*,m)$ are weakly homotopy equivalent to each other by Theorem~\ref{thm:weak_inverse}. Therefore the $\Psi$-invariants of these two classically minimal algebras are equal by Proposition~\ref{prop:vals}. Hence $\Psi(C^*,m)$ does not depend on the choice of the classically minimal model used to compute it. The same argument proves the invariance under weak homotopy equivalences of $(C^*,m)$. 
\end{proof}

\begin{dfn}
Let $L\subset X$ be a Lagrangian submanifold.
We define the $\Psi$-invariant of $L$,
$$
\Psi(L)\in [0,+\infty],
$$
to be the $\Psi$-invariant of its Fukaya algebra $CF^*(L;\Lo;J,\ss)$ for some choice of a tame almost complex structure $J$ and a perturbation scheme $\ss$. In Corollary~\ref{cor:psi_pos} we will show that $\Psi(L)$ is strictly positive.
\end{dfn}

\begin{thm}
	\label{thm:psi_L_inv}
The invariant $\Psi(L)$ is well-defined and is invariant under Hamiltonian isotopies of $L$.
\end{thm}

\begin{proof}
The invariance under 
the choices of $J$ and $\ss$ follows from Theorem~\ref{thm:psi_inv} and Theorem~\ref{th:inv}.
The invariance under Hamiltonian isotopies is automatic  by pulling back $J$ under the given Hamiltonian diffeomorphism.
\end{proof}

\subsection{Expanding the Maurer-Cartan equation}
From now on, let $(C^*,m)$ be a \emph{classically minimal} gapped curved \ai algebra over $\Lo$.
Choose a basis $b_1,\ldots,b_N$ of $\bC^{odd}$; it induces a basis of $C^{odd}$ denoted by the same symbols. Recall that $\bC$ is a \emph{minimal} \ai\ algebra over $\R$. In practice, we shall use \ai algebras based on vector spaces
$$
\bC^{*}=H^{*}(L;\R),\quad C^{*}=H^{*}(L;\Lo),
$$
but the discussion applies generally. 

We introduce the following notation. A \emph{$k$-type} is a function 
$$\nu: [N]=\{1,\ldots,N\} \to \Z_{\geq 0}$$ such that $\sum_{i \in [N]} \nu(i) = k.$ 
A map $f\co [k] \to [N]$ is said to belong to a $k$-type $\nu$ if $\nu(i) = \#f^{-1}(\{i\})$ for all $i \in [N]$. When $f$ belongs to the $k$-type $\nu$, we write $f\in \nu$. 

It is clear that an arbitrary pair of functions $f,f'\in\nu$ belonging to the same type differs by a permutation in $S_k$, i.e. $f'$ is the composition
$$
[k]\xrightarrow{\sigma}[k]\xrightarrow{f}[N]
$$ 
for $\sigma\in S_k$.

The language of $k$-types is useful for expanding polynomial expressions in non-commuting variables.
Specifically, consider an arbitrary linear combination of basic classes:
$$b = \sum_{i \in [N]} l_i b_i\in C^{odd}, \; b_i \in \bC^{odd}, \; l_i \in \Lop.$$
Denote  $l(\nu)  = (l_1)^{\nu(1)} \cdot \ldots \cdot (l_N)^{\nu(N)}.$ We also put = $l^0 = 1$ for any $l \in \Lop$.
Expanding $m(e^b)$ by  definition we obtain: 
\begin{equation}
\label{eq:meb_exp}
m(e^b) = \sum_{k \geq 0} \ \sum_{\nu \in k\text{-types}} l(\nu) \sum_{\alpha \in H_2(X,L)} T^{\omega(\alpha)} q^\alpha \sum_{f \in \nu} m_{k,\alpha} (b_{f(1)},\ldots,b_{f(k)}).
\end{equation}
Given $\alpha \in H_2(X,L;\Z),$ $k \geq 0$ and a $k$-type $\nu$, define 
\begin{equation}
\label{eq:sm_def}
sm_{k,\alpha}(\nu) : = \sum_{f \in \nu} m_{k,\alpha} (b_{f(1)},\ldots,b_{f(k)}).
\end{equation}
Now let $(i^\nu_1,\ldots,i^\nu_k)=(f(1),\ldots, f(k))$ be the sequence of values some {\it fixed} $f\in \nu$; we make one choice of $f$ for each $k$-type $\nu$ and each $k\ge 0$.
Then, by the above observation about the $S_k$-symmetry,
\begin{equation}
\label{eq:sm_sym}
sm_{k,\alpha}(\nu)=\sum_{\sigma\in S_k}m_{k,\alpha}(b_{i_{\sigma(1)}^\nu},\ldots,b_{i_{\sigma(k)}^\nu}).
\end{equation}
The meaning of $sm_{k,\alpha}(\nu)$ is that it is the sum of symmetrised \ai operations, evaluated on a collection of basic vectors $b_i$ such that the repetitions among those inputs are governed by the type $\nu$.

\begin{rmk}
	\label{rmk:sm}
	Tautologically, any sequence of numbers $(i_1,\ldots,i_k)$ where $i_j\in[N]$ is the sequence of values of a function $f\in \nu$ for \emph{some} $k$-type $\nu$. Therefore
	$$
	\sum_{\sigma\in S_k}m_{k,\alpha}(b_{i_{\sigma(1)}},\ldots,b_{i_{\sigma(k)}})=sm_{k,\alpha}(\nu)
	$$
	for some $\nu$.
\end{rmk}

\subsection{Irrationality}
Recall the valuation $\val\co \Lo\to \R_{\ge 0}$ and $\val\co \Lop\to \R_{>0}$. It induces a valuation on $C^{odd}$ denoted by the same symbol. Recall the crucial property 
$$\val(x+y) \geq \min\{\val(x),\, \val(y)\},$$ which turns into the equality whenever the lowest-valuation terms of $x$ do not cancel with those of $y$.

Denote 
$$\Om := \{\omega(\beta):\beta \in H_2(X,L; \Q)\}\subset\R.$$ This is a finite-dimensional vector subspace of $\R$ over $\Q.$ Recall that $\R$ itself is infinite-dimensional over $\Q$, in fact uncountably so. 

A $k$-type $\nu$
defines a linear map $\nu:\R^N \to \R/\Om$ by  \[(\la_i)_{i \in [N]} \mapsto \lambda(\nu) = \sum \nu(i) \lambda_i.\] 

\begin{df}\label{definition: generic}
Consider a vector $\lambda = (\lambda_i)_{i \in [N]}\in \R^N$. We call it: 
	
	\begin{enumerate}
		\item {\em generic} if the map $\{k\textit{-types}\} \to \R/\Om$ given by $\nu \mapsto \lambda(\nu)$ is injective.
		\item {\em $\Q$-independent} if $(\lambda_i)_{i \in [N]}$ induces an independent collection of $N$ vectors of in $\R/\Om$ considered as a vector space over $\Q$.
	\end{enumerate}
We call an element
$$b = \sum_{i \in [N]} l_i b_i\in C^{odd},\quad  b_i \in \bC^{odd},\quad  l_i \in \Lop,$$ generic (respectively $\Q$-independent) if $(\la_i)_{i \in [N]} = (\val(l_i))_{i \in [N]}$ is generic (respectively $\Q$-independent).
\end{df}

\begin{rmk}
It is clear that   $\Q$-independence implies  genericity. We shall further use $\Q$-independence, but genericity would also suffice for most statements.
\end{rmk}

\begin{rmk}
	Since $\R/\Omega$ is (uncountably) infinite-dimensional, the set of $\Q$-independent elements of $\R^N$ is dense.
\end{rmk}

\begin{lma}\label{lemma: independent no cancellations}
	If $b \in C^{odd}$ is $\Q$-independent, the expansion (\ref{eq:meb_exp}): \[m(e^b) = \sum_{k \geq 0}\ \sum_{\nu \in k\text{-types}} l(\nu) \sum_{\alpha \in H_2^{+}(X,L; J)} T^{\omega(\alpha)} q^\alpha sm_{k,\alpha} (\nu),\] has the following property.
	All non-zero summands in \[s(b,\nu) = l(\nu) \sum_{\alpha \in H_2(X,L;\Z)} T^{\omega(\alpha)} q^\alpha sm_{k,\alpha} (\nu)\] corresponding to different types $\nu$ have different valuations, and $s(b,\nu)$ is zero if and only if $sm_{k,\alpha} (\nu) = 0$ for each $\alpha \in H_2(X,L;\Z).$
\end{lma}

\begin{proof}[Proof of Lemma \ref{lemma: independent no cancellations}]
Denote the summands in $s(b,\nu)$ by 
	$$s_{\alpha}(b,\nu) = l(\nu) T^{\omega(\alpha)} q^\alpha sm_{k,\alpha} (\nu).$$
	Consider the vector $\la=(\la_i)_{i\in[N]} = (\val(l_i))$. First, one has $$\val(s_{\alpha}(b,\nu)) = \omega(\alpha) + \lambda(\nu)$$ if $sm_{k,\alpha} (b,\nu) \neq 0$, and $+\infty$ otherwise. Since $\lambda$ is $\Q$-independent, it holds that $$\val(s_{\alpha}(b,\nu)) \neq  \val(s_{\beta}(b,\nu'))$$ for $\nu \neq \nu',$ and each $\alpha,\beta \in H_2(X,L;\Z)$. Moreover, it holds that 
	$$s_{\alpha}(b,\nu) \neq s_{\beta}(b,\nu)$$ whenever $\alpha \neq \beta,$ and $s_{\al}(b,\nu) \neq 0$, $s_{\beta}(b,\nu) \neq 0$ because the factors $q^{\al}$, $q^{\beta}$ distinguish them. 
\end{proof}

\begin{prop}\label{prop: inf independent}
	If $(C^*,m)$ is classically minimal, it holds that
	$$\Psi(C^*,m) = \inf \{\val(m(e^b)) : b\in  C^{odd},\ b \textit{ is }\Q\textit{-independent}\}.$$
\end{prop}

%\[ b = \sum T^{\lambda_i} b_i, \; b_i \in H^{odd} \]
%
%$$ m(e^b) = \sum_{\alpha \in H_2^{+}(X,L; J)}\sum_{k \geq 0} \sum_{\nu \in k\mathrm{-types}} T^{\omega(\alpha) + \lambda(\nu)} \sum_{f \in \nu} m_{k,\alpha} (b_{f(1)},\ldots,b_{f(k)}).$$ yyy modify the above with q and l yyy
%yyy explain that all summands become of different valuations for independent elements yyy

\begin{proof}
	Let $\Psi'$ be the right hand side of the claimed equality.
	It is clear that $\Psi \leq \Psi'$. To show the converse, consider an element $b = \sum l_i b_i$ such that $\val(m(e^b)) \leq \Psi +\epsilon$. We can find $(\eps^n_i)_{i \in [N]}$ such that $b_{\eps^n} = \sum T^{\eps^n_i} l_i b_i$ is $\Q$-independent, and $\eps^n_i \to 0$ as $n\to \infty$.
	Then for some $k$ and a $k$-type $\nu$, and $\alpha \in H_2(X,L)$ 
	$$\val (m(e^b)) = \val \left(l(\nu) T^{\om(\alpha)} q^{\alpha} sm_{k,\alpha}({\nu})\right) = \omega(\alpha) + \lambda(\nu).$$ 
	By Lemma \ref{lemma: independent no cancellations}, taking the summand corresponding to these $\alpha$ and $\nu,$ $$\val (m(e^{b_{\eps^n}})) \leq \omega(\alpha) + \lambda(\nu) + \sum \nu(i) \eps^n_i.$$ Hence for $n$ sufficiently large, 
	$$\val (m(e^{b_{\eps^n}})) \leq \val (m(e^b)) + \epsilon \leq \Psi + 2\epsilon.$$ Therefore $\Psi' \leq \Psi +2\eps,$ which implies $\Psi' \leq \Psi$ since $\epsilon$ is arbitrary.
\end{proof}

\subsection{A computation of $\Psi$}
The next theorem is very useful for computing $\Psi$  in practice. 

\begin{thm}\label{theorem: explicit}
		If $(C^*,m)$ is classically minimal, it holds that
	\[\Psi(C^*,m) =  \min \left\{ \omega(\alpha): \exists\,  c_1,\ldots,c_k \in \bC^{odd}\textit{ s.t. } \sum_{\sigma \in S_k} m_{k,\alpha} (c_{\sigma(1)},\ldots,c_{\sigma(k)}) \neq 0\right\}.\]
\end{thm}

Recall that the operations $m_{k,\alpha}$ are from (\ref{eq:m_k_b}); they land in  $\bC^*$. The set from the statement includes the case when $k=0$, $m_{0,\alpha}(1)\neq 0$. 

\begin{proof}
	Expanding the brackets by multi-linearity reduces the statement to the case when all $c_{i}$ belong to a basis $(b_i)_{i \in [N]}$ of $\bC^{odd}$. So it is enough to show that
	$$\Psi =\min \left\{ \omega(\alpha): \exists\, k\textit{ and a } k\textit{-type } \nu \textit{ such that } sm_{k,\alpha}(\nu) \neq 0\right\}.$$
	Denote the right hand side by $\Psi''$.
	 This minimum exists by the gapped condition (or by Gromov compactness, in the case $C^*$ is the Fukaya algebra). 
	 
	 Returning to the definition of $\Psi$, recall that by Proposition \ref{prop: inf independent}, we may restrict to elements
	 $b = \sum l_i b_i$ that are $\Q$-independent. For such an element, Lemma~\ref{lemma: independent no cancellations} says that
	 $$
	 \val(m(e^b)) = \min \{\omega(\alpha) + \lambda(\nu):\exists\, k\textit{ and a } k\textit{-type } \nu \textit{ such that } s_{\alpha}(b,\nu) \neq 0\},$$ 
	 where the minimum exists since $l_i \in \Lambda_+$ the \ai algebra is gapped. From this it is immediate that $\Psi \geq \Psi''.$ 
	 
	 To show the other inequality, let $\nu$ be a $k$-type and $\alpha\in H_2(X,L;\Z)$ a class such that  $sm_{k,\alpha}(\nu) \neq 0$ and $\Psi''=\omega(\alpha)$. Consider the ring automorphism $r_n:\Lo \to \Lo$ given by $T \mapsto T^{1/n}$. Clearly, it preserves $\Lop$ and satisfies 
	 $$\val (r_n(l)) = \frac{1}{n} \val(l).$$ It induces a map $r_n: C^{odd} \to C^{odd}$, which preserves $\Q$-independence. Then, evaluating at the $\nu$-summand of the expansion (\ref{eq:meb_exp}), we have  $$\val(m(e^{r_n(b)})) \leq \om(\al) + \frac{1}{n} \lambda(\nu).$$ 
	 The right hand side converges to $\Psi''$ as $n \to \infty$, so $\Psi \leq \Psi''.$
\end{proof}

\begin{rmk}
	\label{rmk:l_infty}
	Because  $\Psi$ is determined by the symmetrised \ai operations, which is evident from both Definition~\ref{dfn:psi} and Theorem~\ref{theorem: explicit}, the $\Psi$-invariant can alternatively be defined in the setting of $L_\infty$ algebras, cf.~\cite{CoLa05,CL06,Cho12} and \cite[A.3]{FO3Book}. A comparison between Definition~\ref{dfn:psi} and Theorem~\ref{theorem: explicit} in the $L_\infty$ context follows the same arguments as above.
\end{rmk}	

\subsection{Concavity} 
Below is the main geometric property of $\Psi$, that it is
concave under Lagrangian star-isotopies. This is a very powerful property that
makes $\Psi$ amenable to explicit computation, for example, for fibres of
various (singular) torus fibrations.

\begin{thm}
\label{th:psi_concave} Let $X$ be a symplectic manifold and $\{L_t\subset X\}_{t\in[0,1]}$ a Lagrangian star-isotopy, i.e.~a Lagrangian isotopy whose flux develops linearly. Then the function
$$
\Psi(L_t)\co [0,1]\to[0,+\infty]
$$
is continuous and concave in $t$.
\end{thm}

\begin{rmk}
If $\Psi(L_t)$ achieves value $+\infty$ for some $t$, concavity means that $\Psi(L_t)\equiv+\infty$ for all $t$.
\end{rmk}

\begin{proof}
By definition of star-isotopy, there is a vector $\ff\in H^1(L_0)$ such that
$$
\Flux(\{L_t\}_{t\in[0,t_0]})=t_0 \cdot \ff,\quad t_0\in[0,1].
$$
Denote $L_0=L$.
		
Consider an integer $k\ge 0$
and arbitrary classes $\beta\in H_2(X,L;\Z)$, $c_1,\ldots,c_k\in \bC^{odd}=H^{odd}(L;\R)$.
By continuity, they canonically define classes 
$\beta^t\in H_2(X,L_t;\Z)$ and $c_1^t,\ldots,c_k^t\in H^{odd}(L_t;\R)$.
Let $\{m^t_k\}_{k\ge 0}$ be structure maps of the Fukaya algebra of $L_t$.
By Fukaya's trick, 
there exists a choice of almost complex structures such that for all sufficiently small $t$,
$$
m^t_{k,\beta^t} (c_1^t,\ldots,c_k^t)\quad\textit{does not depend on }t,
$$
for any $k$, $\beta$ and $c_i$ as above. Recall that these operations, by definition (\ref{eq:m_k_b}), do not remember the area of $\beta^t$ which can change in $t$.
Now let 
$$I=\{\beta_i\}\subset H_2(X,L;\Z)$$ be the set of all classes $\beta$ such that $\omega(\beta)=\Psi(L)$ and $\beta$ satisfies the property from the right hand side of the formula from Theorem~\ref{theorem: explicit}: that is, holomorphic disks in class $\beta$ witness the value $\Psi(L)$. By Gromov compactness, this set is finite.

By Theorem~\ref{theorem: explicit} and the Fukaya trick, it is precisely some of the disks among $\{\beta_i^t\}_{i\in I}$ that  witness the value $\Psi(L_t)$. 
Namely:
$$
\Psi(L_t)=\min\{\omega(\beta_i^t)\}_{i\in I}.
$$
By definition of flux, for all  $t$: 
$$
\omega(\beta_i^t)=\omega(\beta_i)+t(\ff\cdot\del \beta_i)=\Psi(L)+t(\ff\cdot\del \beta_i)
$$
where $\ff\cdot\del \beta_i\in\R$ is the pairing. Observe that this function is linear in $t$.
So $\Psi(L_t)$, for small enough $t$, is the minimum of several linear functions. It follows that $\Psi$ is continuous and concave at $t=0$. The argument may be repeated at any other point $t$.
\end{proof} 

%!TEX root = Shapes.tex

\section{Bounds on flux}
\label{sec:gen_bounds}

\subsection{Indicator function}
For the first two subsections, we continue working in the general setting of Sections~\ref{sec:Fuk},~\ref{sec:Psi}.
Let $(C^*,m)$ be a classically minimal (Definition~\ref{dfn:clas_min})  curved gapped \ai algebra over the Novikov ring $\Lo$.
Define the indicator function
$$
\eta_L\co H_2(X,L;\Z)\to \{0,1\}
$$
as follows.
We put $\eta_L(\beta) = 1$ if and only if $\omega(\beta)=\Psi(L)$ and 
$$
\exists k\ge 0,\  \exists\,c_1,\ldots,c_k \in \bC^{odd}\textit{ such that } \sum_{\sigma \in S_k} m_{k,\beta} (c_{\sigma(1)},\ldots,c_{\sigma(k)}) \neq 0.
$$
It means that holomorphic disks in class $\beta$ witness the minimum from Theorem~\ref{theorem: explicit}, computing $\Psi$. 
In other words, holomorphic disks in class $\beta$ are the lowest area holomorhic disks that contribute non-trivially to some symmetrised \ai\ structure map on odd-degree elements. The indicator function depends on  $(C^*,m)$, although this is not reflected in the notation.

\begin{prp} \label{prp:etabeta}
	The indicator function $\eta_L$ is invariant under weak homotopy equivalences $(C^*,m)\to (C'^*, m')$ between classically minimal \ai algebras.
\end{prp}

\begin{proof}
This is analogous to the proof of Proposition~\ref{prop:vals}. Specifically, see Remark~\ref{rmk:etabeta}.
\end{proof}

\begin{rmk}
	One can use 
	$\eta_L$ to distinguish non-monotone Lagrangian submanifolds up to Hamiltonian isotopy.
	In particular, this implies \cite[Conjecture~1.5]{Vi16b} 
	distinguishing certain Lagrangian tori in $(\CP^1)^{2m}$.
\end{rmk}

\subsection{Curvature term}
The following lemma will be useful in Section~\ref{sec:Ex}.

\begin{lem}
	\label{lem:m0_invt}
Suppose $\beta\in H_2(X,L;\Z)$ is a class satisfying
$\eta_L(\beta)=1$. 
 Then  $m_{0,\beta}(1)\in \bC^{even}$ is invariant under weak homotopy equivalences
 $f\co (C^*,m)\to (C'^*,m')$ between classically minimal algebras. Namely, one has $$m'_{0,\beta}(1)=\overline{f}_1(m_{0,\beta}(1))$$ where $\overline{f}_1\co \bC^*\to \bC'^*$ is an isomorphism.
\end{lem}

\begin{proof} Let us look at the \ai functor equation (\ref{eq:f_1}): 
$$
f_1(m_0(1))=m_0'(1)+m_1'(f_0(1))+m_2'(f_0(1),f_0(1))+\ldots $$ 
Recall that $f_0(1)$ has positive valuation, by definition. It now follows from
Theorem~\ref{theorem: explicit} and the invariance of $\Psi$ that the above
equation taken modulo the ideal $t^{\Psi(L)}\cdot \Lambda_+$ is: 
$$ \sum_{\beta\
:\ \eta_L(\beta)=1}\overline{f}_1(m_{0,\beta}(1))q^\beta
T^{\omega(\beta)}=\sum_{\beta\ :\ \eta_L(\beta)=1}m'_{0,\beta}(1)q^\beta
T^{\omega(\beta)}. $$ 
It implies that $m'_{0,\beta}(1)=\overline{f}_1(m_{0,\beta}(1))$ for all $\beta$
such that $\eta(\beta)=1$. \end{proof}

\subsection{Positivity}
Let $L\subset X$ be a Lagrangian submanifold, and consider  a classically minimal model of its Fukaya \ai algebra over 
$C^*=H^*(L;\Lo)$, for some choice of a compatible almost complex structure. Recall that $\bC^*=H^*(L;\R)$;
it supports the topological part of the \ai algebra, $\{\m_k\}_{k\ge 2}$. The topological differential $\m_1$ vanishes by the minimality condition. Recall that $S_k$ denotes the symmetric group. The lemma below says that the $\Psi$-invariant of Lagrangian submanifolds is non-trivial; if the lemma were false, the $\Psi$-invariant would vanish. The following is shown in \cite[Theorem~A3.19]{FO3Book}.

\begin{lem}
	\label{lem:top_sym}
	For all $k\ge 2$ and all $c_1,\ldots,c_k \in \bC^{odd}$ it holds that $$\sum_{\sigma\in S_k}\m_{k} (c_{\sigma(1)},\ldots,c_{\sigma(k)}) = 0.
$$
\end{lem}

\begin{rmk}
	The case $k=2$ follows from the fact that the product on $H^*(L;\R)$ is anti-commutative on odd degree elements.
		If $L$ is topologically formal, e.g.~$L$ is the $n$-torus, the full statement  also quickly follows from Theorem~\ref{thm:equiv_min}.
\end{rmk}

\begin{rmk}
We provide an argument for completeness.
By the invariance of $\Psi$ shown in Section~\ref{sec:Psi}, if the lemma holds true for some minimal \ai algebra $\bC^*$, it holds for any other minimal \ai algebra quasi-isomorphic to it.	
Let $C^*_{dR}(L;\R)$ be the de Rham complex of the space of differential forms on $L$, considered as an \ai algebra with the de Rham differential, exterior product $\mu$ and trivial higher structure operations. By Theorem~\ref{thm:top_ai}, there is an \ai quasi-isomorphism $\bC^*\to C^*_{dR}(L;\R)$. Moreover, starting from $C^*_{dR}(L;\R)$,
one can construct a minimal \ai\ algebra quasi-isomorphic to it (hence quasi-isomorphic to $\bC^*$) explicitly by  the Kontsevich perturbation formula. Looking at the formula, see e.g.~\cite[Proposition~6]{Mar06}, one sees that starting with $\mu$ that is skew-symmetric on odd-degree elements, the  symmetrisations of the \ai operations on the minimal model also vanish on odd-degree elements.
\end{rmk}

\begin{cor}
	\label{cor:psi_pos}
For any Lagrangian submanifold $L\subset X$, $\Psi(L)$ is strictly positive.
\end{cor}

\begin{proof}
	This follows from Theorems~\ref{theorem: explicit},\ref{thm:equiv_min} and Lemma~\ref{lem:top_sym}.
\end{proof}

\subsection{General shape bound}
From now on, $\eta_L$ shall denote the indicator function of the Fukaya algebra of $L\subset X$. 
The theorem below is the main bound on star-shape in terms of the $\Psi$-function. It says that every class $\beta$ such that $\eta_L(\beta)=1$ constrains $Sh^\st_L(X)$ to an affine half-space bounded by the hyperplane in $H^1(L;\R)$ which is a translate of the annihilator of $\del \beta\in H_1(L;\Z)$. If the indicator function $\eta_L$ equals~1 on several classes $\beta$ with different boundaries, the shape  consequently becomes constrained to the intersection of the corresponding half-spaces.

\begin{thm}
	\label{th:star_shape_char}
	Let $X$ be a symplectic manifold and $L\subset X$ an oriented spin Lagrangian submanifold. Then, for any $\beta \in H_2(X,L;\Z)$ such that $\eta_L(\beta)=1$, the following holds.
	$$Sh_L^\st(X)\subset \{\ff \in H^1(L;\R): \Psi(L) + \ff\cdot \del \beta > 0\}.$$
\end{thm}

\begin{proof}
	Suppose $\eta_L(\beta)=1$, and consider a star-isotopy $\{L_t\}_{t\in [0,1]}$, $L_0=L$ with total flux $\ff\in H^1(L;\R)$.
	We begin as in the proof of Theorem~\ref{th:psi_concave}.
	By definition of star-isotopy:
	$$
	\Flux(\{L_t\}_{t\in[0,\epsilon]})=\epsilon \cdot \ff,\quad \epsilon\in[0,1].
	$$
	Let $\{m^t_k\}_{k\ge 0}$ be the structure maps of the Fukaya algebra of $L_t$.
	Recall Fukaya's trick used in the proof of proof of Theorem~\ref{th:psi_concave}: there exists a sufficiently small $t_0>0$ such that for all $\epsilon<t_0$,
$$
m^\epsilon_{k,\beta^\epsilon} (c_1^\epsilon,\ldots,c_k^\epsilon)\quad\textit{does not depend on }\epsilon.
$$
So by Theorem~\ref{theorem: explicit},
$$
\Psi(L_\epsilon)\le \omega(\beta^\epsilon).
$$
But
$$
\omega(\beta^\epsilon)=\omega(\beta)+\epsilon(\ff\cdot\del \beta)=\Psi(L)+\epsilon\ff\cdot\del \beta.
$$
It follows that
$$
\Psi(L_\epsilon)\le \Psi(L)+\epsilon (\ff\cdot\del \beta),\quad \epsilon\in[0,t_0].
$$	
Because $\Psi(L_t)$ is concave and continuous in $\epsilon$, while the right hand side depends linearly on $\epsilon$, the same bound holds globally in time:
$$
\Psi(L_t)\le \Psi(L)+t (\ff\cdot\del \beta),\quad t\in[0,1].
$$	
The theorem follows by taking $t=1$ and using the fact that $\Psi(L_1)$ is positive.
\end{proof}

Here is a convenient reformulation of Theorem~\ref{th:star_shape_char}.
Define the \emph{low-area Newton polytope} to be the following convex hull:
$$
\P_L^{low}=\Conv\left\{\del \beta\in H_1(L;\Z): \eta_L(\beta)=1\right \}\subset H_1(L;\R).
$$
Now consider its (open) dual polytope:
$$
(\P_L^{low})^\vee=\left\{\alpha\in H^1(L;\R):\alpha\cdot a>-1 \ \ \forall a\in \P_L^{low}\right\}\subset H^1(L;\R)
$$ and call \begin{equation}\label{def: QL} Q(L) = \Psi(L)\cdot (\P_L^{low})^\vee,\end{equation} where $\Psi(L)$ is used as a scaling factor, the {\em obstruction polytope} of $L$. Then Theorem~\ref{th:star_shape_char} may be rewritten as follows:
\begin{equation}
\label{eq:sh_p_low}
Sh^\st_L(X) \subset Q(L).
\end{equation}

\subsection{Stabilizations}\label{sec: stab}

For dynamical applications, for instance extending \cite[Remark 2.13]{EGM16}, it is useful to consider the following situation. Let $L \subset X$ as before be an orientable spin Lagrangian submanifold. Consider its product $L'=L \times 0_{S^1} \subset X'=X \times T^*S^1,$ where $0_{S^1} \subset T^*S^1$ is the zero-section. Then, since $0_{S^1}$ is exact, $\Psi(L') = \Psi(L),$ and we have the following calculation. 

\begin{cor}\label{cor: stab}
The following identity of obstruction polytopes holds \[Q(L') = Q(L) \times \R.\] Moreover, if $Sh^\st_L(X) = Q(L)$ then $Sh^{\st}_{L'}(X') = Sh^{\st}_L(X) \times \R.$
\end{cor}

Indeed, choosing a split almost complex structure $J' = J \times j_0$ on $X',$ for the standard almost complex structure $j_0$ on $T^*S^1,$ we see that the classes $\beta' \in H_2(X',L';\Z)$ with $n_{L'}(\beta') = 1$ are given by $\beta' = i(\beta)$ for $\beta \in H_2(X,L;\Z)$ with $n_L(\beta) = 1,$ where $i: H_2(X,L;\Z) \to H_2(X',L';\Z)$ is the natural map. The identity of obstruction polytopes follows immediately. The moreover part holds by \eqref{eq:sh_p_low} applied to $L'$ in $X'$ and lifting a star-isotopy $\{L_t\}$ in $X$ to a star-isotopy $\{L'_t = L_t \times \mathrm{graph}(t \cdot c d\theta)\}$ in $X',$ for an arbitrary $c \in \R,$ where $d\theta$ is the standard one-form on $S^1.$

\subsection{Landau-Ginzburg potential} \label{subsec:LG_Pot}
	Let $L\subset X$ be a \emph{monotone} Lagrangian submanifold and $J$ a tame almost complex structure.
	The Landau-Ginzburg potential \cite{Au07,Au09, FO3Book} is the following Laurent polynomial in  $d=\dim H_1(L;\R)$ variables:
	\begin{equation}
	\label{eq:potential}
	W_L(J)=\sum_{\beta\in H_2(X,L)\ :\ \mu(\beta)=2}\#\M_\beta(J)\cdot\mathbf{ x}^{\del \beta}.
	\end{equation}
	Here $\# \M_\beta(J)$ counts $J$-holomorphic disks in class $\beta$ passing through a specified point on $L$.
	We use the notation
	$\mathbf{x}^l=x_1^{l_1}\ldots x_d^{l_d}$ where $l_i$ are the coordinates of an integral vector $l\in H_1(L;\Z)/Tors\cong\Z^d$ in a chosen basis. 
	Since $L$ is monotone, its LG potential does not depend on $J$ and on Hamiltonian isotopies of $L$. 

Consider the Newton polytope of $W_L$, denoted by $\P_L$. Explicitly:
$$
\P_L=\Conv\left\{\del \beta\in H_1(L;\Z):\mu(\beta)=2,\ \# \M_\beta(J)\neq 0\right\}\subset H_1(L;\R).
$$
The convex hull is taken inside $H_1(L;\R)$, where we consider each $[\del
\beta]\in H_1(L;\Z)$ as a point in $H_1(L;\R)$ via the obvious map $H_1(L;\Z)\to
H_1(L;\R)$. We define $\P_L = H_1(L;\Z)$ if the considered set is empty.

\begin{cor} \label{cor:P_monot} 
	Let $L\subset X$ be a monotone Lagrangian submanifold with monotonicity constant $c$, i.e.~$\omega=c\mu\in H^2(X,L)$. Then
	$$
	Sh^\st_L(X) \subset 2
	c\,\cdot\P_L^\vee. 
	$$
\end{cor}

\begin{proof}
	Since $L$ is monotone (and orientable), Maslov index~2 classes have lowest positive symplectic area in $H_2(X,L;\Z)$.
	Note that when $\mu(\beta)=2$, $\#\M_\beta(J)$ contributes to the $m_0$-operation of the Fukaya algebra as follows:
	$$
	m_{0,\beta}(1)=T^{\omega(\beta)}q^\beta\#\M_\beta(J)\cdot [1_L]
	$$
	where $[1_L]\in H^0(L;\R)$ is the unit.

In view of this and Lemma~\ref{lem:top_sym}, we conclude the following about the characteristic function 
$\eta_L$. 
If $\mu(\beta)=2$ and $\M_\beta(J)\neq 0$, one has $\eta_L(\beta)=1$. By definition, the classes $\del\beta$ for such $\beta$ span $\P_L$. Furthermore, $\Psi(L)=\omega(\beta)=2c$. The statement follows from Theorem~\ref{th:star_shape_char}.
\end{proof}

\subsection{A modification in dimension four}
\label{subsec:sh_bound_dimf}
	Suppose  $\dim X = 4$ and $L\subset X$ is a Lagrangian submanifold. 
	Fixing a compatible almost complex structure $J$, define $\Psi_2(L)$ as the lowest area among
	Maslov index~2 classes $\beta$ such that $\# \M_\beta(J) \ne 0$. Let
	$\P_L^{low,\, 2}$ be the convex hull of the boundaries $\del\beta$ of such classes. We
	claim that with these adjustments, it still holds that 
	$$Sh^\st_L(X) \subset \Psi(L)(\P_L^{low,\, 2})^\vee.$$ 
	Indeed, for a generic path $J_t$ of almost complex structures, all
	$J_t$-holomorphic disks have Maslov index $\ge 0$, hence disks of Maslov
	index~$>2$ cannot bubble from disks of Maslov index~$2$ and do not interfere with the counts. We leave the details to the reader.

\section{Computations of shape}

\label{sec:Ex}

In this section we 
 study the $\Psi$-function and shapes for \emph{Gelfand-Cetlin fibrations}, which are generalisations of toric fibrations.
Next, we 
compute shapes and star-shapes of Clifford and Chekanov tori in $\C^n$; and study the wild behaviour of (non-star) shape in $\CP^2$.

\subsection{Gelfand-Cetlin fibrations} \label{subsec:ATF-GCF}
Only a small fraction of Fano varieties are toric, but one may broaden this class by allowing more singular Lagrangian torus fibrations which are reminiscent of the toric ones. 
We shall work with a class of fibrations which we call  Gelfand-Cetlin fibrations.
The name is derived from Gelfand-Cetlin fibrations on partial flag varieties \cite{NNU10,CKO17}; fibrations with similar properties on $\CP^2$ and $\CP^2\times\CP^1$ appeared in
\cite{Wu15}. More generally, 
 any toric Fano degeneration gives rise to a Gelfand-Cetlin fibration by the result of \cite{HaKa15}.

We axiomatise the general properties of those constructions in the following notion.

\begin{dfn} \label{dfn:GCF} 
	A Gelfand-Cetlin fibration (GCF) on a symplectic $2n$-manifold $X$ is given by a \emph{continuous} map $\mu: X \to \Pol = \mathrm{Im}(\mu) \subset \R^n$, called the moment map, whose image is a compact convex lattice polytope $\Pol$ with the following property. Denote by $\Pol^{codim\ge 2}$ the union of all faces of $\Pol$ of dimension at most $n-2$ (i.e.~all faces  except for the facets). 
	
	The requirement is that
	$\mu$ is a smooth map over  $\Pol\setminus \Pol^{codim\ge 2}$, and is an actual toric fibration over it.
	This means that over the interior of $\Pol$, $\mu$ is a Lagrangian torus fibration without singularities, and over the open parts of  the facets of $\Pol$ it and has
	standard elliptic corank one singularities.
\end{dfn}

Note that in most examples, $\mu$ is not  smooth (only continuous) over $\Pol^{codim\ge 2}$, and the preimage of any point in $\Pol^{codim\ge 2}$ is a smooth isotropic submanifold of $X$. 
Observe that it is a reasonable conjecture that every Fano variety admits a toric Fano degeneration, and if this holds true, it follows every Fano variety admits a Gelfand-Cetlin fibration.

%\begin{lem}
%	Suppose $X$ is monotone and $\Pol$ is a Fano compact polytope. Then the barycentre of $\Pol$ defines a monotone Lagrangian torus $L\subset X$.
%\end{lem}

%\begin{proof}
%The normal line segment from the barycentre onto a facet defines a Maslov index~2 disk in $H_2(X,L)$. Its area is proportional to the affine length of the segment. By the Fano condition, all such segments have the same affine length. Finally, because $\Pol$ is compact, the boundaries of such disks rationally span $H_1(L;\Q)$. Hence monotonicity follows.
%\end{proof}

%\begin{thm} \label{th:starshapeMonfibre}
%	Let $X$ be a monotone symplectic manifold, $L\subset X$ the monotone fibre of a GCF with moment polytope $\Pol$ translated so that $L$ corresponds to the origin,
%	and  $\Pol^0$ be its interior. Then $Sh_L^\st(X) =\Pol^0$. 
% \end{thm}
Recall the statement of Theorem~\ref{th:GCF}:

\begin{thm}\label{thm:GCF2}
	Let $X$ be a Fano variety, $X\to \Pol\subset\R^n$ a Gelfand-Cetlin fibration, and  $L\subset X$ its monotone Lagrangian fibre. 
	
	Let $\P_L^\vee\subset H^1(L;\R)$ be the interior of the dual of the Newton polytope
	associated with the Landau-Ginzburg potential of $L$ (Section~\ref{subsec:LG_Pot}). Let $c$ be the monotonicity constant of $X$, and assume $\Pol$ is translated so that the origin corresponds to the fibre $L$. Then the following three subsets of $H^1(L;\R)\cong \R^n$ coincide:
	$$2c\cdot  \P_L^\vee = \Pol^0=Sh^\st_L(X).$$
\end{thm}

\blk 

\begin{rmk}
	\label{rmk:shape_atf}
	The statement is also true if $\dim X=4$, $L$ is the 
	monotone fibre of an almost toric fibration over a disk, and $\Pol$ is the polytope of the limiting orbifold
	\cite[Definition~2.14]{Vi16a}. We leave the obvious modifications to the reader.
\end{rmk}

Using the fact that $\mu$ is the standard toric fibration away from $\Pol^{codim\ge 2}$, the normal segment $I$ from the point $x_L=\mu(L)$ onto that facet determines a Maslov index~2 class in $H_2(X,L;\Z)$; see Figure~\ref{fig:facet}. Using the identification $H^1(L;\R) \cong \R^n$ coming from the embedding $\Pol\subset \R^n$, we see that the boundary of $\beta$ is given by the exterior normal to the chosen facet. Given a class $\beta$ arising from this construction using some facet, we will denote that facet by $\Pol_\beta$.

\begin{lem}
	\label{lem:Count=1}
	For each facet of $\Pol$, the corresponding class $\beta\in H_2(X,L;\Z)$ satisfies $\eta_L(\beta)=1$ where $\eta_L$ is the characteristic function from Section~\ref{sec:gen_bounds}.
\end{lem}

%In fact, a stronger statement holds.

%\begin{lem}
%	\label{lem:gc_count_one}
%	For each facet of $\Pol$,
% the count of disks in the corresponding  class $\beta$ through a specified point on $L$ equals one.
% \end{lem}

We need some preliminary lemmas first.
For each $x\in I$, denote by $L_x\subset X$ the fibre over it, and by $\beta_x\in H_2(X,L_x;\Z)$ the class obtained from  a class $\beta\in H_2(X,L;\Z)$ by continuity.

Pick a facet $\Pol_\beta$ of the moment polytope, and  consider
a star-isotopy $\{L_t\}_{t\in[0,1)}$ 
corresponding to the segment $I \cong [0,1] \subset \Pol$ starting from the
monotone torus $L = L_0$, and going towards $\Pol_\beta$ in the direction normal to it. 
Here $\beta$ stands for the class in $H_2(X,L;\Z)$ corresponding to the chosen facet.

\begin{lem} \label{lem:Count1_p_epsilon}
	For any $t \in I\setminus\{1\}$ sufficiently close to the facet in question, there exists a compatible almost complex structure $J'$ for which
	the algebraic count of $J'$-holomorphic disks in class $\beta_{t}$ passing through a 
	fixed point on $L_{t}$ equals one. Moreover, $\beta_{t}$ 
	has minimal area among all $J'$-holomorphic disks on $(X,L_{t})$. Thus, $\Psi(L_{t}) 
	= \omega(\beta_{t})$ and $\eta_{L_t}(\beta_t)=1$.
\end{lem}

\begin{proof}
	Denote by $o$ the endpoint $I \cap \Pol_\beta$. There is a symplectomorphism
 between the $\mu$-preimage of a neighbourhood $U$ of $o$ in $\Pol$ and a
 neighbourhood of $T^{n-1} \times \{0\}$ inside $$T^*T^{n-1} \times \C$$ with
 the standard symplectic form; see Figure~\ref{fig:facet}. Moreover, one can
 arrange this symplectomorphism to map $L_t$ to a product torus of the form
 $T^{n-1}\times \{\text{a circle}\}$.
	
 The class $\beta_{t}$ is represented in $H_2(U,L_{t})$ and is identified in
 this model with the disk class in the second $\C$-factor. Moreover, for the
 standard Liouville structure on $T^*T^{n-1}\times \C$, the algebraic count of
 disks in $(U,L_{t})$ equals one, by an explicit computation. We want to show
 that for some $J'$ on $X$, there are no holomorphic disks in the same class
 that escape $U$.
	
\begin{figure}[h]
	\includegraphics[]{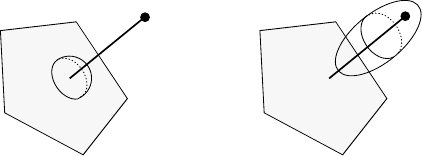}
	\caption{Domains $U$ and $V$ used for stretching in the proof of Lemma~\ref{lem:Count1_p_epsilon} resp.~Lemma~\ref{lem:Count=1}. The dotted point corresponds to the monotone fibre $L$.}
	\label{fig:facet}
\end{figure}

 The idea is to fix $U$ and consider almost complex structures $J'$ which are
 sufficiently stretched around $\del U$, with respect to some Liouville form on
 $U$. Note that as $t \to o$, $\omega(\beta_{t})\to 0$, in particular this area
 becomes eventually smaller than the action of any 1-periodic Reeb orbit in
 $\del U$.  	Now suppose that for each almost complex structure in the
 stretching sequence, there is a holomorphic disk in $ (X,L_{t})$ in class
 $\beta_{t}$ that escapes $U$. In the SFT limit, such disks converge to a
 holomorphic building with a non-trivial holomorphic piece $u'$ in  $X\setminus
 U$ having punctures asymptotic to Reeb chords in $\del U$. In particular, $u'$
 defines a 2-chain in $(X\setminus U,\del U)$.
	
	Let $\omega$ be the initial symplectic form on $X$. First, we have that
	$$
	\omega(u')<\omega(\beta_{t}).
	$$ 	
	But the Lemma~\ref{lem:area_u'} below gives a bound in the other direction; this contradiction proves Lemma~\ref{lem:Count1_p_epsilon}.

 As in the proof of Corollary~\ref{cor:P_monot}, our computation of Maslov
 index~2 disks shows that $m_{0,\beta_t}(1)\neq 0$ for the stretched almost
 complex structure. To argue that $\Psi(L_{t}) = \omega(\beta_{t})$ and
 $\eta_{L_t}(\beta_t)=1$, we must show that $m_{0,\beta_t}(1)\neq 0$ in some
 classically minimal model, while strictly speaking we have computed
 $m_{0,\beta_t}(1)=[L_t]$ as the fundamental chain in some (not necessarily
 minimal) chain model, depending on the setting of the Fukaya algebra. In the
 case of the stabilising divisor approach, one can take a perfect Morse function
 on $L_t$ which automatically gives the computation in a minimal model. In
 general, the application of homological perturbation lemma will take
 $m_{0,\beta_t}(1)$ to its cohomology class (the fundamental class).
\end{proof}

\begin{rmk}
If one uses the stabilising divisor approach to Fukaya algebras, one needs to make sure that the above SFT stretchings are compatible with keeping the stabilising divisor complex. The simplest way to ensure this is by imposing an extra condition in the definition of the Gelfand-Cetlin fibration, which is again satisfied in examples and should  be generally satisfied for the fibrations arising by the general mechanism of \cite{HaKa15}. The condition is that $X$ has a smooth anticanonical divisor projecting, after a suitable Hamiltonian isotopy, to an arbitrarily small neighbourhood of $\del \Pol$, and which coincides with the preimages of the facets of $\Pol$ away from an arbitrarily small neighbourhood of the set of codimension~$\ge 2$ faces of $\Pol$. This divisor should be stabilising for the monotone torus (and hence it will be stabilising for any torus which is the preimage of an interior point of $\Pol,$ after a Hamiltonian isotopy if necessary). This way, the neighbourhood in Figure~\ref{fig:facet} (left) intersects the divisor in the standard way which makes it possible to stretch the almost complex structure keeping it complex. The neighbourhood in Figure~\ref{fig:facet} (right) does not intersect the divisor at all, which again makes consistent stretching possible.
\end{rmk}

\begin{lem} \label{lem:area_u'}
	In the setting of the previous proof, it holds that $\omega(u')$ greater than the sum of the actions of its asymptotic  Reeb orbits.
\end{lem}	

\begin{proof}
	Consider the space
	\[ (X \setminus U, \omega_+^\infty) \cong (X \setminus U, \omega) \cup ([-\infty,0] \times
	\del U, d(e^t\theta))\]
	obtained by attaching the infinite negative Liouville collar to $(X\setminus U,\omega)$. 
	Here $\theta$ is the Liouville contact form on $\del U$,
	$d\theta = \omega|_{\del U}$. 
	By the construction of neck-stretching, $u'$ is a curve  which is holomorphic with respect to a cylindrical almost complex structure taming $\omega_+^\infty$.
	It implies that $\omega_+^\infty(u')>0$.
	Finally, one has that 
	\[ \omega(u') = \omega_+^{\infty}(u') + \textstyle\sum_i \int_{\gamma_i}\theta > \textstyle\sum_i \int_{\gamma_i}\theta\]
	where $\lambda_i$ are the asymptotic Reeb orbits of $u'$ and $\int_{\gamma_i}\theta$ are their actions. 
\end{proof}

\begin{proof}[Proof of Lemma \ref{lem:Count=1}]
	In the proof of Lemma~\ref{lem:Count1_p_epsilon} we have shown that Maslov index~2 disks $\beta_t$ satisfy $\eta_{L_t}(\beta_t)=1$.
	
 Consider the domain $V\subset X$ which is the $\mu$-preimage of a convex open
 neighbourhood of the segment $[x_L=0,t]\subset I\subset \Pol$ connecting $x_L$
 to the point $t$ that is sufficiently close to the facet of $\Pol$, so that
 Lemma~\ref{lem:Count1_p_epsilon} applies. See Figure~\ref{fig:facet}. Then $V$
 is a Weinstein neighbourhood of $L\subset X$, which is moreover a Liouville
 neighbourhood. By Lemma~\ref{lem:neck_stretch}, one finds an almost complex
 structure $J''$ on $X$, sufficiently stretched around $V$, for which the fibres
 over the segment $[x_L,t]$ bound no holomorphic disks of non-positive Maslov
 index. Note that this stretching happens along a different domain than
 considered in the proof of Lemma~\ref{lem:Count1_p_epsilon}.
	
 So Maslov index~2 disks undergo no bubbling as we move the Lagrangian torus
 from $L$ to $L_t$ along the segment. Since $\eta_{L_t}(\beta_t)=1$, it follows
 that $\eta_L(\beta)=1$.
	\end{proof}

\begin{proof}[Proof of Proposition~\ref{prop:Psi_GCF}]
	
	Let $\{L_t\}_{t\in[0,1)}$ be the star-isotopy corresponding to a segment
	starting from the monotone torus $L = L_0$ going towards a codimension one facet
	of $\Pol$ in the normal direction. Let $\beta\in H_2(X,L;\Z)$ be class of the corresponding the Maslov index~2 disk, and $\beta_t\in
	H_2(X,L_t;\Z)$ the continuation of this class.
	
	Observe that showing that the values of $\Psi$ on fibres are not
	below the cone specified in Proposition~\ref{prop:Psi_GCF}
	is equivalent to showing that
$$\Psi(t) = \Psi(\{L_t\}) \ge \omega(\beta_t).$$ Assume for a contradiction that $\Psi(t_0) < \omega(\beta_{t_0})$. There are two posibilities.

The first possibility is that
	for some $t_0 \in [0,1)$, the (right) derivative of $\Psi(t)$ at that point
	is smaller than the derivative of $\omega(\beta_t)$. In this case concavity of $\Psi$ forces $\Psi(s) < 0$ for some 
	$t_0 < s < 1$, contradicting the positivity of $\Psi$.

	The second possibility is that  
	the left derivative of $\Psi(t)$ at some point
	is greater or equal  to the derivative of $\omega(\beta_t)$.
	In this case we would get $0 < \Psi(0) < \omega(\beta)$, contradicting 
	the fact that $L$ is orientable monotone Lagrangian and $\beta$ has Maslov 
	index 2. 
	
	To prove the desired equality, it remains to check it for some $t$, by concavity of 
	$\Psi(t)$. For $t$ close to $1$, we have that $\Psi(t) = \omega(\beta_t)$ by Lemma~\ref{lem:Count1_p_epsilon}.
\end{proof}

\begin{proof}[Proof of Theorem~\ref{thm:GCF2}]  

Clearly, $\Pol^0 \subset Sh^\st_L(X)$. To show the converse, assume that there exists a star-isotopy with flux leaving
$\Pol^0$. This would mean $\Psi(t) \le 0$ for some $t$ during this isotopy, by concavity of $\Psi$ and the fact that it tends to zero at the boundary of $\Pol^0$. It follows that $\Pol^0 = Sh^\st_L(X)$.  	
Now by Lemma~\ref{lem:Count=1} and Corollary~\ref{cor:P_monot},
$2c\cdot \P_L^\vee = \Pol^0$.
\end{proof}

Using a similar argument as above we can  show a  result stronger than
Lemma~\ref{lem:Count=1}, saying that for
any Gelfand-Cetlin toric fibre $L_x$, $x \in \Pol^0$, $\Psi$ is realized
by disks with boundary in $L_x$ corresponding to the facets that have the least area. 
This allows to get bounds on 
star-shapes relative to $L_x$, see Corollary~\ref{cor:gc_shape_nonmon} below.

\begin{lem} \label{lem:gc_count_one}
  Let $\{\beta_i\}$ be the set of classes in $H_2(X,L;\Z)$ corresponding to the 
  facets of $\Pol$, as described above. For each $x \in \Pol^0$, consider the corresponding classes
  $\{{\beta_x}_i\}$ in $H_2(X,L_x;\Z)$ as above. Then for all 
  $${\beta_x}_i \in \mathcal{B} =  \{{\beta_x}_i| \ \omega({\beta_x}_i) \le \omega({\beta_x}_j), \forall j\},$$ 
  we have that $\eta_{L_x}({\beta_x}_i)=1$.  
 \end{lem}

%\begin{proof}
%	One proceeds as in the previous proof, noting that the lowest-area operation in class $\beta_t$ on $L_t$ is witnessed by $m_{0,\beta_t}(1)=[L_t]$ using the almost complex structure from Lemma~\ref{lem:Count1_p_epsilon}. Then  $m_{0,\beta_t}(1)=[L_t]$ for any other almost complex structure by Lemma~\ref{lem:m0_invt}. Therefore $m_{0,\beta}(1)=[L]$ by the non-bubbling argument of the previous proof.
%\end{proof}

\begin{proof}
  For each $x \in \Pol^0$, we can consider a Liouville neighbourhood $V$ of $L$
  containing $L_x$ and an almost complex structure $J$ as in 
  Lemma~\ref{lem:neck_stretch}, so that there is a correspondence between 
  Maslov index 2 $J$-holomorphic disks with boundary on $L$ and on $L_x$. 
  Since, by Proposition~\ref{prop:Psi_GCF}, $\Psi(L_x) = \omega({\beta_x}_i)$
  for all ${\beta_x}_i \in \mathcal{B}$, it follows that 
  $\eta_{L_x}({\beta_x}_i)=1$.
\end{proof}

Let $\mathcal{B} = \beta_1,\ldots,\beta_k$ be the corresponding classes in
$H_2(X,L_x;\Z)$, with the same area $\Psi(L_x)$ (this means that the
ray from $x_L$ in the direction of $x$ intersects a $\codim = k$ facet).

\begin{cor} \label{cor:gc_shape_nonmon}
 It holds that
	\[ Sh^\st_{L_x}(X) \subset \bigcap_{i}  B_{\beta_i}, \]
	where $B_{\beta_i} = \{\ff \in H^1(X;\R): \Psi + \ff\cdot \del \beta_i > 
	0\}$.\qed  
\end{cor}

\begin{proof}
  It follows from Lemma~\ref{lem:gc_count_one}, 
  Theorem~\ref{th:star_shape_char}, and arguments similar to the above.
\end{proof}

\begin{ex}
Figure~\ref{fig:shcp2} shows bounds on star-shapes of various toric fibres  in $\CP^2$ coming from Corollary~\ref{cor:gc_shape_nonmon}. For non-monotone fibres, we do not know whether they are sharp.
\end{ex}

\begin{figure}[h]
	\includegraphics[]{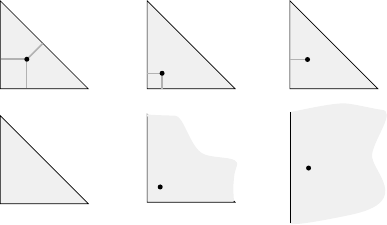}
	\caption{Toric fibres in $\CP^2$ equidistantly close to 3, 2 and 1 side(s) of the triangle (top row), and the corresponding bounds on the star-shape (bottom row).}
	\label{fig:shcp2}
\end{figure}

\subsection{Shapes in complex space} \label{subsec:ShapesCn}
Consider the product Lagrangian torus 
$$T_{\br} = S^1(r_1)\times
\cdots \times S^1(r_n) \subset \C^n.$$
Here $\br = (r_1,\dots,r_n)$, $r_i>0$ and  $S^1(r_i)=\{|z|=r_i\}\subset \C$ is the circle of radius $r_i$.

\begin{thm} \label{th:Shape_Tori_Cn}
	Identify $H^1(T_\br; \R) \cong \R^n$ 
	using the standard basis. For any $n \ge 2$, it holds that: 
	
	\begin{enumerate}[label= (\roman*)] 
		\item $Sh_{T_\br}(\C^n) = \R^n \setminus \left\{-\br +t(-1,\dots,-1) : t \ge 
		0\right\}$;   
		\item $Sh^\st_{T_\br}(\C^n) =
		\left\{(x_1,\dots,x_n) -\br : x_i\in \R\textit{ and }x_i > 0 \ \ \text{if} \ \ r_i = \underset{j=1,\dots ,n}{\min} r_j 
		\right\}.$
	\end{enumerate} 
	
\end{thm} 

Partial estimates on these shapes have been obtained earlier in \cite[Theorem~1.15, Corollary~1.17,
Corollary~1.18]{EGM16}.

\begin{ex} \label{ex:Mon_Cn_Shape}
	When $\br = (r,\dots,r)$, one has
	$$Sh^\st_{T_\br}(\C^n)=(\R_{> -r})^n.$$ 
	This is the interior of the moment polytope $(\R_{\ge 0})^n$ of the standard toric fibration on $\C^n$, translated so that
	$T_\br$ becomes the preimage of the origin. In particular, all possible star-fluxes can be achieved by the obvious isotopies among toric fibres.     
\end{ex}

\begin{figure}[h!]   
	
	\begin{center}
		
		\includegraphics[scale=0.75]{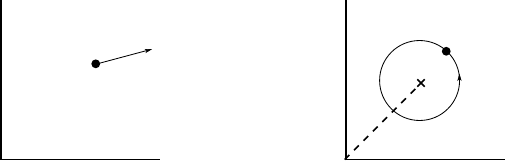}
		
		\caption{The ``segment'' isotopy $\{\t{L}_t\}$  and the ``loop'' isotopy $\{L^1_t\}$ in an almost toric fibration on $\C^2$.}
		\label{fig:C2_LagIsot} 
	\end{center} 
\end{figure}

\begin{proof}[Proof of Theorem \ref{th:Shape_Tori_Cn}~(i)]
	Consider the standard toric fibration $\C^n\to (\R_{\ge 0})^n$;
	the point $\br$ is the image of $T_\br$.  Since,
	$Sh_{T_\br}(\C^n)$ are all the same, up to translation, it is enough to prove the
	result for a specific $\br$. For convenience, we take a monotone fibre
	corresponding to $\br = (r,\dots,r)$. 
	  
	First, let us show that
	\[
	Sh_{T_\br}(\C^n) \subset \R^n \setminus \{-\br +t(-1,\dots,-1) : t \ge 0\}.
	\]
	Suppose that there is a Lagrangian
	isotopy $\{L_s\}$ starting from $L_0=T_\br$ and with flux $-\br +t(-1,\dots,-1) \in \R^n$. Then the Lagrangian torus  $L_1$ satisfies 
	$$[\omega] = - t \mu/{2} \in H^2(\C^n, L_1; \R).$$ 
	
  This contradicts the established Audin conjecture (whose proof for $\C^n$
 readily adopts from \cite{CM14}) asserting that $L_1$ bounds a Maslov index~2
 disk of positive symplectic area.
	
 Now let us restrict to the case $n = 2$. Consider the star-isotopy $$\t{L}_t =
 T_{\br + t\bv},\quad t\in[0,1],\quad \bv \in \left\{(x_1,x_2) -\br : x_i > 0
 \right\}\subset \R^2.$$ We have that $\bv + \br \in (\R_{>0})^2$, see Figure
 \ref{fig:C2_LagIsot}, and tautologically $\bv = \Flux(\{\t{L}_t\})$. 
	
	Consider  an almost toric fibration (ATF) obtained by applying a sufficiently small nodal
	trade \cite[Section~6]{Sym03} to the standard toric fibration on $\C^2$. 
	For additional background on nodal trade, see e.g.~\cite[Section~2]{Vi13}, \cite{Au07,Au09}.
	One can ensure that nodal trade does not
	modify the fibres $\t{L}_t = T_{\br + t\bv}$, for $t\in[0,1]$, see Figure
	\ref{fig:C2_LagIsot}. 
	
 Let $\{L^1_t\}$ be a Lagrangian isotopy from $T_\br$ to itself given by a loop
 in the base of our ATF starting at $\br$ and going once around the nodal
 singularity, say in the counter-clockwise direction, as in Figure
 \ref{fig:C2_LagIsot}. More generally, for each $k\in\Z$ consider the Lagrangian
 isotopy $\{L^k_t\}$ given by a loop in the base of our ATF with wrapping number $k$ around
 the node.

   The isotopy $\{L^1_t\}$ induces a monodromy 
   $$
	M^T \co H_1(T_\br;\Z)\to H_1(T_\br;\Z)$$ whose matrix is the transpose of the monodromy $M$ of the affine structure on the base around the considered loop. See
	\cite[Section~4]{Sym03} and \cite[Section~2.3]{Vi13}. Using the standard identifications
	$$H^1(T_\br; \R) \cong T_\br(\R_{>0})^2 \cong \R^2 \cong (\R^2)^*
	\cong T^*_\br(\R_{>0})^2 \cong H_1(T_\br; \R),$$
	the monodromy matrix $M$ is explicitly given by
	\[M = \begin{pmatrix}
	0 & 1\\
	-1 & 2
	\end{pmatrix}.
	\]
	
 Let $\alpha$ be a class in $H_2(\C^2, T_\br;\Z)$ corresponding to the
 vanishing cycle associated with the nodal fibre, i.e., $\alpha$ can be
 represented by a disk projecting onto a segment connecting $\br$ to the node. Note that $\del \alpha$ corresponds to the invariant
 cycle $(1,1)$ of $M$, up to sign. So when we follow the isotopy $\{L^1_t\}$,
 $\del \alpha_t$ closes up to a nunnhomologous cycle in $\C^2$, hence has zero
 area. Now, consider the cycle $\gamma \in H_1(T_\br; \Z)$
 corresponding to $(1,0)$. Following the isotopy $\{L^1_t\}$, let this cycle sweep a
 cylinder with ends on cycles $\gamma$ and $\gamma - \del \alpha$, the latter one in  class corresponding
 to $(0,-1)$. To close it up to a contractible cycle in $\C^2$, we can just add
 a representative of the class $\alpha$. Then the area of the cylinder, which is
 $\mathrm{Flux}(\{L^1_t\})\cdot\gamma$, equals $-\omega(\alpha)$. Because we chose
 $\br$ so that $L_\br$ is monotone, $\omega(\alpha) = 0$. Indeed, $\alpha$ is
 Maslov 0 and can be represented by a Lagrangian disk. Hence, 
  $\mathrm{Flux}(\{L^1_t\}) = 0$.

	Consider  the concatenation $\{\t{L}_t\}*\{L^k_t\}$. This is a Lagrangian isotopy which first follows the ``loop'' $\{L^k_t\}$ and then ``segment''  $\{\t{L}_t\}$.
	Since $\{L^k_t\}$ has zero flux, one has that 
	$$\Flux(\{\t{L}_t\}*\{L^k_t\}) = M^k\bv.$$ 
	Recall that the vector $\bv$ may be freely chosen from the domain
	$$Q =\left\{(x_1,x_2) -\br : x_i > 0 
	\right\}.$$ 
	
	 Again, we take $\br = (r,\dots,r)$, so it is invariant under $M$.  
	We have that  
	\[ \bigcup_{k\in \Z} M^k Q = \R^2 \setminus \{-\br +t(-1,-1): t \ge   0 \} \subset Sh_{T_\br}(\C^2). \]
	Indeed, this follows by noting that
	\[M^k = \begin{pmatrix}
	1 - k & k\\
	-k & k + 1
	\end{pmatrix},
	\]   
	and that the columns $[1 \mp k, \mp k] \to |k|(-1,-1)$ as $k \to \pm \infty$.
	This completes the proof of Theorem \ref{th:Shape_Tori_Cn}~(i) for $n=2$.  
	
	For a general $n \ge 2$, consider the splitting $\C^n = \C^2 \times \C^{n-2}$,
and the SYZ fibration which is the product of the previously considered ATF on the $\C^2$-factor with the standard toric fibration on the $\C^{n-2}$-factor. 
	
	There is a loop in the base of this SYZ fibration starting at $\br$
	whose monodromy is the block matrix $$\begin{pmatrix} M & 0\\
	0 & \id_{n-2}
	\end{pmatrix}$$ 
	where $M$ appears above. Arguing as before, we conclude that 
	\begin{equation} \label{eq:Subs_Sh_Cn}
	\R^n \setminus \{-\br +(-t,-t, x_3, \dots, x_n) : t \ge 
	0,\ x_k\le 0\textit{ for } k\ge 3\} \subset Sh_{T_\br}(\C^n).  
	\end{equation} 
	The argument can be applied to any pair of coordinates instead of the first two ones. The union of sets as in
	\eqref{eq:Subs_Sh_Cn} arising this way covers the whole claimed shape:
	$$ 
	\begin{array}{l}
	\bigcup_{i,j} \left(\R^n \setminus \{-\br +(x_1, \dots, x_n) : x_i=x_j \le 0 ,\  x_k \le    0\textit{ for } k\neq i,j  \} \right) \\
	= \R^n \setminus \bigcap_{i,j} \{-\br +(x_1, \dots, x_n): x_i=x_j \le 0,\ x_k \le  0\textit{ for  } k\neq i,j \}
	\\
	= \R^n \setminus \{-\br +t(-1,\dots,-1) \in \R^n: t \ge    0\}.   
	\end{array}
	$$
	The result follows.
\end{proof}

\begin{proof}[Proof of Theorem \ref{th:Shape_Tori_Cn} (ii)]
	The inclusion of the star-shape into the desired set
	\begin{equation} \label{eq:Inclusion}
	Sh^\st_{T_\br}(\C^n) \subset
	\left\{(x_1,\dots,x_n) -\br : x_i\in \R\textit{ and }x_i > 0 \ \ \text{if} \ \ r_i = \underset{j=1,\dots ,n}{\min} r_j 
	\right\}
	\end{equation} 
	follows from Theorem~\ref{th:star_shape_char}. 
	
	\begin{figure}[h!]   
		
		\begin{center}
			
			\includegraphics[scale=0.75]{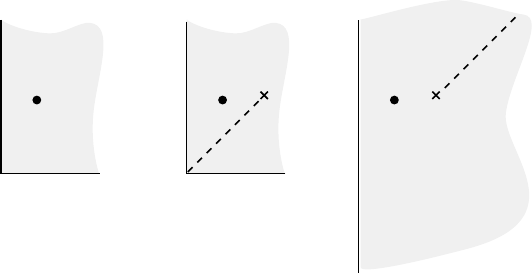}

			\caption{Left: the standard toric  fibration on $\C^2$. Middle and right: two diagrams representing the same almost toric fibration with one nodal fibre, with two different ways of making a cut.}
			\label{fig:2DNodaltrade} 
		\end{center} 
	\end{figure} 
	
	To prove the reverse inclusion, we again start with  $n=2$.  The case $r_1 = r_2$ is clear, since the shape in question
	$\{(x_1,x_2) -\br : x_i > 0 \}$ is realised by isotopies in the standard toric fibration. Assuming $r_1 < r_2$, the set which we must prove to coincide with star-shape is
	$$\left\{(x_1,x_2) -\br : x_1 > 0 \right\}.$$ To compare, isotopies  within the standard toric fibration achieve flux of the form $(x_1,x_2) -\br$ where both $x_1,
	x_2 >0$. Figure \ref{fig:2DNodaltrade} shows star-isotopies that achieve the remaining flux, i.e.~of the form $(x_1,x_2)
	-\br$ where $x_1 >0$ and $x_2 \le 0$. This completes the proof for $n = 2$.
	
	Unlike  the proof  of Theorem \ref{th:Shape_Tori_Cn}~(i), in higher dimensions it will not be enough to consider SYZ fibrations which almost look like the product of the 4-dimensional ATF with the standard toric fibration; we must consider a larger class of SYZ fibrations that exist on $\C^n$.
	
	Let us discuss the case $n = 3$; the details in the general case are analogous. The monotone case $r_1=r_2=r_3$ is again clear.
	Now, the case $r_1 = r_2 < r_3$ is precisely the one when considering the SYZ fibration from Theorem \ref{th:Shape_Tori_Cn}~(i)
	\emph{is} sufficient. Looking at the product of the 4-dimensional ATF on  $\C^2$
	with the standard toric fibration on $\C$, we  obtain 
	any flux of the form $(x_1,x_2,x_3) -\br$, where $x_1,x_2 >0$
	and $x_3 \le 0$, see Figure \ref{fig:3DNodaltrade2}. The remaining flux is realised by isotopies in the standard toric fibration, and we conclude that  
	$$Sh^\st_{T_\br}(\C^3) =
	\left\{(x_1,x_2,x_3) -\br: x_1,x_2 > 0  \right\},$$
	as desired.
	\begin{figure}[h!]   
		
		\begin{center}
			
			\includegraphics[scale=0.75]{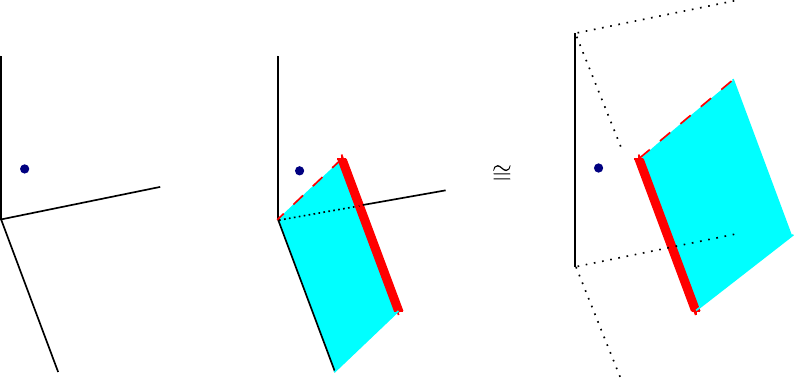}
			
			\caption{Left: the standard toric fibration on $\C^3$. Middle and right: SYZ fibration which is the product of the 4-dimensional ATF with the standard toric fibration on $\C$. Middle and right show the same fibration with two different ways of making a cut.}
			\label{fig:3DNodaltrade2} 
		\end{center} 
	\end{figure} 
	
	We move to the most complicated  case $r_1 < r_2 \le r_3$. We must show that 
	$$Sh^\st_{T_\br}(\C^3) =
	\left\{(x_1,x_2,x_3) -\br : x_1 > 0  \right\},$$
	 and the constructions we discussed so far miss out the subset where $x_2,x_3
 \le 0$. To see the remaining flux, consider an SYZ fibration (see \cite[Example~3.3.1]{Au09}), depending on $c
 >0$, whose fibres are parametrised by $(\fr,\mu_1,\mu_2)$ and are given by:
	
	\begin{multline}
	T^c_{\fr,\mu_1,\mu_2} = \{(z_1,z_2,z_3) :z_i\in\C,\ |z_1z_2z_3 - c| = \fr ,
	\\ 
	\pi(|z_1|^2 - |z_2|^2) = \mu_1,\ \pi(|z_1|^2 - |z_3|^2) = \mu_2 \}\subset \C^3,
	\end{multline}
	 
  see Figure~\ref{fig:3DNodaltrade}. We point out that $(\fr,\mu_1,\mu_2)$ are
  not locally affine coordinates on the base, although $(\mu_1,\mu_2)$ are a
  part of locally affine coordinates. 
	
	A non-singular torus $T^c_{\fr,\mu_1,\mu_2}$ can be understood as follows. Consider the map $f: \C^3
	\to \C$, $(z_1,z_2,z_3) \mapsto z_1z_2z_3$ whose fibres are invariant under
	the $T^2$-action 
	$$(e^{i \theta_1},e^{i \theta_2})\cdot (z_1,z_2,z_3) = (e^{i
		(\theta_1 + \theta_2)}z_1,e^{ -i \theta_1}z_2,e^{-i\theta_2}z_3).$$ Its moment map is
	$(\mu_1,\mu_2) = (\pi(|z_1|^2 - |z_2|^2), \pi(|z_1|^2 - |z_3|^2))$. Then
	$T^c_{\fr,\mu_1,\mu_2}$ is the parallel transport  of an orbit of this $T^2$-action 
	with respect to
	the symplectic fibration $f$, over the radius-$\fr$ circle centred at $c$.
	
	Because $f$ has a singular fibre over $0$, some Lagrangians
	$T^c_{c,\mu_1,\mu_2}$  will be singular. This happens precisely when 
	\begin{eqnarray} \label{eq:condDegFibre}
	\mu_1 = 0, \ \mu_2 <
	0, \quad  \text{or} \quad \mu_1 < 0, \ \mu_2 = 0, \quad \text{or}  \quad  \mu_1 = \mu_2 > 0, 
	\nonumber
	\\ \text{or} \quad \mu_1 = \mu_2 = 0\quad \text{(the most degenerate case)}. 
	\end{eqnarray}
	All other fibres  are smooth Lagrangian tori, for $\fr >0$. Observe that our SYZ fibration extends over $\fr = 0$, where it becomes a singular Lagrangian $T^2$-fibration on $z_1z_2z_3 = c$. 
	Also recall that this construction depends on the parameter $c$,
	and the limiting case $c=0$ is actually the standard toric fibration on $\C^3$.

	\begin{figure}[h!]   
		
		\begin{center}
			
			\includegraphics[scale=0.6]{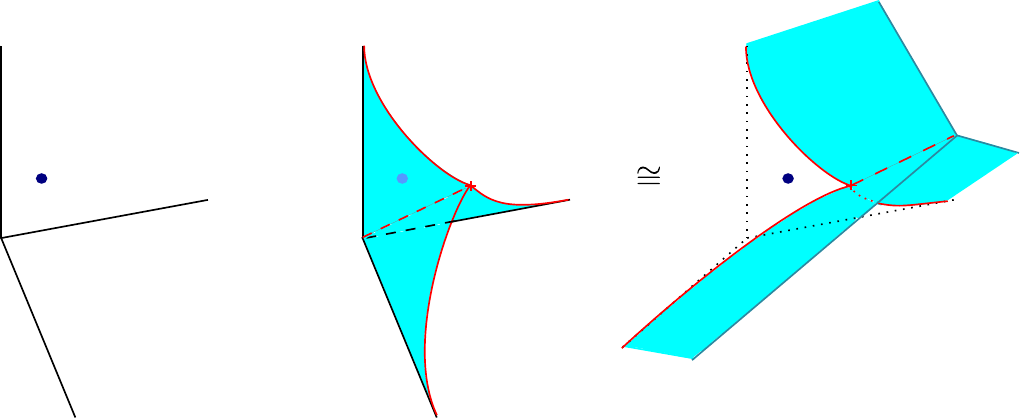}
			
			\caption{Left: the standard  toric fibration on $\C^3$.
				Middle and right: the SYZ fibration  described by
				$T^c_{\fr,\mu_1,\mu_2}$, for some fixed $c$. The red curves show the discriminant locus.
				The most singular fibre $T^c_{c,0,0}$ is marked by the node. The node moves with $c$ in the direction of the dashed line. Middle and right show two ways of making cuts, which are the shaded surfaces.
			}\label{fig:3DNodaltrade} 
		\end{center} 
	\end{figure} 
	
	The complement of $(\R_{\ge 0})^3$ to the discriminant locus of the constructed fibration carries a natural affine structure. Since this affine structure has monodromy around the discriminant locus, it is not globally isomorphic to one induced from the standard affine structure on $(\R_{\ge 0})^3$. However, it is isomorphic to the standard one in the
	complement of a codimension-one set  called a
	\emph{cut}. There are various ways of making a cut, and two of them are shown in Figure
	\ref{fig:3DNodaltrade}. 
	
 Finally, let us discuss the effect of changing the parameter $c$. It
 corresponds to sliding the dashed segment in Figure~\ref{fig:3DNodaltrade};
 this is a higher-dimensional version of nodal slide. Intuitively, by taking $c$
 to be sufficiently large (i.e.~sliding the node sufficiently far towards
 infinity), one can see the existence of a star-isotopy from $T_\br$ to with any
 flux of the form $(x_1,x_2,x_3) - \br \in \R^3$ such that $x_1>0$.
	
	% Explicitly we start with a Lagrangian torus fibre $(\rho_1 e^{i\phi_1}, 
	%  \rho_2 e^{i\phi_2}, \rho_3 e^{i\phi_3}) = T^0(\fr,\mu_1,\mu_2)$, where  
	% $(e^{i\phi_1}, e^{i\phi_2}, e^{i\phi_3}) \in (S^1)^3$, $r =
	% \rho_1\rho_2\rho_3$, $\mu_1 = \pi(\rho_1^2 - \rho_2^2)$, $\mu_2 =
	% \pi(\rho_1^2 - \rho_3^2)$. Our assumption $r_1 < r_2 \le r_3$ is equivalent
	% to $\rho_1 < \rho_2 \le \rho_3$. Now, we want to get an
	% isotopy for which the flux corresponds to $(x_1,x_2,x_3) -\br \in \R^3$. We
	% start getting a candidate for the end point of the isotopy of the form
	% $T^c(\fr,\mu_1,\mu_2)$, with $\mu_1 = x_1 - x_2$ and $\mu_2 = x_1 - x_3$.
	% It is shown in \cite[Section~3.3]{Au09} that the tori $T^c(\fr,\mu_1,\mu_2)$
	% with $c > \fr$ only bounds one family of holomorphic disk projecting 
	% injectively under $f(z_1,z_2,z_3) = z_1z_2z_3$ to the circle of radius $\fr$ 
	% centred at $c$. So we choose a pair $(c,r)$ depending on $(x_1,\mu_1,\mu_2)$
	% for which $c > r$ and the area of the above mentioned holomorphic disk is 
	% $x_1$. 
	
	To construct this more explicitly,  start with  the initial torus 
	$T_\br$, $\br =(r_1,r_2,r_3)$, $r_1 < r_2 \le r_3$, and observe that
	$$
	T_\br=
	T^0_{\fr,\mu_1,\mu_2},\quad \fr = \left(\frac{r_1r_2r_3}{\pi^3}\right)^{1/2},\quad 
	\mu_1 = r_1 - r_2\quad \mu_2 = r_1 -
	r_3.
	$$
	Note that $\mu_2 \le \mu_1 < 0$. Denote by $\beta \in H_2(\C^3,T_\br)$
	the class of  holomorphic disks with area $r_1$. We will build a
	star-isotopy of the form $T^{c(t)}_{\fr(t),\mu_1(t),\mu_2(t)}$ with flux
	$t((x_1,x_2,x_3) - \br)$, $x_1 > 0$. Note that in order for $T^{c(t)}_{\fr(t),\mu_1(t),\mu_2(t)}$ to be a
	star-isotopy, the relative class 
	$$\beta(t) \in
	H_2\left(\C^3,T^{c(t)}_{\fr(t),\mu_1(t),\mu_2(t)}\right),$$ which is the continuous
	extension of $\beta$, must satisfy $\omega (\beta(t)) = r_1 +
	t(x_1 - r_1)$. We arrange: 
	
	\begin{enumerate}[label= (\Roman*)]
		\item  $\mu_1(t) = r_1 - r_2 + t[x_1 -
		x_2 - (r_1 - r_2)]$; 
		
		\item$\mu_2(t) = r_1 - r_3 + t[x_1 - x_3 - (r_1 - r_3)]$;
		
		\item  $c(t) = \psi(t)c_0$, where $c_0$ is  a large real
		number and $\psi(t)$ is a non-decreasing smooth cutoff function: it satisfies $\psi(0) = 0$ and identically equals $1$
		for $t \ge \epsilon$, where  $\epsilon$ is
		sufficiently small;
		
		\item $\fr(t)$ is chosen so that $\omega (\beta(t)) = r_1 +
		t(x_1 - r_1)$. 
	\end{enumerate}
	
	We first set (I), (II) and choose
	$\epsilon$ small enough to ensure  $\mu_j(t) < 0$ for $t \in [0,
	\epsilon], \ j=1,2$. We need now to set the endpoint of our isotopy by choosing 
	$c_0$ and the corresponding $\fr(1)< c_0$. We can make the area of 
	the corresponding $\beta$ class in a torus of the form 
	$T^{c_0}_{\tilde{\fr},\mu_1(1),\mu_2(1)}$ as big as we want, in particular 
	bigger than $x_1$, by taking $c_0$ sufficiently
	large and then $\tilde{\fr}<c_0$ sufficiently close to $c_0$. Taking such
	$c_0$, we may take $\fr(1)< c_0$, so that for $T^{c_0}_{\fr(1),\mu_1(1),\mu_2(1)}$,
	we have $\omega(\beta(1)) = x_1$. 
	
	Now we choose our cutoff function $\psi(t)$, setting item (III) of our desired 
	list. Since $x_1 > 0$, the expression $r_1+t(x_1-r_1)$ is non-negative, so we can find $\fr(t)$ 
	to ensure we have (IV).
	
	Our setup guarantees that $T^{c(t)}_{\fr(t),\mu_1(t),\mu_2(t)}$ is a smooth 
	torus for all $t \in [0,1]$. Indeed, $T^{c(t_0)}_{\fr(t_0),\mu_1(t_0),\mu_2(t_0)}$ 
	could be non-singular only at
 the moment $t_0$ when $c(t_0) = \fr(t_0)$, but our choice of $\psi(t)$
 ensures that $t_0 < \epsilon$. This implies that $\mu_1(t_0) < 0$,
 $\mu_2(t_0) < 0$ and, hence, $T^{c(t_0)}_{\fr(t_0),\mu_1(t_0),\mu_2(t_0)}$ is
 smooth, recall \eqref{eq:condDegFibre}. This finishes the proof of Theorem
 \ref{th:Shape_Tori_Cn} $(ii)$ for $n = 3$. Conditions (I),(II) and (IV) ensures
 we have a star-isotopy.
	
	The situation in higher dimensions is very similar to the $n
	= 3$ case. Assume that
	$r_1 = \cdots = r_{k+1} < r_{k+2} \le \cdots \le r_n$. The monotone case
	$n = k+1$ is trivial, so we assume $n > k+1$. 
	
	Let us split $\C^n$ as $\C^n = \C^k \times
	\C^{n-k}$, take the standard toric
	fibration in the $\C^k$-factor and the following SYZ fibration 
	in the $\C^{n-k}$ factor. Its fibres 
	$T^c_{\fr,\mu_1, \dots, \mu_{n-k-1}}$ are defined analogously to the previous construction, using the auxiliary symplectic fibration
	$f(z_1,\dots,z_{n-k}) = z_1\dots z_{n-k}$, the corresponding
	$T^{n-k-1}$-action, and the similar parallel transport. One again uses $c$ as a parameter of the construction. 
	
	Given $(x_1,\dots,x_n) \in
	\R^n$ with $x_1,\dots,x_{k+1} > 0$, we can construct a star-isotopy  from  $T_\br =
	T_{(r_1,\dots,r_k)}\times T_{(r_{k+1},\dots,r_{n})}$ to a torus of the form
	$$T_{(x_1,\dots,x_k)} \times T^c_{\fr,\mu_1, \dots, \mu_{n-k-1}},$$ such that the flux of this isotopy
	equals $(x_1,\dots,x_n) - \br$. Indeed, using the condition that $x_1,\dots,x_{k} > 0$, first consider a
	star-isotopy from $T_{(r_1,\dots,r_k)}$ to $T_{(x_1,\dots,x_k)}$ in the $\C^k$-factor
	through toric fibres. Using $x_{k+1} > 0$, there is now a star-isotopy from
	$T_{(r_{k+1},\dots,r_{n})}$ to $T^c_{\fr,\mu_1, \dots, \mu_{n-k-1}}$ in the
	$\C^{n-k}$-factor, analogously to what we did in the case of $\C^3$. 
\end{proof}

\begin{thm} \label{th:Shape_Chekanov_Cn}
	For $n\ge 2$, let $\Theta^{n-1}(r) \subset \C^n$  be the Chekanov torus introduced in \cite{ChSch10},  bounding a Maslov index 2 disk
	of symplectic area $r$. The following holds.
	
	\begin{enumerate}[label= (\roman*)] 
		\item $Sh_{\Theta^{n-1}(r)}(\C^n) = \R^n \setminus \{-\br +t(-1,\dots,-1): t \ge  0\}$ where $\br = (r,\dots,r)$;    
		\item $Sh^\st_{\Theta^{n-1}(r)}(\C^n)=\{(x_1,\dots,x_n) - (r,\dots,r):x_1>0,\ x_i\in\R\}$,  the half-space bounded by the hyperplane $x_1=-r$. 
	\end{enumerate} 
	
\end{thm}

\begin{proof} 
	The tori $\Theta^{n-1}(r)$ are Hamiltonian isotopic to the tori of the
	form $T^c_{\fr,0, \dots, 0}$ described above, provided that $\fr < c$. It is shown in
	\cite[Section~3.3]{Au09} that there is a unique family of holomorphic disks with
	boundary on $T^c_{\fr,0, \dots, 0}$, and each disk projects via
	$f(z_1, \dots, z_n) = z_1\cdots z_n$ isomorphically onto the disk of radius $\fr$
	centred at $c$. (The values of $c,\fr$ are such that these disks have area $r$.)
	Note that $T^c_{\fr,0, \dots, 0}$ is Lagrangian isotopic with
	\emph{zero flux} to the product torus $$T^0_{(r^{1/2}/\pi)^n,0, \dots, 0}= S^1(r)
	\times \cdots \times S^1(r),$$ 
	Now (i) follows from Theorem
	\ref{th:Shape_Tori_Cn}~(i). For (ii), observe that the proof of
	Theorem \ref{th:Shape_Tori_Cn}~(ii) achieves star-flux of any desired form from the statement. The reverse inclusion follows from  Theorem~\ref{th:star_shape_char}.   
\end{proof}

We note that   Theorem \ref{th:Shape_Chekanov_Cn} improves the computation in \cite[Theorem~1.19]{EGM16}.

\subsection{Wild shapes of toric manifolds} \label{subsec:Unbounded shapes}
In contrast to star-shapes, the (non-star) shapes of compact toric manifolds behave wildly. The idea is that in toric manifolds, there exist loops of embedded Lagrangian tori with various monodromies, and  these monodromies together generate big subgroups of $SL(n,\Z)$.  We shall illustrate the phenomenon by 
looking at $\CP^2$. 

What is perhaps more surprising,   tori $T$ in compact toric varieties usually possess \emph{unbounded product neighbourhoods} $T\times Q$. Figure~\ref{fig:Sausage} from the introduction shows an example for $\CP^2$, where $Q\subset\R^2$ is the unbounded open set shown on the left. Such products cannot be convex or Liouville with respect to the zero-section, by Theorem~\ref{th:nbhood_comp}.

\begin{cor}
	The symplectic neighbourhood $T\times Q$ does not admit a Liouville structure making $T\times \{\mathrm{pt}\}$ exact, where $\mathrm{pt}$ is marked in  Figure~\ref{fig:Sausage} (it is sent to the monotone fibre under the above embedding).\qed
\end{cor}

We continue to focus on $\CP^2$. The monotone Clifford torus  $T$ is the fibre 
corresponding to the barycentre of the standard moment triangle $\Pol$ of $\CP^2$. Let us apply 
nodal trades  to each of the three vertices of $\Pol$. Let  $\Pol^0$
be the interior of $\Pol$. For the cuts shown in Figure~\ref{fig:Sausage}, the monodromies around the nodes  are respectively given by:
\[M_1 = \begin{pmatrix}
2 & -1\\
1 &  0
\end{pmatrix}, \ \ 
M_2 = \begin{pmatrix}
3 &  1\\
-4 & -1
\end{pmatrix}, \ \ 
M_3 = \begin{pmatrix}
3 &  4\\
-1 & -1
\end{pmatrix}.
\]

Consider the subgroup $G_{\CP^2}< SL(2,\Z)$, $G_{\CP^2} = \langle M_1, M_2, M_3 \rangle$, generated by the $M_i$. 
Revisiting the proof of Theorem \ref{th:Shape_Tori_Cn} $(i)$ for 
$n=2$, one concludes that $Sh_T(\CP^2)$ contains the orbit of $\Pol^0$ under the total monodromy group action: 
$$G_{\CP^2}\cdot \Pol^0= \{g(x): g \in
G_{\CP^2}, \ x \in \Pol^0  \}\subset \R^2.$$

First, let us check that this orbit is unbounded. If $Q\subset\R^2$ is the domain shown in  Figure~\ref{fig:Sausage}, by
consecutively applying the monodromies one sees that 
\[ Q \subset \Pol^0 \cup \bigcup_{k=1}^{\infty} M_k\cdots M_1 \cdot \Pol^0 \subset G_{\CP^2}, \ x \in \Pol^0, \]
where the subscripts are taken modulo 3: $M_i = M_j$ if and only if $i \equiv j \mod 3$.  Next comes  a question we were not able to answer.

\begin{qu} \label{qu:Dense}
	Is the orbit $G_{\CP^2} \cdot \Pol^0$ dense in $\R^2$?
\end{qu}

Although we do not have an answer, it will be useful to pursue this question. To this end, one computes

\[ M_3 M_2 M_1 = \begin{pmatrix}
1 & -9\\
0 &  1
\end{pmatrix}.
\]
Conjugating by 
\[ P = \begin{pmatrix}
0 & 1\\
-1 & 1
\end{pmatrix}
\]
we get: 

\[ P M_1 P^{-1} = \begin{pmatrix}
1 &  1\\
0 &  1
\end{pmatrix},\ 
P M_2 P^{-1} = \begin{pmatrix}
-5 & 4\\
-9 & 7 
\end{pmatrix},\ 
P M_3 M_2 M_1 P^{-1} = \begin{pmatrix}
1 &  0\\
9 &  1
\end{pmatrix}.
\]
So $PG_{\CP^2}P^{-1}$ is generated by the three matrices above. In particular, $G_{\CP^2}$  contains a subgroup isomorphic to $G_9$ where:

\begin{equation} \label{eq:Gk}
G_k = \left\langle t = \begin{pmatrix}
1 &  1\\
0 &  1
\end{pmatrix} , \  
h_k = \begin{pmatrix}
1 &  0\\
k &  1
\end{pmatrix} \right\rangle.
\end{equation}

Let $G$ be a locally compact Lie group with  the right-invariant Haar measure $\mu$. A discrete subgroup $\Gamma$ of $G$  is called a \emph{lattice} \cite[Section~1.5~b]{FeKa02} if
the induced measure on
$G/\Gamma$ has finite volume. The Haar measure on $PSL(2,\R)$ is induced from the hyperbolic metric on
$\bH^2 \cong PSL(2,\R)/PSO(2,\R)$, so  $\Gamma<PSL(2,\R)$ is a lattice if and only if the induced action of $\Gamma$ on $\bH^2$ produces the quotient
$\bH^2/\Gamma$ of finite area. 

Let us view $S^1$ as the projectivisation of the plane: $S^1 = \mathrm{proj} (\R^2) \cong PSL(2,\R)/U$, where $U < PSL(2,\R)$ is the subgroup
of upper-triangular matrices.
Howe-Moore ergodicity theorem  implies that the action of any lattice $\Gamma < PSL(2,\R)$ 
on $\mathrm{proj}(\R^2)$ is ergodic; see \cite[Theorem~3.3.1, Corollary~3.3.2, Proposition~4.1.1]{FeKa02}.

Now suppose $D\subset \R^2$ is any open subset containing the origin, and $\Gamma<PSL(2,\R)$ is a lattice.
It quickly follows that the orbit $\Gamma\cdot D$  is dense in $\R^2$.

\begin{prp} \label{prp:Lattice}
	The subgroup $G_k < PSL(2,\R)$ is a lattice if and only if $0 < |k| \le 4$.
\end{prp}

\begin{proof}
	For $k = 0$, one has
	\[ \bH^2/G_0 = \{ (x,y) \in \bH^2: -1/2 \le x \le 1/2 \}/ 
	(-1/2,y) \sim (1/2,y)\] which is of infinite volume.
	We claim that for $k \ne 0$, the fundamental domain of the action of $G_k$ on $\bH^2$ is: 
	\[ D_{G_k} =  \{ (x,y) \in \bH^2: -1/2 \le x \le 1/2,\ \|(x \pm 1/k , y)\| \ge 1/|k| \}. \]
	We are using the upper half-plane model for the hyperbolic plane.
	Indeed, since $t$ \eqref{eq:Gk} acts by integer translation in 
	the coordinate $x\in\bH^2$, we may assume $-1/2 \le x \le 1/2$. 
	Next, the $y$-coordinate of $h_k^n\cdot(x,y)$ equals 
	\[ \frac{y}{(nkx + 1)^2 + (nky)^2}. \] 
	So $(x,y)$ is the representative of its $\langle h_k \rangle$-orbit
	with the largest value of $y$ if and only if $(nkx + 1)^2 +
	(nky)^2 \ge 1$ for all $n$, equivalently, if and only if $ \|(x \pm 1/k , y)\| \ge 1/|k|.$ It explains that $D_{G_k}$ is a fundamental domain.
	Finally, $D_{G_k} \subset \bH^2$ has finite volume if and only if $0 < |k| \le 
	4$.
\end{proof}

As we have seen above, $G_{\CP^2}$ is generated by $G_9$ and $\left(\begin{smallmatrix}
-5 & 4\\
-9 & 7 
\end{smallmatrix}\right)$. We do not know whether $G_{\CP^2}$ is a lattice, so we could not answer Question~\ref{qu:Dense}.
However, we can now answer the analogous question  for some other symplectic 4-manifolds.

\begin{cor} \label{cor:DenseOrbits}
	Let $X=Bl_{k}\CP^2$ be the blowup of $\CP^2$ at $k\ge 5$ points, with any (not necessarily monotone) symplectic form.
	Let $L$ a fibre of an almost toric fibration on  $X$ whose base is diffeomorphic to a disk (e.g.~a fibre of a toric fibration). Then 
	$Sh_L(X)$ is dense in $\R^2$.
\end{cor}

\begin{rmk}
	A symplectic manifold admitting an almost toric fibration with base homeomorphic to a 
	disk is  diffeomorphic to $Bl_k\CP^2$ or $\PxP$ by \cite{SyLe10}.
\end{rmk}

\begin{proof}
	Consider an almost toric fibration from  the statement, and let $\Pol$ be its base. Performing small 
	nodal trades when necessary, we may assume that the preimage of the boundary of
	$\Pol$ is a smooth elliptic curve representing the anticanonical class  $-K_X$  
	\cite[Proposition~8.2]{Sym03}.
	Consider the loop in $\Pol$ which goes once around the boundary $\del \Pol$ sufficiently closely to it, and encloses all singularities of the almost toric fibration. The affine monodromy around this loop is conjugate to
	$$
	\begin{pmatrix} 1 & 0\\ k & 1 \end{pmatrix},
	$$
	because it has an eigenvector given by the fibre cycle of the boundary elliptic curve. Furthermore, it can be shown that $k = K_X^2$; see e.g~\cite{SmG15}.
	Following the proof of Theorem \ref{th:Shape_Tori_Cn}~(i), one argues that 
	$$G_{K_X^2} \cdot
	\Pol^0 \subset
	Sh_L(X).$$ If $0 \ne |K_X^2| \le 4$, the result follows from Proposition~\ref{prp:Lattice} and
	the Howe-Moore theorem, in particular it hods for $Bl_5\CP^2$.

	\begin{figure}[h!]   
		
		\begin{center}
			
			\includegraphics{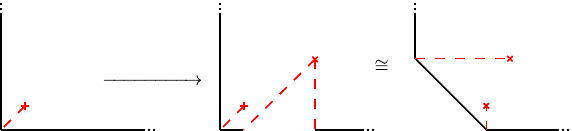}

			\caption{Almost toric blowup of the 
				$A^1_{k-1}$ ATF on $Bl_{k-1}\CP^2$ (left) to the $A^1_{k}$ ATF on 
				$Bl_k\CP^2$ (middle and right). The complement of a neighbourhood of a cut in the
				$A^1_{k}$ ATF (the triangular cut
				in the middle diagram) embeds into the $A^1_{k-1}$ ATF on $Bl_k\CP^2$ (left diagram).}
			\label{fig:ATBlup} 
			
		\end{center} 
	\end{figure}
	In general, denote by $G_X<SL(2,\Z)$ the group generated by all monodromies of an almost toric fibration on $X$ as above.
	We claim that $G_{Bl_m\CP^2}$ is a subgroup of $G_{Bl_k\CP^2}$ for $m \le k$.
	This implies that the result holds for $Bl_k\CP^2$, $k \ge 5$.  
	Indeed, starting with an ATF $A^0_{k}$ on $(Bl_k\CP^2, \omega)$, one  deduces  from
	\cite[Theorem~6.1]{SyLe10} that there is a different ATF $A^1_{k}$ on
	$(Bl_k\CP^2, \omega')$ satisfying: $A^1_{k}$ is obtained from an ATF $A^1_{k-1}$ on $Bl_{k-1}\CP^2$ via
	almost toric blowup (\cite[Section~4.2]{SyLe10}, see also
	\cite[Example~4.16]{Zu03}, \cite[Section~5.4]{Sym03}, and Figure
	\ref{fig:ATBlup}); and $A^1_{k}$ is obtained from the ATF $A^0_{k}$ by deforming
	$\omega$ to $\omega'$ and applying nodal slides. In particular, they have the
	same groups of monodromies. By disallowing to travel around the distinguished nodal fibre coming from
	the almost toric blowup, one gets  an embedding of the monodromy group of $A^1_{k-1}$ into one of $A^1_{k}$, which is the same as for the initial ATF $A^0_{k}$. 
  
\end{proof}

\section{Space of Lagrangian tori in $\CP^2$} \label{sec:LagModuli_CP2}
Given a symplectic manifold $X$, the space of all (not necessarily monotone)
Lagrangian embeddings of a torus into $X$ is usually non-Hausdorff. Despite the
indications that this space should be in some way related to the rigid-analytic
mirror of $X$ (if it exists), we do not seem to have a rigorous understanding of
this connection so far. More basically, there is a lack of examples in the
literature computing these spaces. We shall study this question for $\CP^2$.
Recall that all symplectomorphisms of $\CP^2$ are Hamiltonian.

In \cite{Vi13,Vi14}, it is shown that monotone tori in $\CP^2$ are associated with Markov 
triples. We recall that a \emph{Markov triple} $(a,b,c)$ is a triple of positive integers satisfying the Markov equation:
\begin{equation} \label{eq:Markov}
a^2 + b^2 + c^2 = 3abc.
\end{equation}
All Markov triples are assumed to be unordered.
They form the vertices of the infinite Markov tree with root $(1,1,1)$, whose beginning is shown below.

\begin{table}
	\begin{center}
$${
	\xymatrixcolsep{-1em}
	\xymatrixrowsep{1em}
	\small
	\xymatrix{
		&&&(1,1,1)\ar@{-}[d]&&&\\
		&&&(1,1,2)\ar@{-}[d]&&&\\
		&&&(1,2,5)\ar@{-}[dll]\ar@{-}[drr]&&&\\
		&(2,5,29)\ar@{-}[dl]\ar@{-}[dr]&&&&(1,5,13)\ar@{-}[dl]\ar@{-}[dr]&\\
		(5,29,433)&&(2,29,169)&&(5,13,194)&&(1,13,34)\\
	}
}
$$
\caption*{The Markov tree.}
\end{center}
\end{table}

Two Markov triples  connected by an edge are related by  mutation of the form:
\begin{equation} \label{eq:Markov_mutation}
(a,b,c) \to (a,b, 3ab - c).
\end{equation}
Besides the univalent vertex $(1,1,1)$ and the bivalent vertex $(1,1,2)$,
all vertices of the tree are trivalent.

There is an almost toric fibration (ATF) on $\CP^2$ corresponding to each Markov
triple $(a,b,c)$, constructed in \cite{Vi13,Vi14}. Its base can be represented
by a triangle (with cuts) whose sides have affine lengths $(a^2,b^2,c^2)$. Imposing restrictions on the cuts, one get that the base diagram representing an
ATF containing the monotone fibre $T(a^2,b^2,c^2)$, uniquely determine the
above mentioned triangle, up to $SL(2,\Z)$, c.f.~\cite{Vi13,Vi14}. Slightly
abusing terminology, we call it the \emph{moment triangle} associated to
$T(a^2,b^2,c^2)$. We shall call an $(a,b,c)$-ATF any ATF containing $T(a^2,b^2,c^2)$
as a monotone fibre.

From now, we maintain the following agreement: the nodes of these fibrations are
assumed to be slided arbitrarily close to the vertices of the \emph{moment triangle}. So
when we speak of \emph{a fibre} of the $(a,b,c)$-ATF, we always mean the
preimage of a point in the base triangle with respect to an ATF whose nodes are
closer to the vertices than the point in question. More formally, the fibres of
the $(a,b,c)$-ATF are the regular fibres of the corresponding fibration on the
weighted projective space, pulled back to $\CP^2$ via smoothing (which defines
them up to Hamiltonian isotopy).

By \cite{Vi14},  two monotone tori corresponding to different Markov triples are not Hamiltonian isotopic to each other. Our aim is to study all (not necessarily monotone) fibres of all the ATFs together, modulo symplectomorphisms of $\CP^2$. They form a non-Hausdorff topological space:
$$
\H=\{T\subset \CP^2 \textit{ a Lag.~torus fibre of the }(a,b,c)\text{-ATF}\textit{ for some Markov triple}\}/\sim
$$
where $T_1\sim T_2$ is there exists a symplectomorphism of $\CP^2$ taking $T_1$ to $T_2$.
It is a plausible but hard conjecture that \emph{any} Lagrangian torus in $\CP^2$ is actually isotopic to some $(a,b,c)$-ATF fibre. If this is true, then $\H$ is the space of all Lagrangian tori in $\CP^2$ up to symplectomorphism.

We shall study $\H$
with the help of the numerical invariant arising from the remark made in Section~\ref{subsec:sh_bound_dimf}: 

\begin{equation}
  \Xi_2(L) = \sum_{\substack{\beta\ :\ \mu(\beta)=2,\\ \omega(\beta) = \Psi_2(L)}} \# 
  \M_\beta(J)\in\Z,
\end{equation}
where $\# \M_\beta(J)$ is the number of
$J$-holomorphic disks in class $\beta$ of Maslov index~2,
passing through a fixed point on $L$. Observe that we are only counting disks of lowest area $\Psi_2(L)$.

We start by analysing the space of Lagrangian fibres of the standard toric
fibration on $\CP^2$ up to symplectomorphism. In the above terminology, this is
an $(1,1,1)$-ATF. Let us first take the quotient of the space of toric fibres by
the group $S_3$ of symplectomorphisms permuting the homogeneous coordinates on
$\CP^2$. This leaves us with a ``one-sixth'' slice of the initial moment
triangle. That slice is a closed triangle with one edge removed, see Figure
\ref{fig:ToricFibres}. We will now show that toric fibres corresponding to
different points in this slice are not related by symplectomorphisms of
$\Symp(\CP^2)$.

\begin{figure}[h!]   
  
\begin{center}

\includegraphics{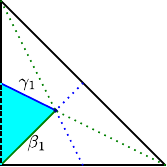}

\caption{The moment polytope of $\CP^2$. Fibres over the shaded region 
represent fibres of the toric fibration modulo action of $\Symp(\CP^2)$.}
\label{fig:ToricFibres} 

\end{center} 
\end{figure}

\begin{prp} \label{prp:ToricFibres}
  Let $T_{(x_1,y_1)}$ and $T_{(x_2,y_2)}$ be toric fibres over distinct points $(x_1,y_1) \ne 
  (x_2,y_2)$ belonging to  the shaded region of Figure \ref{fig:ToricFibres}. Then
  there is no symplectomorphism of $\CP^2$ taking $T_{(x_1,y_1)}$ to $T_{(x_2,y_2)}$.
\end{prp}

\begin{proof}
  We normalise the symplectic form so that the area of the complex line equals~$1$. Assume that there
  is a symplectomorphism $\phi \in \Symp(\CP^2)$ such that $\phi(T_{(x_1,y_1)}) =
  T_{(x_2,y_2)}$. We aim to show that $(x_1,y_1) = 
  (x_2,y_2)$.  The condition that $(x_i,y_i)$ belongs to the shaded region means:

\begin{equation}
\label{eq:ToricCond}
    0   <  x_i  \le  1/3,\quad
  x_i \le y_i  \le  \frac{1 - x_i}{2}.
\end{equation}

Since there is only one monotone fibre $T_{(1/3,1/3)}$, we may assume $x_i < 1/3$ for 
$i=1,2$.  
The torus $T_{(x_i,y_i)}$ bounds Maslov index 2 holomorphic disks 
in three relative classes $\alpha_i$, $\beta_i$, $H - \alpha_i -\beta_i$ and areas
$x_i$, $y_i$ and $1 - x_i - y_i$, respectively. By the 4-dimensional modifications of our results mentioned in Section~\ref{subsec:sh_bound_dimf}, the counts of minimal area holomorphic disks are invariant under symplectomorphisms. 

If $x_1 = y_1$, there are two minimal area holomorphic
disks in classes $\alpha_1$, $\beta_1$, so $\phi$ must send them to classes 
$\alpha_2$, $\beta_2$. In particular, $x_1 = y_1 = x_2 = y_2$.  

Now consider the case $x_i < y_i$. Now $\alpha_i$ is the unique class supporting
the minimal area holomorphic disk. So we must have $\phi_*(\alpha_1) = \alpha_2$
and $x_1 = x_2$. Assume, without loss of generality, that $y_1 \ge y_2$. We
have: 

\begin{align}
    0  \  < \ & x_1 \ = \ x_2 \ < \ 1/3;  \nonumber \\
  x_1 = x_2 \ < \ & y_2 \ \le \ y_1 \  \le \  \frac{1 - x_1}{2} = \frac{1 - x_2}{2}. 
  \label{eq:Txy1}
\end{align}

Since $\phi_*(\beta_1)$ has Maslov index 2, and $\{\alpha_2, \beta_2, H\}$ generate $H_2(\CP^2,T_{(x_2,y_2)};\Z)$, one may write
\[ \phi_*(\beta_1) = \beta_2 + k(H - 3\alpha_2) + l(\alpha_2 - \beta_2)\]
for some $k,l \in \Z$. 
Since $\del \alpha_1 \cdot \del \beta_1 = 1$ and $\phi$ preserves the intersection form on the tori up to sign,
\[ \del \phi_*(\beta_1)\cdot \del \phi_*(\alpha_1) = \pm 1
 = \del(\beta_2 + k(H - 3\alpha_2) + l(\alpha_2 - \beta_2)) \cdot \del\alpha_2 = 
1 - l. \] 
Let us first analyse the case $\del \phi_*(\beta_1)\cdot \del\phi_*(\alpha_1) = - 1$, which means $l = 2$.
Then  $\phi_*(\beta_1) = \beta_2 + k(H - 3\alpha_2) + 2(\alpha_2 - 
\beta_2)$, and by computing symplectic area:
\begin{equation}
  y_1 = 2x_2 -y_2 + k(1 - 3x_2). \label{eq:Txy2}
\end{equation}  
We get from \eqref{eq:Txy1}, \eqref{eq:Txy2} that
\[ 2x_2 < y_1 + y_2 = 2x_2 +  k(1 - 3x_2) \le 1 - x_2.\]
Since $1 - 3x_2 > 0$,  we 
obtain $k \ge 1$ from the first inequality, and  $k \le 1$ from the second inequality.
So $k = 1$ and $y_1 = y_2 = (1 - x_2)/2$. 

It remains to analyse the case $\del\phi_*(\beta_1)\cdot \del\phi_*(\alpha_1) = 1$, which means $l = 0$.
In this case, $\phi_*(\beta_1) = \beta_2 + k(H - 3\alpha_2)$, and by computing  
symplectic area:
\begin{equation}
  y_1 = y_2 + k(1 - 3x_2). \label{eq:Txy3}
\end{equation}
Since $1 - 3x_2 > 0$, $y_1 \ge y_2$, we have $k \ge 0$. 
From  \eqref{eq:Txy1}, \eqref{eq:Txy3}, we get
\[ x_2 + k(1 - 3x_2) < y_2 + k(1 - 3x_2) = y_1 \le \frac{1 - x_2}{2}.\]
We conclude that $k < 1/2$. Hence $k = 0$ and $y_1 = y_2$.
\end{proof}

Below is a useful corollary of the previous proof.

\begin{cor} \label{cor:ToricPrp}
  Let $T \subset \CP^2$ be a Lagrangian torus for which there is a unique class
  $\alpha \in H_2(\CP^2,T)$ realising $\Psi_2(T)$, i.e.~satisfying $\mu(\alpha) = 2$,
  $\omega(\alpha) = \Psi_2(T)$, such that that the count of holomorphic disks in class
  $\alpha$ through a fixed point of $L$ is non-zero. 
  
  Assume that there is another class     
  $\beta \in H_2(\CP^2,T)$ with  $\del \beta \cdot \del \alpha = \pm 1$ and
  such that the areas $x = \omega(\alpha)$, $y = \omega(\beta)$ 
  satisfy inequalities \eqref{eq:ToricCond}. Then the only 
  possible toric fibre in the shaded region of Figure \ref{fig:ToricFibres} that could be symplectomorphic to $T$ is $T_{(x,y)}$.   
    \qed
\end{cor}

The following notation will help us describe the space $\H$. Consider the
$(a,b,c)$-ATF, and mark the vertices of the moment triangle by the corresponding
Markov numbers. Consider three line segments connecting a vertex to the opposite
edge via the barycentre. The barycentre divides each segment (say, corresponding
to the vertex $a$) into two pieces; for $a \ne 1$, we call them $\beta_a$ and $\gamma_a$ where
$\beta_a$ is the segment containing $a$, see Figure \ref{fig:ATFsabc}. 
We name $\beta_1$ and $\gamma_1$ only the segments showing in 
Figure \ref{fig:ToricFibres}. We call the
fibres over these segments \emph{fibres of type $\beta_a$, $\gamma_a$, etc.}.

%The six segments divide the moment triangle into six sub-triangles. For each of
%those six sub-trangles we denote by $$\Delta_{\beta_a,\gamma_b}\textit{ etc.}$$
%the union of the \emph{interior} of that sub-triangle with the two sides
%e.g.~$\beta_a$, $\gamma_b$.

\begin{figure}[h!]   
  
\begin{center}

\includegraphics{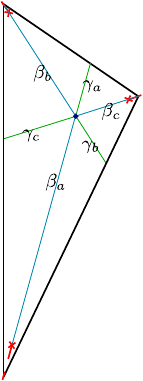}

\caption{Special fibres of type $\beta$ and $\gamma$ type in an $(a,b,c)$-ATF.}
\label{fig:ATFsabc} 

\end{center} 
\end{figure}

%\begin{dfn} \label{dfn:uplus}
 %Given two sets $A,B$, define $A \uplus B = A \cup B \setminus A \cap B$. 
%\end{dfn}

\begin{figure}[h!]   
  
\begin{center}

\includegraphics[scale=0.7]{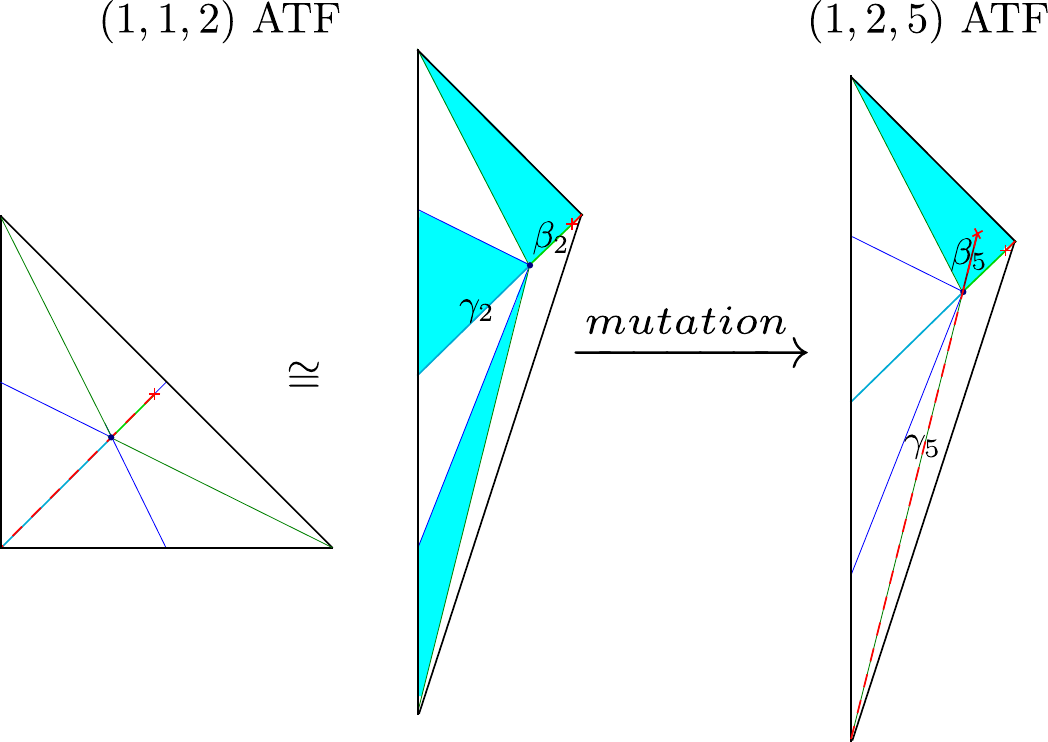}

\caption{Left and middle: $(1,1,2)$-ATF. Right: $(1,2,5)$-ATF.}
\label{fig:ATFs} 

\end{center} 
\end{figure}

 \begin{prp} \label{prp:betaCount}
  If $L$ is a $\beta_a$-type fibre of the $(a,b,c)$-ATF, then $\Xi_2(L) = 2^a$. 
  If $L$ is a $\gamma_a$-type fibre, then $\Xi_2(L) = 1$.
\end{prp}

\begin{proof} By the proof of Theorem~\ref{th:nbhood_comp}, there is an almost complex structure for which $L$ has the same enumerative geometry as the monotone torus  $T(a^2,b^2,c^2)$. The potential of the latter  is computed by the wall-crossing formula \cite{PT17} and its Newton polytope is dual to the moment triangle of the ATF \cite{Vi14}.

There is a unique term in the potential that becomes a lowest area holomorphic disk on a $\gamma_a$-type fibre; it corresponds to the vertex of the Newton polytope, and the count of disks in that class equals~1.
Next, those terms of the potential that become lowest area disks on a $\beta_a$-type fibre correspond to an edge of the Newton polytope. That edge has  $a + 1$ lattice points
and the coefficients in front of the corresponding monomials are the binomial coefficients; their sum is $2^a$.
\end{proof}

\begin{prp} \label{prp:gammaToric}
 Any $\gamma_a$-type fibre is equivalent to a fibre of the $(1,1,1)$-ATF or a $\gamma_2$-type fibre of the $(1,1,2)$-ATF.
\end{prp}

\begin{proof}
 If $L$ is a $\gamma_a$-type fibre of the $(a,b,c)$-ATF, then it is also a fibre
 of one of the two mutated ATFs corresponding to $(a,3ac - b,c)$ or $(a,b, 3ab
 -c)$, see \eqref{eq:Markov_mutation} and \cite[Proposition~2.4]{Vi14}. Now note
 that $L$ can be also seen as a fibre of an $(3bc - a,b,c)$-type ATF if
 (breaking the standing convention) we allow to slide the node associated with $\gamma_a$
 pass the monotone fibre, but not pass $L$. 
 This way, we are able to mutate the
 $(a,b,c)$-type ATF all the way to an $(1,1,1)$-type ATF
 \cite[Section~3.7]{KaNo98} \cite[Proposition~4.9]{Vi16a} maintaining $L$ as a
 fibre. The result follows, noting that fibres of an $(1,1,1)$-ATF in the
 relaxed sense (where we allow the nodes to slide arbitrarily close to the 
 monotone $T(1,1,1)$ fibre) are precisely the
 fibres of the $(1,1,1)$-ATF in our standing agreement, plus the
 $\gamma_2$-fibres. 
 
 \end{proof} 
 
Denote by $\H(a,b,c)$ the space of all fibres of the $(a,b,c)$-ATF modulo
symplectomorphisms of $\CP^2$. Let $\Delta$ be a closed triangle minus an edge,
which is affinely isomorphic to one of the six triangles in Figure
\ref{fig:ToricFibres}, which describes $\H(1,1,1)$. We shall see how copies of
$\Delta$ are embedded in a $(a,b,c)$-ATF in order to describe $\H(a,b,c)$.
In order to keep track of special fibres in an embedding of $\Delta$,
more specifically $\gamma_a$'s and $\beta_a$'s type fibres, we develop
the following notation. Let $x$, $y$, $w_1$, $w_2$, $w_3$ some be half-open segments in $\Delta$ starting at the vertex and ending at 
the missing edge. Among them, $x$ and $y$ must be the edges of $\Delta$, but $w_1,w_2,w_3$ can be arbitrary and go through the interior.
We then denote by $\Delta_{x,y}^{w_1,w_2,w_3}$,
the triangle $\Delta$ labeled by $x$, $y$, $w_1$, $w_2$, $w_3$. 
The set of $w_i$ is also allowed to be empty, in which case  they do not 
appear in the notation.
We can write 
$$\H(1,1,1) \cong \Delta_{\beta_1,\gamma_1},$$
since we can embed $\Delta$ in the toric $(1,1,1)$-ATF as one of the six triangles 
in Figure \ref{fig:ToricFibres}
with edges corresponding to the $\beta_1$ and $\gamma_1$ type fibres.   
    
Let us see how $\H(a,b,c)$ changes as we mutate the Markov triple, beginning
with an analysis of how $\H(1,1,2)$ differs from $\H(1,1,1)$. Their base
diagrams differ by performing nodal trade to a vertex of the standard
$(1,1,1)$-triangle, and sliding the node all the way to the opposite side. Since
nodal slide preserves the Hamiltonian isotopy class of fibres away from the cut,
the only (potentially) new tori lie over the line containing the cut, see
Figure~\ref{fig:ATFs}. These tori are of type $\gamma_2$ or $\beta_2$, according
to our notation.

By Proposition \ref{prp:betaCount} the $\beta_2$-type fibres have invariant
$\Xi_2 = 2^2 = 4$. So these fibres are not equivalent to any toric fibre. Next,
by Corollary \ref{cor:ToricPrp}, a $\gamma_2$-type fibre can only be equivalent
to a toric fibre if it is of $\beta_1$-type, but their $\Xi_2$-invariants equal
$1$ and $2$ respectively. Therefore, the $\gamma_2$-type fibres are also not
equivalent to a toric fibre.
                      
The three shaded triangles in the middle diagram of Figure \ref{fig:ATFs}
are embeddings of $\Delta$ accordingly labeled as $\Delta_{\beta_1,\gamma_1}$,
$\Delta_{\beta_1,\beta_2}$, $\Delta_{\gamma_2,\gamma_1}$.

Let's introduce further notation. Given a finite set of labeled 
triangles $\{\Delta_{x^i,y^i}^{w_1^i,w_2^i,w_3^i}; i = 1, \dots, k\}$, we
define                      
\[\Delta_{x^1,y^1}^{w_1^1,w_2^1,w_3^1}\vee \cdots \vee \Delta_{x^k,y^k}^{w_1^k,w_2^k,w_3^k}
= \amalg_{i=1}^k \Delta_{x^i,y^i}^{w_1^i,w_2^i,w_3^i} / \sim \] 
where we have for $p \in \Delta_{x^i,y^i}^{w_1^i,w_2^i,w_3^i}$ and
$q \in \Delta_{x^j,y^j}^{w_1^j,w_2^j,w_3^j}$, $p \sim q$ if:

\begin{enumerate}[label= (\roman*)]
  \item $p$ and $q$ are the vertices of $\Delta_{x^i,y^i}^{w_1^i,w_2^i,w_3^i}$,
  respectively, $\Delta_{x^j,y^j}^{w_1^j,w_2^j,w_3^j}$; or 
  \item 
  
  $p \notin x^i \cup y^i \cup w_1^i \cup w_2^i \cup w_3^i$,  $q \notin x^j \cup y^j \cup w_1^j \cup w_2^j \cup w_3^j$, and $p,q$ correspond to the same underlying point of $\Delta$; or
  
  \item $p \in x^i \cup y^i \cup w_1^i \cup w_2^i \cup w_3^i$,
  $q \in x^j \cup y^j \cup w_1^j \cup w_2^j \cup w_3^j$, $p,q$ correspond to the same underlying point of $\Delta$, and their segments labels match for $i$ and $j$. 
   
\end{enumerate}

Using this notation, we get that 
\[\H(1,1,2) \cong \Delta_{\beta_1,\gamma_1}
\vee \Delta_{\beta_1,\beta_2} \vee \Delta_{\gamma_2,\gamma_1}.\] 
The vertex corresponds to the monotone Chekanov torus $T(1,1,2^2)$.  
Now, if we forget about the embeddings of $\Delta$ in the 
$(1,1,2)$-ATF, we can describe the $\H(1,1,2)$ space using only 
two labeled triangles as   
\[\H(1,1,2) \cong \Delta_{\beta_1,\gamma_1}
\vee \Delta_{\gamma_2,\beta_2}.\]
We represent $\H(1,1,2)$ in the second diagram in Figure \ref{fig:Moduli},
by superposing these two triangles and keeping track of the labels.

\begin{figure}[h]   
  
\begin{center}

\includegraphics{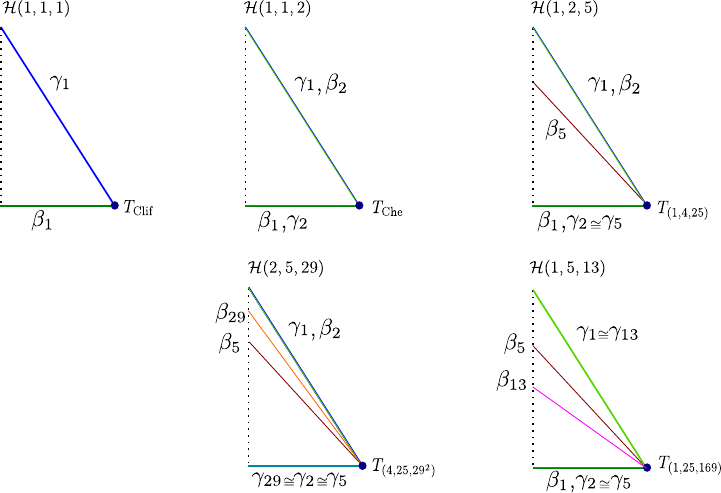}

\caption{Spaces $\H(a,b,c)$ of Lagrangian almost toric fibres.}
\label{fig:Moduli} 

\end{center} 
\end{figure}

Now, consider the diagram in Figure \ref{fig:ATFs} representing the
$(1,2,5)$-ATF. We obtain new fibres of type $\beta_5$ having invariant $\Xi_2 = 2^5$ 
by Proposition \ref{prp:betaCount}. So they are not
equivalent to any of the previously considered fibres. On the other hand, following the proof
of Proposition \ref{prp:gammaToric}, we see that the $\gamma_5$-type fibres are
equivalent to $\gamma_2$-type fibres. So the new triangles that arise in the
$(1,2,5)$-ATF can be labeled as $\Delta_{\beta_1,\beta_2}^{\beta_5}$ and
$\Delta_{\gamma_5,\gamma_1}$, where the fibres of the latter can de identified 
with the fibres of $\Delta_{\gamma_2,\gamma_1}$. 
We arrive at the description:

\[\H(1,2,5) \cong 
\Delta_{\beta_1,\beta_2}^{\beta_5} \vee \Delta_{\gamma_2,\gamma_1},\]
see the third diagram in Figure \ref{fig:Moduli}.

In general, the spaces $\H(a,b,c)$ follow the same pattern.

\begin{thm} \label{th:Moduli}
 Besides the spaces $\H(1,1,1)$, $\H(1,1,2)$, $\H(1,2,5)$ described above, the 
 spaces $\H(a,b,c)$ are given by:

\begin{enumerate}[label= (\roman*)] 
  \item  $\H(1,b,c) \cong \Delta_{\beta_1,\gamma_1}^{\beta_b,\beta_c}
\vee 
\Delta_{\gamma_2,\gamma_1}$;
  \item $\H(2,b,c) \cong  \Delta_{\gamma_2,\beta_2}^{\beta_b,\beta_c}
\vee 
\Delta_{\gamma_2,\gamma_1}$;
\item $\H(a,b,c) \cong \Delta_{\gamma_2,\gamma_1}^{\beta_a,\beta_b,\beta_c} \vee
\Delta_{\gamma_2,\gamma_1}$ for $a,b,c > 2$.
\end{enumerate} 
Moreover, assuming $a \le b < c \ne 2$ and after applying congruences,  the segment corresponding to 
$\beta_c$ lies between the segments corresponding to $\beta_a$ and $\beta_b$.
\end{thm}

\begin{figure}[h!]   
  
\begin{center}

\includegraphics[scale=0.7]{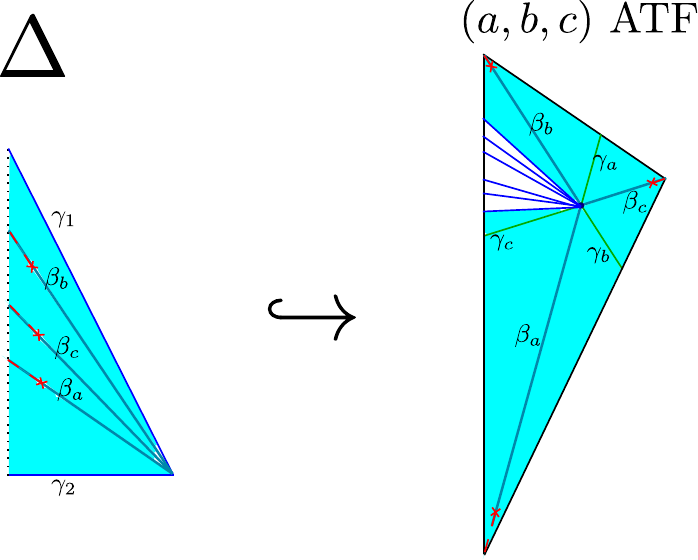}

\caption{The triangle $\Delta_{\gamma_2,\gamma_1}^{\beta_a,\beta_b,\beta_c}$ embedded into an $(a,b,c)$-ATF.}
\label{fig:ATFsabcDelta} 

\end{center} 
\end{figure}

\begin{proof}
  Using Propositions \ref{prp:betaCount}, \ref{prp:gammaToric},
  we are able to give a complete description of what happens to the  space of tori as one mutates from the $(a,b,c)$-ATF
  to the $(a,b,3ab -c)$-ATF. Namely, this mutation introduces a new family of $\beta_{c'}$-type fibres, while the $\gamma_{c'}$-type fibres are not new: they are either toric or 
  of $\gamma_2$-type. Furthermore, the $\beta_c$-type fibres disappear,
unless $c=a$ or $c=b$ as in the mutations 
  $\H(1,1,1) \to \H(1,1,2) \to \H(1,2,5)$. 
  In the case $a,b,c > 2$, this means we are just replacing 
  $\Delta_{\gamma_2,\gamma_1}^{\beta_a,\beta_b,\beta_c}$, by 
  $\Delta_{\gamma_2,\gamma_1}^{\beta_a,\beta_b,\beta_c'}$ as we change from 
  $\H(a,b,c)$ to $\H(a,b,c')$, see Figure~\ref{fig:ATFsabcDelta}.
  The remaining cases are equally easy to consider: just note that, 
  if our mutation is increasing the Markov sum $a + b + c$ \cite[Section~3.7]{KaNo98}\cite[Proposition~4.9]{Vi16a} then, after 
  $(1,2,5)$, all new $\beta$ type fibres will lie inside the consecutive modifications
  of $\Delta_{\beta_1,\beta_2}^{\beta_5}$, see the third diagram of
  Figure \ref{fig:ATFs}.  
  
  The last part of the claim can be proved inductively,
  we still work in the setting of increasing the Markov sum. 
  In Figure \ref{fig:ATFsabcDelta}, this means that $c> a,b$ and to 
  increase the Markov sum we mutate either $a$ or $b$, creating
  $a'> c,b$, resp. $b'> c,a$, with $\beta_{a'}$ appearing between 
  $\beta_b, \beta_c$, resp. $\beta_{b'}$ appearing between 
  $\beta_a, \beta_c$.  
\end{proof}

Joining all the spaces $\H(a,b,c)$ together, one obtains the following
description of the full space $\H$. We develop
another notation $\dot{\vee}$ for describing this space.

Before describing $\dot{\vee}$, we point out that we interpret a wedge of
labeled triangles, as a ``labelled non-Hausdorff triangle'', that projects to
$\Delta$. For instance, $\H(1,1,2) = \Delta_{\beta_1,\gamma_1} \vee
\Delta_{\gamma_2,\beta_2}$, has over the bottom edge of $\Delta$, two
non-Hausdorff open segments, one receiving the label $\beta_1$ and the other
$\gamma_2$, see Figure \ref{fig:Moduli}. Also we point out that, most 
of the points are thought to have empty label. For instance, 
in $\H(1,2,5) =
\Delta_{\beta_1,\beta_2}^{\beta_5} \vee \Delta_{\gamma_2,\gamma_1}$,
we have two non-Hausdorff open segments projecting over the 
segment in $\Delta$ corresponding to $\beta_5$. One of these
segments receive the $\beta_5$ label (corresponding to  
$\beta_5$-type tori) and the other receive the 
empty label (corresponding to tori in the toric diagram), see again Figure \ref{fig:Moduli}.

With the above in mind, we set $\dot{\vee}$ between wedges of labelled triangles
in the same way as $\vee$ with the exception that the vertices are not
identified -- also replace ``same underlying point of $\Delta$'', by 
``projecting to the same point in $\Delta$'' in items (ii) and (iii).

For instance, if we consider 
\[ \H(1,1,1) \dot{\vee} \H(1,1,2) =  \Delta_{\beta_1,\gamma_1} \dot{\vee}  (\Delta_{\beta_1,\gamma_1}
\vee \Delta_{\gamma_2,\beta_2}), \]
we have two non-Hausdorff vertices (corresponding to the Clifford and Chekanov monotone tori),
and two pair of non-Hausdorff edges $(\beta_1, \gamma_2)$ and $(\beta_2, 
\gamma_1)$. All remaining points are separable. Note also that, 
if we approach the vertex, say 
over the $\beta_1$ edge, we have two possible limits. So the space is not
the same as two copies of $\Delta$ identified over the interior. 

Before stating the Theorem, we 
recall the Markov conjecture, that says that a
Markov triple is uniquely determined by its biggest Markov number.

\begin{thm} \label{th:globalModuli} 
 Assume that the Markov conjecture holds
(otherwise see Remark \ref{rmk:MarkovNumbers}). The space $\H$ of all almost
toric fibres of $\CP^2$ modulo symplectomorphisms is the following non-Hausdorff
space: 
\[ \H = \dot{\bigvee}_{(a,b,c) \, \mathrm{Markov} \ \mathrm{triple}} \H(a,b,c) \] 

Note that the vertex of each $\H(a,b,c)$-space is not
identified with the analogous vertex for any other Markov triple, as these
vertices correspond to the different monotone tori $T(a^2,b^2,c^2)$.
\end{thm}

\begin{rmk} \label{rmk:MarkovNumbers} Each time a Markov number $c$ appears in the
Markov  tree, it generates an infinite ray obtained by
mutating the other two Markov numbers indefinitely. The Markov conjecture implies that there is a unique such ray involving the number $c$. 
If the Markov conjecture does not hold, in
Theorem \ref{th:globalModuli}, the $\beta_c$ corresponding to different rays associated to $c$, 
should not be identified, in other words, the label should be indexed by the corresponding ray, instead of the
Markov number. \end{rmk}

\bibliography{Symp_bib,SympRefs}

\begin{thebibliography}{10}

\bibitem{Ab14}
M.~Abouzaid.
\newblock {Family Floer cohomology and mirror symmetry}.
\newblock {\em Proceedings of the International Congress of Mathematicians},
  2014.

\bibitem{Au07}
D.~Auroux.
\newblock {Mirror symmetry and T-duality in the complement of an anticanonical
  divisor}.
\newblock {\em J. G{\"o}kova Geom. Topol.}, 1:51--91, 2007.

\bibitem{Au09}
D.~Auroux.
\newblock {Special Lagrangian fibrations, wall-crossing, and mirror symmetry}.
\newblock In {\em {Geometry, analysis, and algebraic geometry}}, volume~13 of
  {\em {Surveys in Differential Geometry}}, pages 1--47. Intl. Press, 2009.

\bibitem{CompSFT03}
F.~Bourgeois, Y.~Eliashberg, H.~Hofer, K.~Wysocki, and E.~Zehnder.
\newblock {Compactness results in Symplectic Field Theory}.
\newblock {\em Geom. Topol.}, 7:799--888, 2003.

\bibitem{CW17}
F.~Charest and C.~Woodward.
\newblock Floer trajectories and stabilizing divisors.
\newblock {\em J. Fixed Point Theory Appl.}, 19(2):1165--1236, 2017.

\bibitem{CW15}
Fran\c{c}ois Charest and Chris~T. Woodward.
\newblock Floer cohomology and flips.
\newblock {\em Mem. Amer. Math. Soc.}, 279(1372), 2022.

\bibitem{ChSch10}
Y.~Chekanov and F.~Schlenk.
\newblock Notes on monotone {L}agrangian twist tori.
\newblock {\em Electron. Res. Announc. Math. Sci.}, 17:104--121, 2010.

\bibitem{Cho12}
C.-H. Cho.
\newblock {On the obstructed Lagrangian Floer theory}.
\newblock {\em Adv. Math.}, 229:804--853, 2012.

\bibitem{CKO17}
Yunhyung Cho, Yoosik Kim, and Yong-Geun Oh.
\newblock Lagrangian fibers of {G}elfand-{C}etlin systems.
\newblock {\em Adv. Math.}, 372:107304, 57, 2020.

\bibitem{CM07}
K.~Cieliebak and K.~Mohnke.
\newblock {Symplectic hypersurfaces and transversality in Gromov-Witten
  theory}.
\newblock {\em J.~Symplectic Geom.}, 5(3):281--356, 2007.

\bibitem{CM14}
K.~Cieliebak and K.~Mohnke.
\newblock Punctured holomorphic curves and {L}agrangian embeddings.
\newblock {\em Invent. Math.}, 212(1):213--295, 2018.

\bibitem{CCGGK14}
T.~Coates, A.~Corti, S.~Galkin, V.~Golyshev, and A.~M. Kasprczyk.
\newblock {Mirror symmetry and Fano manifolds}.
\newblock In {\em {European Congress of Mathematics Krakow, 2--7 July, 2012}},
  pages 285--300, 2014.

\bibitem{CCGK16}
T.~Coates, A.~Corti, S.~Galkin, and A.~M. Kasprczyk.
\newblock {Quantum periods for $3$-dimensional Fano manifolds}.
\newblock {\em Geom. Topol.}, 20(1):103--256, 2016.

\bibitem{CKP17}
T.~Coates, A.~M. Kasprzyk, and T.~Prince.
\newblock Laurent inversion.
\newblock {\em Pure Appl. Math. Q.}, 15(4):1135--1179, 2019.

\bibitem{CoLa05}
O.~Cornea and F.~Lalonde.
\newblock {Cluster homology}.
\newblock {\em arXiv:math/0508345}, 2005.

\bibitem{CL06}
O.~Cornea and F.~Lalonde.
\newblock {Cluster homology: an overview of the construction and results}.
\newblock {\em Electron. Res. Announc. Amer. Math. Soc.}, 12:1--12, 2006.

\bibitem{DGMS75}
P.~Deligne, P.~Griffiths, J.~Morgan, and D.~Sullivan.
\newblock {Real homotopy theory of K\"ahler manifolds}.
\newblock {\em Invent. Math.}, 29:245--274, 1975.

\bibitem{Eli91}
Y.~Eliashberg.
\newblock New invariants of open symplectic and contact manifolds.
\newblock {\em J. Amer. Math. Soc.}, 4(3):513--520, 1991.

\bibitem{EGH00}
Y.~Eliashberg, A.~Givental, and H.~Hofer.
\newblock {Introduction to Symplectic Field Theory}.
\newblock In {\em {Visions in Mathematics: GAFA 2000 Special volume, Part II}},
  pages 560--673, 2000.

\bibitem{EGM16}
Michael Entov, Yaniv Ganor, and Cedric Membrez.
\newblock Lagrangian isotopies and symplectic function theory.
\newblock {\em Comment. Math. Helv.}, 93(4):829--882, 2018.

\bibitem{FeKa02}
R.~Feres and A.~Katok.
\newblock Ergodic theory and dynamics of {$G$}-spaces (with special emphasis on
  rigidity phenomena).
\newblock In {\em Handbook of dynamical systems, {V}ol.\ 1{A}}, pages 665--763.
  North-Holland, Amsterdam, 2002.

\bibitem{Fuk10}
K.~Fukaya.
\newblock {Cyclic symmetry and adic convergence in Lagrangian Floer theory}.
\newblock {\em Kyoto J. Math.}, 50(3):521--590, 2010.

\bibitem{FO3Book}
K.~Fukaya, Y.-G. Oh, H.~Ohta, and K.~Ono.
\newblock {\em {Lagrangian Intersection Floer Theory: Anomaly and
  Obstruction}}, volume~46 of {\em {Stud. Adv. Math.}}
\newblock American Mathematical Society, International Press, 2010.

\bibitem{HaKa15}
M.~Harada and K.~Kaveh.
\newblock {Integrable systems, toric degenerations and Okounkov bodies}.
\newblock {\em Invent. Math.}, 202(3):927--985, 2015.

\bibitem{KaNo98}
B.~V. Karpov and D.~Y. Nogin.
\newblock Three-block exceptional sets on del {P}ezzo surfaces.
\newblock {\em Izv. Ross. Akad. Nauk Ser. Mat.}, 62(3):3--38, 1998.

\bibitem{SyLe10}
N.~C. Leung and M.~Symington.
\newblock Almost toric symplectic four-manifolds.
\newblock {\em J. Symplectic Geom.}, 8(2):143--187, 2010.

\bibitem{Mar06}
M.~Markl.
\newblock {Transferring $A_\infty$ (strongly homotopy associative) structures}.
\newblock {\em Rend. Circ. Mat. Palermo (2)}, 79:139--151, 2006.

\bibitem{NNU10}
T.~Nishinou, Y.~Nohara, and K.~Ueda.
\newblock Toric degenerations of {G}elfand-{C}etlin systems and potential
  functions.
\newblock {\em Adv. Math.}, 224(2):648--706, 2010.

\bibitem{PT17}
James Pascaleff and Dmitry Tonkonog.
\newblock The wall-crossing formula and {L}agrangian mutations.
\newblock {\em Adv. Math.}, 361:106850, 67, 2020.

\bibitem{Sei08}
P.~Seidel.
\newblock {A biased view of symplectic cohomology}.
\newblock {\em Current Developments in Mathematics}, 2006:211--253, 2008.

\bibitem{SeiBook08}
P.~{S}eidel.
\newblock {\em {Fukaya Categories and Picard-Lefschetz Theory}}.
\newblock European Mathematical Society, Zurich, 2008.

\bibitem{Si89}
J.-C. Sikorav.
\newblock {Rigidit\'e symplectique dans le cotangent de $T^n$}.
\newblock {\em Duke Math. J.}, 59(3):759--763, 1989.

\bibitem{SmG15}
G.~Smirnov.
\newblock {On the monodromy of almost toric fibrations on the complex
  projective plane}.
\newblock {\em arXiv:1503.04458}, 2015.

\bibitem{Sym03}
M.~Symington.
\newblock {Four dimensions from two in symplectic topology}.
\newblock In {\em {Proceedings of the 2001 Georgia International Topology
  Conference}}, pages 153--208, 2003.

\bibitem{Vi13}
R.~Vianna.
\newblock {On exotic Lagrangian tori in $\mathbb{C}P^2$}.
\newblock {\em Geom. Topol.}, 18:2419--2476, 2014.

\bibitem{Vi14}
R.~Vianna.
\newblock {Infinitely many exotic monotone Lagrangian tori in $\mathbb{C}P^2$}.
\newblock {\em J. Topol.}, 9(2):535--551, 2016.

\bibitem{Vi16a}
R.~Vianna.
\newblock Infinitely many monotone {L}agrangian tori in del {P}ezzo surfaces.
\newblock {\em Selecta Math. (N.S.)}, 23(3):1955--1996, 2017.

\bibitem{Vi16b}
Renato Vianna.
\newblock Continuum families of non-displaceable {L}agrangian tori in {$(\Bbb
  CP^1)^{2m}$}.
\newblock {\em J. Symplectic Geom.}, 16(3):857--883, 2018.

\bibitem{Wu15}
W.~Wu.
\newblock {On an exotic Lagrangian torus in $\mathbb{C} P^{2}$}.
\newblock {\em Compositio Math.}, 151(7):1372--1394, 2015.

\bibitem{Zu03}
N.~T. Zung.
\newblock Symplectic topology of integrable {H}amiltonian systems. {II}.
  {T}opological classification.
\newblock {\em Compositio Math.}, 138(2):125--156, 2003.

\end{thebibliography}
\bibliographystyle{plain}

\end{document}